\documentclass[3p]{elsarticle}

\makeatletter
\def\ps@pprintTitle{%
    \let\@oddhead\@empty
    \let\@evenhead\@empty
    \let\@evenfoot\@oddfoot 
    }
\makeatother
\usepackage{tabularx}
\usepackage{amsmath}
\usepackage{amssymb}
\usepackage{amsthm}
\usepackage{bbm}
\usepackage{mathrsfs}
\usepackage{bm}
\usepackage{graphicx}
\usepackage[final]{hyperref}
\usepackage[ruled,vlined]{algorithm2e}
\usepackage{enumitem}
\usepackage{array}
\usepackage{multirow}
\usepackage{makecell}
\usepackage[title]{appendix}
\usepackage{empheq}
\usepackage{cleveref}
\usepackage{cases}
\usepackage{mathrsfs}
\usepackage{textcomp}
\usepackage{gensymb}
\usepackage{subcaption}
\hypersetup{
	colorlinks=true,       
	linkcolor=blue,        
	citecolor=blue,        
	filecolor=magenta,     
	urlcolor=blue         
}
\usepackage{float}
\usepackage{multirow}
\usepackage{multicol}
\usepackage{booktabs}  

\usepackage[colorinlistoftodos,prependcaption,textsize=normalsize]{todonotes}
\usepackage{regexpatch}

\usepackage{tcolorbox} 
\usepackage{xcolor} 

\makeatletter
\xpatchcmd{\@todo}{\setkeys{todonotes}{#1}}{\setkeys{todonotes}{inline,#1}}{}{}
\makeatother

\usepackage{tikz}
\usepackage{tikz-3dplot}
\usetikzlibrary{calc,patterns,decorations.pathreplacing,calligraphy}
\tikzstyle{bag} = [align=center]
\usetikzlibrary {arrows.meta}
\usepackage{float}

\newtheorem{theorem}{Theorem}[section]
\newtheorem{lemma}{Lemma}[section]
\newtheorem{proposition}{Proposition}[section]

\newtheorem{definition}{Definition}[section]
\newtheorem{assumption}{Assumption}[section]
\newtheorem{corollary}{Corollary}[section]
\newtheorem{remark}{Remark}[section]

\newcommand{\HS}{\mathrm{HS}}
\newcommand{\const}{C}
\newcommand{\ftimes}{f_{\times}}

\newcommand{\fnorm}{f_{\ell_1}}
\newcommand{\fabs}{f_{\mathrm{abs}}}
\newcommand{\fcutoff}{\tilde{f}_{\mathrm{cut}}}
\newcommand{\funcut}{f_{\mathrm{cut}}}
\newcommand{\ones}{\mathbf{1}}
\newcommand{\GELU}{\sigma_\mathrm{GELU}}
\newcommand{\cutoff}{\mathrm{cut}}

\newcommand{\cB}{\mathcal{B}}
\newcommand{\cC}{\mathcal{C}}
\newcommand{\cD}{\mathcal{D}}
\newcommand{\cE}{\mathcal{E}}

\newcommand{\cG}{\mathcal{G}}
\newcommand{\cH}{\mathcal{H}}

\newcommand{\cJ}{\mathcal{J}}

\newcommand{\cL}{\mathcal{L}}
\newcommand{\cM}{\mathcal{M}}
\newcommand{\cN}{\mathcal{N}}

\newcommand{\cP}{\mathcal{P}}

\newcommand{\cU}{\mathcal{U}}
\newcommand{\cV}{\mathcal{V}}
\newcommand{\cW}{\mathcal{W}}
\newcommand{\cX}{\mathcal{X}}
\newcommand{\cY}{\mathcal{Y}}
\newcommand{\cZ}{\mathcal{Z}}

\DeclareMathOperator*{\esssup}{ess\,sup}

\newcommand{\Op}{\mathcal{L}}
\newcommand{\bE}{\mathbb{E}}
\newcommand{\bR}{\mathbb{R}}
\newcommand{\bN}{\mathbb{N}}
\newcommand{\tr}{\mathrm{Tr}}

\newcommand{\openset}{\mathcal{B}}

\newcommand{\lip}{\kappa}
\newcommand{\ub}{M}

\newcommand{\ubu}{\ub_{u}}
\newcommand{\ubDzu}{\ub_{u_z}}

\newcommand{\ubQ}{\ub_{Q}}
\newcommand{\ubDzQ}{\ub_{Q_z}}
\newcommand{\ubDuQ}{\ub_{Q_u}}
\newcommand{\ubDuzQ}{\ub_{Q_{uz}}}

\newcommand{\flux}{\tau}
\newcommand{\fun}{f}
\newcommand{\salg}{\cG}

\newcommand{\diffeo}{\chi}
\newcommand{\displace}{\mathfrak{d}}
\newcommand{\Omegaref}{\Omega_0}
\newcommand{\Omegaz}{\Omega_z}

\newcommand{\latent}{\mathcal{W}}

\newcommand{\cUz}{\cU(\Omegaz)}
\newcommand{\cMz}{\cM(\Omegaz)}
\newcommand{\cVz}{\cV(\Omegaz)}

\newcommand{\cUref}{\cU(\Omegaref)}
\newcommand{\cVref}{\cV(\Omegaref)}
\newcommand{\cMref}{\cM(\Omegaref)}

\newcommand{\utestz}{\widetilde{y}}
\newcommand{\utest}{y}

\newcommand{\cUuz}{\cU(\Omegaz^s)}
\newcommand{\cVuz}{\cV(\Omegaz^s)}
\newcommand{\cUuref}{\cU(\Omegaref)}
\newcommand{\cVuref}{\cV(\Omegaref)}
\newcommand{\cUvref}{\cU_v(\Omegaref)}
\newcommand{\cUpref}{\cU_p(\Omegaref)}

\newcommand{\cUvz}{\cU_v(\Omegaz^s)}
\newcommand{\cUpz}{\cU_p(\Omegaz^s)}

\newcommand{\uref}{u}
\newcommand{\mref}{m}
\newcommand{\vref}{v}
\newcommand{\pref}{p}

\newcommand{\uz}{\widetilde{u}}
\newcommand{\vz}{\widetilde{v}}
\newcommand{\pz}{\widetilde{p}}
\newcommand{\mz}{\widetilde{m}}
\newcommand{\wz}{\widetilde{w}}
\newcommand{\qz}{\widetilde{q}}

\newcommand{\vol}{\mathrm{Vol}}
\newcommand{\Rz}{\mathcal{R}_z}
\newcommand{\Rref}{\mathcal{R}}

\newcommand{\POD}{\mathrm{POD}}
\newcommand{\AS}{\mathrm{AS}}

\newcommand{\mean}{\mathrm{mean}}
\newcommand{\CVaR}{\mathrm{CVaR}}
\newcommand{\entropic}{\mathrm{entropic}}
\newcommand{\unif}{\mathrm{Unif}}

\newcommand{\dm}{d_m}
\newcommand{\dz}{d_z}
\newcommand{\du}{d_u}
\newcommand{\rru}{r_u}
\newcommand{\rrm}{r_m}
\newcommand{\target}{\mathrm{target}}

\newcommand{\stress}{\mathbf{T}}

\newcommand{\reflect}{\mathrm{Reflect}}
\newcommand{\channel}{\mathrm{channel}}
\newcommand{\object}{\mathrm{object}}
\newcommand{\control}{\mathrm{control}}

\newcommand{\flow}{\mathrm{flow}}
\newcommand{\nitsche}{\mathrm{Nitsche}}

\newcommand{\symm}{\mathrm{sym}}
\newcommand{\differ}{\mathsf{D}}
\newcommand{\anoperator}{\mathsf{A}}
\newcommand{\scalerad}{5}

\newcommand{\inflow}{\text{in}}
\newcommand{\Gammain}{\Gamma_{\inflow}}

\newcommand{\stressz}{\widetilde{\mathbf{T}}}
\newcommand{\strain}{\varepsilon}
\newcommand{\strainz}{\widetilde{\varepsilon}}
\newcommand{\bend}{\mathrm{bend}}
\newcommand{\force}{\mathrm{force}}
\newcommand{\grav}{\mathrm{grav}}

\newcommand{\eastward}{\mathrm{East}}
\newcommand{\northward}{\mathrm{North}}
\newcommand{\SAA}{\mathrm{SAA}}

\newcommand{\projrm}[1]{P_{\cM, #1}}
\newcommand{\projru}[1]{P_{\cU, #1}}
\newcommand{\condexp}[2]{\bE_{\projrm{#1}}\left[#2\right]}

\newcommand{\defgradz}{F_{\diffeo_z}}
\newcommand{\defgrad}{F_{\diffeo}}
\newcommand{\Omegau}{\cD}
\newcommand{\mur}{\mu_{r}}

\newcommand{\shiftu}{b_u}
\newcommand{\shiftm}{b_m}

\newcommand{\closure}[1]{\mathrm{cl}\left(#1\right)}

\crefname{assumption}{Assumption}{Assumptions}

\begin{document}
\begin{frontmatter}
\title{Shape Derivative-Informed Neural Operators\\
with Application to Risk-Averse Shape Optimization\\
}
\author[1]{Xindi Gong\corref{cor1}}
\ead{xindi.gong@utexas.edu}
\author[2,3]{Dingcheng Luo}
\ead{dingcheng.luo@qut.edu.au}
\author[4]{Thomas O'Leary-Roseberry}
\ead{oleary-roseberry.1@osu.edu}
\author[5]{Ruanui Nicholson}
\ead{ruanui.nicholson@auckland.ac.nz}
\author[1,6]{Omar Ghattas}
\ead{omar@oden.utexas.edu}
\address[1]{Oden Institute for Computational Engineering and Sciences, The University of Texas at Austin, Austin, Texas, USA}
\address[2]{School of Mathematical Sciences, Queensland University of Technology, Brisbane, Australia}
\address[3]{Centre for Data Science, Queensland University of Technology, Brisbane, Australia}
\address[4]{Department of Mathematics, The Ohio State University, Columbus, Ohio, USA}
\address[5]{Department of Engineering Science and Biomedical Engineering, University of Auckland, Auckland, New Zealand}
\address[6]{Walker Department of Mechanical Engineering, The University of Texas at Austin,  Austin, Texas, USA.}

\cortext[cor1]{Corresponding authors}

\begin{abstract}
Shape optimization under uncertainty (OUU) is computationally intensive for classical PDE-based methods due to the high cost of repeated sampling-based risk evaluation across many uncertainty realizations and varying geometries, while standard neural surrogates often fail to provide accurate and efficient sensitivities for optimization. We introduce Shape-DINO, a derivative-informed neural operator framework for learning PDE solution operators on families of varying geometries, with a particular focus on accelerating PDE-constrained shape OUU. Shape-DINOs encode geometric variability through diffeomorphic mappings to a fixed reference domain and employ a derivative-informed operator learning objective that jointly learns the PDE solution and its Fr\'echet derivatives with respect to design variables and uncertain parameters, enabling accurate state predictions and reliable gradients for large-scale OUU.

We establish a priori error bounds linking surrogate accuracy to optimization error and prove universal approximation results for multi-input reduced basis neural operators in suitable $C^1$ norms. We demonstrate efficiency and scalability on three representative shape OUU problems, including boundary design for a Poisson equation and shape design governed by steady-state Navier--Stokes exterior flows in two and three dimensions. Across these examples, Shape-DINOs produce more reliable optimization results than operator surrogates trained without derivative information. In our examples, Shape-DINOs achieve 3--8 orders-of-magnitude speedups in state and gradient evaluations. Counting training data generation, Shape-DINOs reduce necessary PDE solves by 1--2 orders-of-magnitude compared to a strictly PDE-based approach for a single OUU problem. Moreover, Shape-DINO construction costs can be amortized across many objectives and risk measures, enabling large-scale shape OUU for complex systems.

\end{abstract}


\end{frontmatter}

\tableofcontents

\section{Introduction}
Computational models of physical systems often must account for geometry that is either variable or uncertain. An example of the former is shape optimization, where many iterations over a changing geometry are required during optimization to maximize given performance objectives, while an example of the latter is uncertainty quantification, where geometric features such as bathymetry or as-built geometry are uncertain and many geometries must be sampled to propagate that uncertainty into quantities of interest. In many settings, the geometric variability is encoded using problem-specific parameterizations that should be accounted for simultaneously with additional sources of uncertainty, such as spatially varying material properties or forcing fields, further complicating the modeling and analysis of these systems.
We are particularly interested in systems governed by partial differential equations (PDEs). Throughout, we will denote the geometric (shape) parameters by $z$, uncertain parameters by $m$, and the PDE state by $u$, and write the implicitly defined solution map as
\begin{equation*}
    (m,z) \mapsto u(m,z)
    \quad \text{subject to} \quad
    \mathcal{R}(u,m,z) = 0,
\end{equation*}
where $\mathcal{R}$ denotes an abstract residual form of the governing PDE.

In this work, we develop efficient neural operator learning formulations for the parametric solution map
\[
(m,z) \mapsto u(m,z),
\]
with the goal of accelerating outer-loop tasks involving simultaneous variations in the model parameters and domains.
While the proposed operator learning framework is broadly applicable to a range such tasks, we restrict our attention here to problems of shape optimization under uncertainty (OUU), 
which arise naturally when geometric design decisions must be made in the presence of uncertain model parameters.

We treat the uncertainty in the parameters by modeling $m$ as a random variable distributed according to a reference probability measure $\mu_m$, which may be informed, for example, by prior beliefs or the solution of a Bayesian inverse problem. 
As a consequence, any quantities of interest that depend on the PDE solution, i.e., $Q = Q(u(m,z), m, z)$ ,
are random variables whose distribution reflects the combined effects of parameter uncertainty and geometric variability. 

Rather than optimizing a single realization of $Q$, one instead seeks to optimize a scalar \emph{risk measure} $\rho$ that encodes the designer’s tolerance to uncertainty \cite{shapiro2021lectures}. This leads to the shape OUU problem
\begin{equation}\label{eq:risk_averse_shape_opt}
    \min_z \mathcal{J}(z)
    = \rho_{m \sim \mu_m}\!\left(Q(u(m,z),m,z)\right) + \mathcal{P}(z),
\end{equation}
where $\mathcal{P}(z)$ is a regularization or cost term on the geometry, enforcing design constraints or prior geometric preferences. Common choices of $\rho$ include the expectation, mean--variance functionals, conditional value-at-risk (CVaR), worst-case risk, and entropic risk measures \cite{shapiro2021lectures,rockafellar2000optimization,kouri2016risk}. These allow practitioners to interpolate between risk-neutral designs and highly conservative designs that emphasize tail events or worst-case performance \cite{conti2011risk,conti_rumpf_schultz_tölkes_2018,geihe_lenz_rumpf_schultz_2012, schillings2011efficient}.

\paragraph{Computational challenges}
Shape OUU presents several interrelated computational challenges:
\begin{enumerate}
    \item \emph{Geometric variability.} Modeling changes in geometry typically requires re-meshing or mesh deformation, leading to substantial computational overhead for each variation of the design variable $z$.

    \item \emph{Sampling complexity.} The estimation of risk measures requires repeated evaluations of the forward model across many realizations of the uncertain parameters $m$, resulting in prohibitively large sampling costs.

    \item \emph{Fast online evaluation.} Efficient (online) query capabilities are essential to mitigate the high sample complexity inherent in optimization under uncertainty.

    \item \emph{Derivative accuracy.} Nonlinear optimization algorithms require accurate gradients and, for fast convergence, Hessian information. 
    In the context of risk-averse shape optimization, this necessitates the ability to differentiate the PDE solution with respect to the design variables $z$ (i.e., the shape sensitivities).  
\end{enumerate}

Because these difficulties are tightly coupled, shape OUU quickly becomes computationally intractable when the cost of a single high-fidelity PDE solve is large, as is typical in complex three-dimensional simulations. In such settings, the combined demands of sampling, geometric variation, and derivative evaluation overwhelm traditional optimization workflows.

This computational bottleneck has driven the development of advanced surrogate modeling techniques, such as neural operators, that explicitly account for geometric variability and thereby enable tasks such as shape optimization. In this context, neural operators \cite{kovachki2023neural} approximate the parametric solution map $(m,z)\mapsto u(m,z)$ by a nonlinear surrogate $u_\theta(m,z)$, parameterized by weights $\theta$, which encodes discretization-invariant structure in the inputs and outputs to provide efficient approximations.

However, existing neural operator approaches typically have not addressed the simultaneous parametrization of auxiliary spatial fields $m$ along with the shape parameters $z$, nor do they explicitly control errors in the surrogate approximations of the
Fr\'echet derivatives of PDE states $u$ with respect to the design and uncertain parameters, $D_z u$ and $D_m u$. As a consequence, their deployment in optimization settings lacks guarantees on the accuracy of the resulting gradients and, ultimately, on the quality of the computed optimal designs. This gap motivates the approach developed in this work.

\subsection{Contributions}

In this work, we develop \emph{shape derivative-informed} neural operators to address the aforementioned challenges in a unified manner. Our approach accommodates geometric variability by representing families of PDEs through diffeomorphic mappings to a common reference domain. This formulation enables the entire learning problem to be posed on a fixed reference mesh, while the trained operator remains capable of making predictions across a broad range of geometric transformations. 

The central contribution of this work is the extension of derivative-informed operator learning to this geometrically complex setting. We construct a unified surrogate that is jointly parameterized by the shape variables $z$ and auxiliary spatial fields $m$, and is explicitly trained to accurately approximate both the PDE solution $u$ and its Fr\'echet derivatives with respect to each set of parameters.

Our main contributions are as follows.
\begin{itemize}
    \item[(C1)] \emph{A priori error bounds for surrogate-driven optimization and inversion.}
    We analyze optimization problems of the form \eqref{eq:risk_averse_shape_opt} where the objective $\mathcal{J}(z)$ is approximated by a surrogate objective $\mathcal{J}_\theta(z) := \rho_{m \sim \mu}(Q(u_{\theta}(m,z),m,z)) + \cP(z)$, 
    using an approximate solution operator $u_{\theta}(m,z)$.
    In particular, standard techniques from optimization theory are used to show that when $\cJ$ is strongly convex,
    the optimization errors of first-order stationary points of the surrogate objective $\cJ_{\theta}$ 
    can be bounded in terms of errors in the objective's derivatives, $D_z \cJ - D_z \cJ_{\theta}$ (\Cref{theorem:optimization_error}).
    Moreover, in the risk neutral setting where $\rho(\cdot) = \bE[\cdot]$,
    we show that this derivative error of the objective is bounded by a combination of the $L^2_{\mu}$ errors of the output $\|u(m,z) - u_{\theta}(m,z)\|_{\cU}$, and the $z$-derivative $\|D_z u(m,z) - D_z u_{\theta}(m,z))\|_{\HS(\cZ,\cU)}$. These results provide theoretical justification for derivative-informed operator learning in the context of risk-neutral optimization, and can be easily extended to other settings such as risk averse optimization and inverse problems.

    \item[(C2)] \emph{Approximation theory for multi-input neural operators.}
    We establish universal approximation results for neural operators demonstrating that multi-input neural operators $u_{\theta}(m,z)$
    can approximate both the solution operator $u(m,z)$ and its $z$-derivatives $D_z u(m,z)$ arbitrarily well over compact domains for $z$ but possibly unbounded domains in $m$ (\Cref{theorem:rbno_ua_semibounded}). Moreover, in the risk neutral setting, this implies that the surrogate objective yields stationary points with arbitrarily small optimization error (\Cref{theorem:rbno_ouu_approx}). These results justify the use of neural operators for shape optimization under uncertainty as characterized by first-order stationary points of the corresponding optimization problems.

    \item[(C3)] \emph{Efficient algorithms for shape derivative-informed operator learning.}
    We develop practical algorithms for learning operators over families of geometries by representing arbitrary diffeomorphic transformations $\diffeo_z$ of a reference mesh via extensions. 
    We present efficient formulations of the governing PDEs posed either on the deformed domain or, equivalently, on the reference domain using the inverse diffeomorphism $\diffeo_z^{-1}$. 
    We further delineate conditions for existence and efficient computation of derivatives of the solution with respect to both $m$ and $z$, leveraging the implicit function theorem together with adjoint and forward sensitivity methods. Building on these components, we extend multi-input reduced-basis derivative-informed neural operators \cite{luo2023efficient} to the shape-parametric setting.

    We demonstrate the effectiveness of the proposed framework through a collection of challenging operator learning problems involving geometric variability, culminating in risk-averse PDE-constrained shape optimization problems. Empirically, the proposed approach yields two key benefits: 
    (i) improved optimization accuracy, consistent with the bounds on the optimization error developed for the risk neutral setting, 
    and (ii) improved generalization accuracy per unit of computational effort.
\end{itemize}

\subsection{Related Work}

\paragraph{Diffeomorphic mappings of geometries}
Diffeomorphisms have been widely used to parameterize transformations of computational domains across a range of settings. In the context of PDEs, such approaches play a central role in shape calculus and PDE-constrained shape optimization. 
Of particular interest are diffeomorphisms constructed from basis expansions, such as Fourier bases, B-splines, and NURBS. These are commonly used in shape optimization and inference (e.g., in the method of free-form deformations) 
due to their expressive power and close integration with computer aided design \cite{Samareh04,rozza2013free, Bui-ThanhGhattas12a, AkcelikBirosGhattasEtAl05, salmoiraghi_scardigli_telib_rozza_2018,he2019robust}.
Other classical developments include velocity-based formulations of diffeomorphic shape calculus \cite{delfour2011shapes,henry2005perturbation,SokolowskiZolesio92,haslinger2003introduction,henrot2018shape, pironneau1984optimal, 
YangStadlerMoserEtAl11}, the method of mappings \cite{hiptmair2015shape,paganini2018higher}, and approaches that construct diffeomorphic domain transformations via PDE-based extension operators, such as elastic deformations \cite{schulz2016efficient,geiersbach2023pde,onyshkevych2021mesh,GhattasLi98b}.
Closely related diffeomorphism-based formulations also arise in medical imaging, particularly in large-deformation registration and computational anatomy \cite{beg_miller_trouvé_younes_2005,miller2002metrics,trouve1998diffeomorphisms,mang2015inexact,mang2017semi,mang2019claire}.

\paragraph{Risk-averse shape optimization}

There has been substantial work done on risk-averse shape optimization under uncertainty \cite{conti2011risk,conti_rumpf_schultz_tölkes_2018,geihe_lenz_rumpf_schultz_2012,kodakkal2022risk,conti_held_pach_rumpf_schultz_2011,kodakkal_keith_khristenko_apostolatos_bletzinger_wohlmuth_wüchner_2022,matthias_heinkenschloss_kouri_2025}. These approaches typically rely on high-fidelity numerical discretizations, such as the finite element method, which restricts their applicability in high-dimensional stochastic settings due to the coupled computational costs of sampling and repeated PDE solves. Similar methods have also been employed in risk-averse topology optimization \cite{doi:10.1137/090754315, eigel_neumann_schneider_wolf_2018, martínez-frutos_herrero-pérez_kessler_periago_2018}.

\paragraph{Machine learning methods for shape optimization}
Motivated by the computational cost of shape optimization using traditional high-fidelity solvers, recent years have seen growing interest in machine learning–based approaches for learning PDE solution operators over varying geometries. Representative examples include geometry-aware variants of the Fourier Neural Operator (FNO) architecture \cite{li2021fourier,li_huang_anandkumar_liu_li_2023,zhao_liu_li_chen_liu_2025}, as well as DeepONet-based models and their multi-input extensions (MIONet) \cite{LuJinKarniadakis2019,jin2022mionet,yin_charon_brody_lu_trayanova_maggioni_2024}.

A variety of strategies have been proposed to incorporate geometric variability into these frameworks. These include learned deformation maps that transform irregular domains to regular latent grids compatible with FNOs \cite{li_huang_anandkumar_liu_li_2023}, low-dimensional geometric parameterizations such as control points or NURBS representations for applications including NACA airfoil optimization \cite{shukla_oommen_peyvan_penwarden_plewacki_bravo_ghoshal_kirby_karniadakis_2024}, and operator learning over diffeomorphic families of domains \cite{zhao_liu_li_chen_liu_2025,yin_charon_brody_lu_trayanova_maggioni_2024,cheng_hao_wang_huang_wu_liu_zhao_liu_su_2024}. Related neural operator–based approaches have also been proposed for topology optimization problems \cite{kou_yin_zhu_jia_luo_yuan_lu_2025}.

While these approaches demonstrate strong performance for deterministic geometric variability, they do not address auxiliary uncertain parameters or risk-averse optimization objectives. Moreover, all of the aforementioned works rely on standard $L^2_\mu$ training formulations and do not incorporate Fr\'echet derivative information into the learning process.

\paragraph{Derivative-informed operator learning}

This work extends derivative-informed operator learning \cite{o2024derivative} to reduced basis neural operator architectures designed to simultaneously accommodate uncertain parameters and varying geometries, with their associated derivatives. It also builds on \cite{luo2023efficient}, which developed a multi-input reduced-basis derivative-informed neural operator framework for risk-averse PDE-constrained optimization. 
In the present work, we substantially expand the approximation theory and focus on risk-averse shape optimization, which introduces substantial additional challenges compared to previous work. Derivative-informed neural operators (DINOs) have also shown substantial improvements in other optimization-related settings such as inverse problems \cite{cao2024lazydino,cao2025derivative}, and have been extended to various architectures such as DeepONet \cite{qiu2024derivative} and FNOs \cite{yao2025derivative}.

The approximation capabilities of neural operators have been thoroughly explored in existing literature for a range of architectures. Most relevant to the present work are results developed for reduced basis architectures, e.g., \cite{bhattacharya2021pca,Lanthaler23, HerrmannSchwabZech24, AdcockDexterMoraga24}, which establish universal approximation properties along with parametric complexity estimates under suitable regularity assumptions of the operator. 
Typically, such results focus on the error in the outputs of the operator but not its derivatives, the latter of which is important to the stability of the optimizers obtained when using the neural operator to approximate the objective function of optimization problems.
In this regard, universal approximation results that additionally include the derivative accuracy have been developed
for the reduced basis architecture in \cite{luo2025dis} and for FNOs in \cite{yao2025derivative}.
We build on these works to develop a universal approximation result for the reduced basis architecture tailored to the OUU setting.

\subsection{Layout of the Paper}
In Section~\ref{sec:formulation}, we introduce the class of PDE problems considered in this work and formally state the target learning and optimization tasks. This section provides the necessary mathematical background, including properties of the diffeomorphic mappings between the reference and deformed domains, and establishes error bounds for optimization problems in which the true PDE solution operator is replaced by a surrogate model. Section~\ref{sec:shape_dino} formulates the derivative-informed operator learning problem and presents universal approximation results 
for a class of reduced-basis neural operator architectures. In Section~\ref{sec:implementation_details} describes practical implementation aspects of the proposed framework, including mesh deformation via basis expansions and elastic extensions along with the construction of the reduced-basis neural operators. Finally, Section~\ref{sec:numerical_results} presents numerical experiments that demonstrate the effectiveness of the proposed methods on several challenging shape optimization problems, including two- and three-dimensional flow optimization under uncertainty with risk-averse design objectives. The results show that, for a fixed amount of training data, the shape derivative-informed approach yields substantially higher optimization accuracy than neural operators trained without derivative information. Moreover, when applied to risk-averse shape optimization, Shape-DINO can achieve greater accuracy in producing optimal designs than conventional PDE-based sample-average approximation methods while requiring 1–2 orders of magnitude fewer state PDE solves (including the cost of generating training data). 
We further demonstrate that the construction cost of Shape-DINO can be amortized across many design objectives and risk measures.
In terms online costs, Shape-DINO enables 3–8 orders-of-magnitude speedups for state evaluations and gradient computations.

\section{Formulation and Motivation}\label{sec:formulation}
\subsection{Notation for Operator Learning on Hilbert Spaces}

We begin by introducing some common notation used for operator learning on Hilbert spaces.
Additional notation will be introduced where they are used.
In this work, we will consider state variables, model parameters, and design variables that belong to (finite- or infinite-dimensional) separable Hilbert spaces over $\bR$. 
Suppose $\cX$ is such a Hilbert space, we will use $\langle v_1, v_2 \rangle_{\cX}$ to denote the inner product on $\cX$, which generates the norm $\|v\|_{\cX} = \sqrt{ \langle v, v \rangle_{\cX}}$.
We will use $\cX'$ to denote its topological dual and $\langle l, v \rangle_{\cX' \times \cX}$ to denote the duality pairing between $l \in \cX'$ and $v \in \cX$.
For Hilbert spaces, $\cX'$ can be identified with $\cX$ by the Riesz map, $\mathsf{R}_{\cX} : \cX' \rightarrow \cX$, i.e., for $l \in \cX'$, 
$\langle \mathsf{R}_{\cX} l, v \rangle_{\cX} = \langle l, v \rangle_{\cX' \times \cX}.$

We use $\Op(\cX, \cY)$ to denote the space of bounded linear operators between the two Hilbert spaces $\cX$ and $\cY$.
This is a Banach space with the operator norm $\|A\|_{\Op(\cX,\cY)} = \sup_{\|v\|_{\cX} = 1} \|A v\|_{\cY}$.
For any $A \in \cL(\cX, \cY)$, we let $A^* \in \cL(\cY, \cX)$ denote the adjoint of $A$
and $A^T \in \cL(\cY', \cX')$ denote the transpose of $A$, 
such that $\langle w, Av \rangle_{\cY} = \langle A^* w, v \rangle_{\cX}$ and
$\langle l, Av \rangle_{\cY' \times \cY} = \langle A^T l, v \rangle_{\cX' \times \cX}$
for any $v \in \cX, w \in \cY$, and $l \in \cY'$.
In particular, when $v \in \cX$ is a vector, we have $v^T = \mathsf{R}_{\cX}^{-1} v$. 
We use this notation to define the outer product of two vectors $v_1$ and $v_2$, 
denoted $v_1 v_2^T$, 
which is the linear operator given by $(v_1 v_2^T)v_3 = v_1 \langle v_2, v_3 \rangle_{\cX}$.

We will also consider the space of Hilbert--Schmidt operators $\HS(\cX,\cY)$, which include continuous linear operators that have finite Hilbert--Schmidt norm, defined as 
\[
    \|A \|_{\HS(\cX,\cY)}^2 = \sum_{i = 1}^{\infty} \| A v_i \|_{\cY}^2,
\]
where $\{v_i\}_{i=1}^{\infty}$ is any orthonormal basis of $\cX$.
We note that when either $\cX$ or $\cY$ is finite dimensional, the Hilbert--Schmidt norm is equivalent to the operator norm. 

We will represent uncertainty in the model parameters through Borel probability measures $\mu$ on Hilbert spaces.
For Hilbert spaces $\cX$ and $\cY$, 
we let $L^p_{\mu}(\cX; \cY)$ 
denote the Bochner space of operators $F : \cX \rightarrow \cY$ that are Bochner integrable with finite norm, i.e.,
\[
    \|F\|_{L^p_{\mu}(\cX;\cY)}^p := \bE_{x \sim \mu}[ \| F(x) \|_{\cY}^p] = \int \|F(x)\|^p_{\cY}  d\mu(x) < \infty
\]
when $p \in [1, \infty)$, and
\[
    \|F\|_{L^{\infty}_{\mu}(\cX;\cY)} = \esssup_{x} \|F(x)\|_{\cY}.
\]
when $p = \infty$. Note that we will occasionally use the shorthand $L^p_{\mu}$ instead of $L^p_{\mu}(\cX;\cY)$ when the input and output spaces are clear.

Additionally, we will consider mappings that are Fr\'echet differentiable. 
For such a mapping, $F : \cX \rightarrow \cY$, we denote its derivative at $x \in \cX$ as $D_x F(x) \in \Op(\cX, \cY)$, such that 
\begin{equation*}
    \lim_{\|h\|_{\cX} \rightarrow 0} \frac{\|F(x + h) - F(x) - D_x F(x) h \|_{\cY}}{\|h\|_{\cX}} = 0.
\end{equation*}
Moreover, we will write the space of continuously Fr\'echet differentiable mappings as 
$C^1(\cX;\cY)$ or $C^1(\cB; \cY)$ 
if we are considering a subset of the input space $\cB \subset \cX$.

\subsection{Partial Differential Equations on Parametric Domains}
For shape optimization, we consider a family of spatial domains $\Omegaz \subset \bR^{d}$, with $d = 2,3$, parametrized by finite-dimensional shape variables $z \in \cZ = \bR^{\dz}$.
Specifically, we assume each $\Omegaz$ is coupled to a reference domain $\Omegaref $ through a 
diffeomorphism $\diffeo_{z} : \Omegau \rightarrow \bR^{d}$ defined by the shape parameters, where $\Omegau \subset \bR^{d}$ is a hold-all domain containing the closure of $\Omegaref$.
Moreover, we assume that the reference domain $\Omegaref$ is bounded and Lipschitz, such that for each $z$ within some closed admissible set $\cZ_{ad} \subset \cZ$, the resulting spatial domain $\Omegaz = \diffeo_z(\Omegaref)$ is also bounded and Lipschitz.

We are concerned with PDE problems defined on the family of domains $\Omegaz$.
Such PDE problems relate the shape variable $z$, along with any additional model parameters $\mz \in \cMz$ such as spatially varying coefficient fields or forcing terms, 
to the state variable $\uz \in \cUz$.
Here, $\cMz$ and $\cUz$ are separable Hilbert spaces for the model parameters and the state, respectively, and are defined over the spatial domain $\Omegaz$.
These are typically given as Lebesgue spaces $L^2(\Omegaz)$ or Sobolev spaces $H^k(\Omegaz)$. Then, for each instance of $\Omega_z$, 
the PDE problem can be stated abstractly in residual form as 
\begin{equation}\label{eq:res_form_Omegaz}
    \text{given } \Omegaz \text{ and } \mz \in \cMz, \quad \text{find } \uz \in \cUz \quad \text{such that} \quad \Rz(\uz, \mz) = 0,
\end{equation}
where $\Rz : \cUz \times \cMz \rightarrow \cVz'$ is a differential operator whose output lies in the topological dual of some test space $\cVz'$.

Alternatively, we can state the PDE on the reference domain $\Omegaref$ using the pullback through $\diffeo_z$. 
To this end, let $\cMref $ and $\cUref $ be the corresponding function spaces defined on the reference domain $\Omegaref$. 
Then, \eqref{eq:res_form_Omegaz} can be equivalently formulated on $\Omegaref$ as 
\begin{equation} \label{eq:res_form_Omega0}
    \text{given }z\in \mathcal{Z} \text{ and } \mref \in \cMref,  \quad \text{find } \uref \in \cUref 
    \quad \text{such that} \quad \Rref(\uref,\mref,z) = 0,
\end{equation}
where \eqref{eq:res_form_Omega0} is derived from \eqref{eq:res_form_Omegaz} 
through the fact that $\uref = \uz \circ \diffeo_z$, $\mref = \mz \circ \diffeo_z$, 
and $\Rref : \cUref  \times \cMref \times \cZ \rightarrow \cVref'$ 
is given by the appropriate change of variable of the associated operators.
We list the coordinate transformations for some commonly used differential and integral operations in \Cref{appendix:coordinate_transforms}. Note that in \eqref{eq:res_form_Omega0}, the shape variables appear explicitly in the form of the residual.
Under the assumption that the PDE admits a unique solution for every $z \in \cZ_{ad}$ and $\mref \in \cMref$ 
such that we can abstractly write 
$\uref = \uref(\mref, z)$
where we refer to $\uref : \cMref \times \cZ_{ad} \rightarrow \cUref$ as the solution operator.

\begin{remark}[On notation]\label{remark:geometry_notation}
For the remainder of our exposition, we will adopt the following notational conventions. We will use $\widetilde{\cdot}$ over a symbol to indicate that this quantity is defined in terms of the spatial domain $\Omegaz$, while quantities without $\widetilde{\cdot}$ will generally be used for their reference domain counterparts, such as in $\uz$ and $\uref$ above.
For geometric quantities, we will use $n, dx, ds$ to denote unit normal vectors, volume elements, and surface elements on the spatial domain and $N, dX, dS$ to that on the reference domain.
Moreover, when writing $\cM$, $\cU$, and $\cV$ without the associated domain, it is assumed that these refer to the reference domain function spaces $\cM := \cMref$, $\cU := \cUref$, and $\cV := \cVref$.
\end{remark}

Before proceeding, we will introduce the notions of \textit{pushforward} and $\textit{pullback}$ operators that will be of use in subsequent sections. 
For any diffeomorphism $\diffeo : \Omega_0 \rightarrow \Omega_1$, 
we let $T_{\diffeo}$ denote the composition operator, $T_{\diffeo} u = u \circ \diffeo$.
We refer to $T_{\diffeo}$ as the \textit{pullback operator} and refer to $T_{\diffeo}^{-1} = T_{\diffeo^{-1}}$ as the \textit{pushforward operator} 
(i.e., $T_{\diffeo}^{-1}u = u \circ \diffeo^{-1}$).
Under these definitions, the changes of variables used in \eqref{eq:res_form_Omegaz} and \eqref{eq:res_form_Omega0} 
can be expressed as 
\[ 
    \uref = T_{\diffeo_z} \uz, \quad \mref = T_{\diffeo_z} \mz
    \quad \text{or} \quad 
    \uz = T_{\diffeo_z}^{-1} \uref, \quad \mz= T_{\diffeo_z}^{-1} \mref,
\]
when $m$ and $u$ are scalar fields. 

When $\diffeo_z \in W^{1,\infty}(\Omegaref)$ and $\diffeo_z^{-1} \in W^{1, \infty}(\Omega_z)$, the resulting composition operators are bounded linear operators 
between low order Sobolev spaces on the reference/spatial domains.
\begin{proposition}\label{prop:pullback_pushforward_operators}
    Suppose $\diffeo : \Omega_0 \rightarrow \Omega_1$ is a diffeomorphism such that 
    $\diffeo \in W^{1,\infty}(\Omega_0)$ and $ \diffeo^{-1} \in W^{1,\infty}(\Omega_1)$.
    We have the following:
    \begin{enumerate}
        \item The pushforward and pullback operators are bounded linear operators between $L^2(\Omega_0)$ and $L^2(\Omega_1)$, 
        i.e., $T_{\diffeo}^{-1} \in \Op(L^2(\Omega_0), L^2(\Omega_1))$
        and
        $T_{\diffeo} \in \Op (L^2(\Omega_1), L^2(\Omega_0) )$.
        Moreover, their norms are bounded by
        \begin{subequations}
        \begin{align}
            \|  T_{\diffeo}^{-1} \|_{\Op(L^2(\Omega_0),L^2(\Omega_1))}^2
            &\leq 
                \|\diffeo\|_{W^{1,\infty}(\Omega_0)}^{d},
            \\ 
            \|  T_{\diffeo} \|_{\Op(L^2(\Omega_1),L^2(\Omega_0))}^2
            &\leq 
                \|\diffeo^{-1}\|_{W^{1,\infty}(\Omega_1)}^{d}.
        \end{align}
        \end{subequations}

        \item The pushforward and pullback operators are bounded linear operators between $H^1(\Omega_0)$ and $H^1(\Omega_1)$, 
        i.e., $T_{\diffeo}^{-1} \in \Op(H^1(\Omega_0), H^1(\Omega_1))$
        and
        $T_{\diffeo} \in \Op (H^1(\Omega_1), H^1(\Omega_0) )$.
        Moreover, their norms are bounded by
        \begin{subequations}
        \begin{align}
            \|  T_{\diffeo}^{-1} \|_{\Op(H^1(\Omega_0),H^1(\Omega_1))}^2
            &\leq 
                \left(
                    1 + \|\diffeo^{-1}\|_{W^{1,\infty}(\Omega_1)}^2
                \right)
                \|\diffeo\|_{W^{1,\infty}(\Omega_0)}^{d}
            ,
            \\ 
            \|  T_{\diffeo} \|_{\Op(H^1(\Omega_1),H^1(\Omega_0))}^2
            &\leq 
                \left(
                    1 + \|\diffeo \|_{W^{1,\infty}(\Omega_0)}^2
                \right)
                \|\diffeo^{-1}\|_{W^{1,\infty}(\Omega_1)}^{d}
            .
        \end{align}
        \end{subequations}
    \end{enumerate}
\end{proposition}
\begin{proof}
    See \Cref{appendix:proof_pushforward}.
\end{proof}

The fact that problems of the form \eqref{eq:res_form_Omegaz} can be equivalently formulated as \eqref{eq:res_form_Omega0} is of essential importance to this work. 
Due to this fact, we can formulate both the shape optimization problem and the operator learning task for all $\Omega_z$ on $\Omegaref$ via the parametrization with shape parameter $z$, which now explicitly appears in the residual operator $\Rref$ in \eqref{eq:res_form_Omega0}.

\subsection{Derivatives of the Solution Operator}
In this work, we assume that that solution operator is continuously differentiable with respect to $m$ and $z$.
Generally, this allows us to solve shape optimization problems involving the PDE by derivative-based optimization algorithms.
More specific assumptions will be made in subsequent sections when deriving formal theoretical results.

Since we are working directly in the reference domain, the Fr\'echet derivative operators $D_m u(m,z)$ and $D_z u(m,z)$ can be computed through implicit differentiation of the PDE $\eqref{eq:res_form_Omega0}$ as in standard forward and adjoint sensitivity methods, and does not need to explicitly involve shape calculus.
That is, at any $m, z$ and $u(m,z)$ and for any direction of the input space $(h_{m}, h_{z}) \in \cM \times \cZ$, 
the derivative action, or the Jacobian-vector product (JVP), $h_{u} := D_m u(m,z) h_{m} + D_z u(m,z) h_{z}$ is given by 
the solution of the linearized state equation
\begin{equation}\label{eq:linearized_forward}
    D_u \Rref(u,m,z)h_{u} = - \left( D_m \Rref(u,m,z) h_{m} + D_z \Rref(u,m,z) h_{z} \right).
\end{equation}
On the other hand, to compute the action of the derivative's adjoint, or the vector-Jacobian product (VJP), in some direction $v \in \cU$,
we solve the adjoint equation 
\begin{equation}\label{eq:linearized_adjoint}
    D_u \Rref(u,m,z)^T p = -v^T,
\end{equation}
for the adjoint variable $p \in \cU$, such that 
\begin{equation}
   v^T D_m u(m,z) = D_m \Rref(u,m,z)^T p, \qquad 
   v^T D_z u(m,z) = D_z \Rref(u,m,z)^T p.
\end{equation}
The forms above imply that having already computed the solution $u = u(m,z)$, the cost of subsequently computing a single VJP or JVP  amounts to solving a single linear PDE with the operator $D_u \Rref$ or its adjoint (transpose).
Moreover, at a single point $(m,z,u(m,z))$, the computation of multiple derivative actions always shares the same linear operator $D_u \Rref$, with the only change being the right hand side of \eqref{eq:linearized_forward} or \eqref{eq:linearized_adjoint}; this often leads to significant computational savings when linear algebraic operations are amortizable.

We note that although we will frequently refer to derivatives with respect to $z$ as \textit{shape derivatives}, 
this type of differentiation aligns more closely with the notion of \textit{material derivatives} in shape optimization, 
since the derivatives are not only affected by the motion of the domain boundary, $\partial \Omega_z$, 
but also by the motion of the interior through the definition of $\diffeo_z$ (for example, see \cite[Chapter 11]{ManzoniQuarteroniSalsa21}).

\subsection{Shape Optimization under Uncertainty}
Let $Q_z : \cUz \times \cMz \rightarrow \bR$
define a scalar quantity of interest (QoI) that depends on the model parameters, state variables, and indexed by the shape variables $z$ that define the domain $\Omega_z$. 
Commonly, these are formulated as integral quantities over the domain $\Omega_z$ or its subsets, such as in tracking objectives, forces, energies, etc.
In shape optimization problems, we are interested in minimizing $Q_z(\uz, \mz)$ with respect to $z \in \cZ_{ad}$ subject to the constraint that $(\uz, \mz)$ solves the PDE \eqref{eq:res_form_Omegaz}.
As in the case of the PDE residual, 
we can reformulate the QoI by writing it in terms of reference-domain quantities, such that 
\begin{equation}
    Q(\uref, \mref, z) = Q_z (\uz, \mz),
\end{equation}
with $\uref = T_{\diffeo} \uz$ and $\mref = T_{\diffeo} \mz$.
One then aims to minimize $Q( \uref, \mref, z)$ where $\uref = \uref(\mref, z)$ solves \eqref{eq:res_form_Omega0}.

Often, the model parameters $\mref$ are uncertain, 
which implies that for each fixed $z$, the resulting value of the QoI is also uncertain.
In this work, we model the uncertainty in $\mref$ through its probability distribution $\mu_m$, defined as a Borel probability measure over $\cM = \cMref$.
Subsequently, for every $z \in \cZ_{ad}$, the quantity $Q(\uref(\cdot, z), \cdot, z)$ is a random variable whose value depends on the realization of $m$.
Thus, under appropriate moment assumptions on $Q(\uref(\cdot, z), \cdot)$ 
(e.g., it belongs to $L^p_{\mu_{m}}(\cM;\bR)$)
the shape optimization problem can be formulated using a risk measure $\rho : L^p_{\mu_m} \rightarrow \bR$; a functional that provides a statistical summary of the distribution and chosen to express a particular risk preference. 
The shape OUU problem is then given by \begin{equation}\label{eq:reference_domain_ouu} 
    \min_{z \in \cZ_{ad}} \cJ(z) = \rho_{m \sim \mu_m}( Q(\uref(m, z), m, z)) + \cP(z)
\end{equation}
where $\cP : \cZ \rightarrow \bR$ is a penalization/regularization term that accounts for any additional cost of the shape associated with $z$.

In this work, we consider the following risk measures in our numerical examples, which we introduce below in terms of a generic random variable $\xi \sim \mu_{\xi}$.

\paragraph{Mean}
The canonical example of a risk measure is the mean (or expectation) of the random variable, 
\begin{equation}\label{eq:mean_risk}
    \rho_{\xi \sim \mu_{\xi}}^{\mean}(\xi) = \bE_{\xi \sim \mu_{\xi}}[\xi].
\end{equation}
The mean is risk neutral and does not explicitly account for events in the upper tails of the distributions, which often represent catastrophic failures in engineering systems.
For this reason, we also consider two risk-averse functionals as examples.

\paragraph{Conditional value-at-risk}
The conditional value-at-risk (CVaR), which is defined for a quantile $\beta \in [0,1]$ 
\begin{equation}
    \rho_{\xi \sim \mu_\xi}^{\CVaR,\beta}(\xi) = \min_{t \in \bR}{t + \frac{1}{1-\beta}\bE[(\xi - t)^{+}]},
\end{equation}
where $(x)^{+} := \max(x,0)$. 
When the underlying random variable $\xi$ is continuously distributed, the CVaR can be interpreted as the conditional expectation of $\xi$ given that it exceeds its $\beta$ quantile.
As a result, the CVaR can be used to explicitly control the upper tail of the distributions and has been explored for risk-averse design in engineering \cite{kouri2016risk,ChaudhuriKramerNortonEtAl22}

\paragraph{Entropic risk measure} 
The entropic risk measure is given by 
\begin{equation}
    \rho_{\xi \sim \mu_\xi}^{\entropic, \beta}(\xi) = \frac{1}{\beta} \log \left( \bE [\exp(\beta \xi)] \right),
\end{equation}
where the parameter $\beta > 0$ specifies the degree of risk aversion. 
The entropic risk measure has also been considered for risk-averse optimization of physical systems, such as in \cite{guth_kaarnioja_kuo_schillings_sloan_2024}.

\begin{remark}[On the distribution of the model parameters]
In this work, we will focus on the case where the distribution of the model parameter $\mu_m$ is defined directly on the reference domain.
In this case, the uncertain parameter field in the spatial domain $\Omega_z$ for any $z$ is the pushforward (in terms of the spatial variables) by $\diffeo_z$ 
from a random field defined on the reference domain $\Omegaref$, i.e., $\mz = T_{\diffeo_z}^{-1} \mref$ and $\mref \sim \mu_m$.
This choice is made to simplify the notation and the required variable transformations.
Nevertheless, our formulations and procedures can be extended to cases where the uncertain parameter field is prescribed on the spatial domain via more complicated transformations of some underlying random element.
\end{remark}

\subsection{Optimization via Sample Average Approximation}
One conventional approach to solving the OUU problem \eqref{eq:reference_domain_ouu}
is by replacing expectations with sample average approximations (SAA) taken from i.i.d.~samples $\mref^{(i)} \sim \mu_m, i = 1, \dots, n_{\SAA}$.
This defines a new objective function, in which the risk measure component of \eqref{eq:reference_domain_ouu} is replaced by its SAA. 

For example, when using $\rho = \rho^{\mean}$, 
the SAA of the objective function becomes
\begin{equation}\label{eq:mean_saa}
    \widehat{\cJ}^{\mean}(z) := \frac{1}{n_{\SAA}}\sum_{i=1}^{n_{\SAA}} Q(\uref(\mref^{(i)}, z), \mref^{(i)}, z).
\end{equation}
Similarly, for the entropic risk measure, we take 
\begin{equation}\label{eq:entropic_saa}
    \widehat{\cJ}^{\entropic}(z) := \frac{1}{\beta}
        \log \left( \frac{1}{n_{\SAA}}\sum_{i=1}^{n_{\SAA}} \exp( \beta Q(\uref(\mref^{(i)}, z), \mref^{(i)}, z)) \right).
\end{equation}
In the case of the CVaR, in addition to the SAA, we will also take a smooth approximation to the maximum function, $(x)_{\varepsilon}^{+} \approx (x)^{+}$,
\begin{equation}
  (x)^+_\epsilon =
    \begin{cases}
      0, & \text{if $x < 0$}\\
      \frac{x^3}{\epsilon^2} - \frac{x^4}{2\epsilon^3}, & \text{if $ 0 \leq x \leq \epsilon$}\\
      x - \frac{\epsilon}{2}, & \text{$x \geq \epsilon$},
    \end{cases}       
\end{equation}
with a small smoothing parameter $\epsilon \ll 1$ \cite{kouri2016risk}. 
This smoothing preserves convexity and ensures uniform convergence $(x)^+_\epsilon \rightarrow (x)^+$ as $\epsilon \rightarrow 0$. 
The resulting approximation to the CVaR objective function is then given by
\begin{equation}\label{eq:smoothed_cvar_saa}
    \widehat{\cJ}^{\CVaR, \beta, \varepsilon}(z,t)
    := t + \frac{1}{n_{\SAA}(1 - \beta)}\sum_{i=1}^{n_{\SAA}} 
    (Q(\uref(\mref^{(i)}, z), \mref^{(i)}, z) - t)_{\varepsilon}^{+},
\end{equation}
which preserves differentiability with respect to $(z,t)$.
Finally, the SAA for the objective function is a combination of the risk measure 
and the penalization, which we minimize with respect to $z \in \cZ_{ad}$
\begin{equation}
    \min_{z \in \cZ_{ad}} \widehat{\cJ}(z) := \widehat{\cJ}^{\mathrm{risk}}(z) + \cP(z) 
\end{equation}
or in the case of CVaR, additionally over $t \in \bR$.

The resulting SAA objective function can then be minimized by conventional derivative-based optimization algorithms, e.g., quasi Newton methods (or Newton if the problem has sufficient smoothness), 
where the derivatives can be efficiently computed by the adjoint method.
That is, in order to compute the derivative of 
the mapping 
$
    z \mapsto Q(u(m^{(i)}, z), m^{(i)}, z)
$
with respect to $z$, we can proceed by the chain rule, 
\[
    D_z (Q(u(m^{(i)}, z), m^{(i)}, z))
    =D_u Q(u(m^{(i)},z), m^{(i)}, z)) D_z u(m^{(i)}, z) + D_z Q(u(m^{(i)}, z), m^{(i)}, z).
\]
The first term is simply a VJP of the solution operator $(v^{(i)})^T D_z u(m^{(i)}, z)$ in the direction 
\[ v^{(i)} = \mathsf{R}_{\cV} D_u Q(u(m^{(i)},z), m^{(i)}, z)\]
and can be evaluated by solving the adjoint PDE \eqref{eq:linearized_adjoint} once for each $m^{(i)}$,
while the latter term typically has an analytically available form and does not require solving additional PDEs.

Evidently, the per-iteration computational cost of shape optimization under uncertainty 
tends to be dominated by (1) evaluating the objective function via $n_{\SAA}$ solutions of the state PDE (i.e., evaluating $u(m^{(i)},z)$)
and (2) subsequently evaluating its gradient through $n_{\SAA}$ solutions of the adjoint equation.
For this reason, accurately solving PDE-constrained OUU problems can often lead to prohibitive computational costs.
This is especially true when the risk measure requires a large number of samples to accurately estimate and when the underlying PDE is expensive to solve.

\subsection{Optimization using Operator Surrogates}
In order to alleviate the computational cost associated with PDE-constrained OUU, 
one can replace the high-fidelity PDE model with a surrogate model that is of lower fidelity but is fast to evaluate.
In this work, we are interested in operator surrogates, i.e., 
approximations to the solution operator $u_{\theta}(m,z) \approx u(m,z)$.
This defines an approximation to the optimization problem \eqref{eq:reference_domain_ouu}, i.e.,
\begin{equation}\label{eq:surrogate_ouu_objective}
    \min_{z \in \cZ_{ad}} \cJ_{\theta}(z) := \rho_{m \sim \mu_m}(Q(u_{\theta}(m,z),m,z)) + \cP(z)
\end{equation}
which uses the surrogate objective function $\cJ_{\theta}(z)$ as an approximation to the true objective function $\cJ(z)$.

Since the surrogate model is fast to evaluate, \eqref{eq:surrogate_ouu_objective} can be rapidly solved using an SAA with large number of samples, essentially eliminating the sampling error.
Instead, the approximation error of the surrogate model becomes the dominant source of error in solving the OUU problem. 
Thus, the tradeoff in surrogate-based optimization becomes one of balancing the approximation accuracy of the surrogate with its construction cost
measured in the size of the training dataset. Conventionally, the training dataset consists of input-output pairs, $(m^{(i)}, z^{(i)}) \mapsto u^{(i)} = u(m^{(i)}, z^{(i)})$ over the input space,
where each data tuple is generated by solving the underlying PDE.

Moreover, as noted in \cite{luo2023efficient}, when $\cZ$ is high dimensional, 
derivative-based algorithms are still required to effectively solve \eqref{eq:surrogate_ouu_objective} (or its SAA). 
Thus, the errors introduced by the surrogate in approximating $\cJ(z)$
and its derivatives $D_z \cJ(z)$ (or equivalently, its gradients $\mathsf{R}_{\cZ} D_z \cJ(z)$) 
are both important.
In particular, mismatch of the derivatives ($D_z \cJ_{\theta}(z)$ compared to $D_z \cJ(z)$) can introduce spurious local minima in $\cJ_{\theta}(z)$ 
even when the original PDE-constrained OUU problem is well-posed.

For example, when the underlying objective function is strongly convex and differentiable, the error in the approximate stationary points are controlled explicitly by their gradient errors.
We summarize this in the following theorem, the proof of which is based on standard variational analysis and is presented in \Cref{appendix:opt_error}.
\begin{theorem}\label{theorem:optimization_error}
    Suppose $\openset \subset \cZ$ is an open set and $\cJ \in C^1(\openset; \bR)$ is a $\lambda$-strongly convex objective function, i.e., given $\lambda > 0$, 
    \begin{equation}
        \cJ(z_2) \geq \cJ(z_1) + \langle D_z \cJ(z_1), z_2 - z_1 \rangle_{\cZ' \times \cZ} + \frac{\lambda}{2} \|z_2 - z_1\|_{\cZ}^2
    \end{equation}
    for any $z_1, z_2 \in \openset$.
    Additionally, suppose $\cZ_{ad} \subset \openset$ is a closed, convex, and nonempty set, and let $z^{\star}$ be the unique minimizer of $\cJ$ over $\cZ_{ad}$.
    Furthermore, let $\cJ_{\theta} \in C^1(\openset; \bR)$ be an approximation of $\cJ$ and let $z^{\dagger} \in \cZ_{ad}$ be any point satisfying the first order optimality conditions, i.e.,
    \begin{equation}
        \langle D_z \cJ_\theta(z^{\dagger}), z - z^{\dagger} \rangle_{\cZ' \times \cZ} \geq 0 \qquad \forall z \in \cZ_{ad}. 
    \end{equation}
    Then, the optimization error is bounded by 
    \begin{equation}
        \| z^{\star} - z^{\dagger} \|_{\cZ}  \leq \frac{1}{\lambda} \| D_z \cJ(z^{\dagger}) - D_z \cJ_{\theta} (z^{\dagger}) \|_{\cZ'}
    \end{equation}
    and the optimality gap is bounded by 
    \begin{equation}
        \cJ(z^{\dagger}) - \cJ(z^{\star}) \leq \frac{1}{2 \lambda}\| D_z \cJ(z^{\dagger}) - D_z {\cJ}_{\theta}(z^{\dagger}) \|_{\cZ'}^2.
    \end{equation}
\end{theorem}

\subsection{Example: Derivative Error for the Expectation}
\Cref{theorem:optimization_error} suggests that it is important to control the derivative error $D_z \cJ - D_z \cJ_{\theta}$.
Formally, by virtue of the chain rule, this error depends on both the surrogate's output error $u(m,z) - u_{\theta}(m,z)$ and derivative error $D_z u(m,z) - D_z u_{\theta}(m,z)$, as shown in \cite{luo2023efficient}.
As an example, we will consider the case for the expectation $\rho_{\mean}$ and show the specific form of this dependence.

To this end, we will focus on solution operators that satisfy additional continuity and moment assumptions over a compact admissible set.
First, recall that we are using the shorthand notation $\cM := \cMref$ and $\cU := \cUref$ to denote the reference domain function spaces.
We then make the following assumptions on the solution operator (in the reference domain) and the distribution of $m$.
\begin{assumption}\label{assumption:pde_properties}
    Given a Borel probability measure $\mu_m$ on $\cM$ with finite $p^{th}$ moments,
    a compact admissible set $\cZ_{ad} \subset \cZ = \bR^{\dz}$,
    and an open set $\openset \supset \cZ_{ad}$,
    assume that the following hold:
    \begin{enumerate}
        \item The solution operator $u : \cM \times \openset \rightarrow \cU$ is continuously differentiable.
        \item There exist constants $\lip_{u} > 0$ and $\lip_{u_z} > 0$ such that for every $z_1, z_2 \in \openset$ and for every $m_1, m_2 \in \cM$, 
        \begin{subequations}
        \begin{align}
            \| u(m_2 ,z_2) - u(m_1 ,z_1)\|_{\cU} & \leq \lip_{u} \left( \|m_2 - m_1\|_{\cM} + \| z_2 - z_1\|_{\cZ}
            \right) \label{eq:u_lip}
            , \\
            \label{eq:Dzu_lip}
            \| D_z u(m_2, z_2) - D_z u(m_1, z_1)\|_{\HS(\cZ,\cU)} & \leq \lip_{u_z} \left( \|m_2 - m_1\|_{\cM} + \| z_2 - z_1\|_{\cZ} \right).
        \end{align}
        \end{subequations}
    \end{enumerate}
\end{assumption}
    
We will also work with assumptions on the quantity of interest $Q : \cU \times \cM \times \cZ \rightarrow \bR$ 
that accommodate quadratic forms in $u$, since these often arise in optimal design problems (e.g., tracking objective, drag reduction).
\begin{assumption}\label{assumption:qoi_properties}
Given a compact admissible set $\cZ_{ad} \subset \cZ = \bR^{\dz}$,
and an open set $\openset \supset \cZ_{ad}$,
assume that the following hold:
\begin{enumerate}
    \item The quantity of interest $Q : \cU \times \cM \times \openset \rightarrow \bR$ is twice continuously differentiable.
    \item There exist constants $\lip_{Q_u} > 0$ and $\lip_{Q_{uz}} > 0$ such that for every $u_1, u_2 \in \cU$, $z \in \openset$, and every $m \in \cM$,
    \begin{subequations}
    \begin{align}
        \| D_u Q(u_2, m, z) - D_u Q(u_1, m, w) \|_{\cU'} &\leq \lip_{Q_u} \|u_2 - u_1 \|_{\cU} , \\
        \| D_{uz} Q(u_2, m, z) - D_{uz} Q(u_1, m, z) \|_{\Op(\cU, \cZ')} & \leq \lip_{Q_{uz}} \|u_2 - u_1 \|_{\cU}.
    \end{align}
    \end{subequations}
    \item There exist constants $\ubQ,\ubDuQ,\ubDzQ,\ubDuzQ \geq 0$ such that for every $z \in \openset$ and every $m \in \cM$,
    \begin{subequations}
    \begin{align}
        |Q(0, m, z) | &\leq \ubQ, \\
        \|D_u Q(0, m, z) \|_{\cU'} & \leq \ubDuQ, \\ 
        \|D_z Q(0, m, z) \|_{\cZ'} & \leq \ubDzQ, \\ 
        \|D_{uz} Q(0, m, z) \|_{\Op(\cU, \cZ')} & \leq \ubDuzQ.
    \end{align}
    \end{subequations}
    
\end{enumerate}
    
\end{assumption}

Under \Cref{assumption:pde_properties} and \Cref{assumption:qoi_properties},
we can show that the gradient approximation error of the expectation 
depends on both the output error and the derivative error as measured in the $L^2_{\mu_m}$ sense.
This is summarized in the following proposition, the proof of which is presented in \Cref{appendix:grad_error_summary}.

\begin{proposition}\label{prop:surrogate_gradient_error}
    Suppose $u, u_{\theta}$ satisfy \Cref{assumption:pde_properties} and $Q$ satisfies \Cref{assumption:qoi_properties}
    with the Borel probability measure $\mu_m$, the compact admissible set $\cZ_{ad}$, and the open set $\openset$ for some $p \geq 2$.
    Let $\cJ^{\mean}(z) := \bE_{m \sim \mu_m}[Q(u(m,z),m,z)]$ and $\cJ_{\theta}^{\mean}(z) := \bE_{m \sim \mu_m}[Q(u_{\theta}(m,z),m,z)]$.
    Then, for any $z \in \cZ_{ad}$, the error in the derivatives of the cost functions is given by 
    \begin{align}
        & \| D_z \cJ^{\mean}(z) - D_z \cJ_{\theta}^{\mean} (z) \|_{\cZ'} 
        \nonumber \\
        & \qquad \leq 
        C_1 \left( 
            \|u(\cdot, z) - u_{\theta}(\cdot, z)\|_{L^2_{\mu_m}(\cM;\cU)}
            + \|D_z u(\cdot, z) - D_z u_{\theta}(\cdot, z)\|_{L^2_{\mu_m}(\cM;\HS(\cZ,\cU))}
        \right) \nonumber \\ 
        & \qquad + C_2 \left ( 
            \|u(\cdot, z) - u_{\theta}(\cdot, z)\|_{L^2_{\mu_m}(\cM;\cU)}^2 
            + \|D_z u(\cdot, z) - D_z u_{\theta}(\cdot, z)\|_{L^2_{\mu_m}(\cM;\HS(\cZ,\cU))}^2
            \right) ,
    \end{align}
    where $C_1$ and $C_2$ are constants independent of $u_{\theta}$.
\end{proposition}

\begin{remark}[On \Cref{assumption:pde_properties,assumption:qoi_properties}]
    By inspecting the proof of \Cref{prop:surrogate_gradient_error},
    alternative results are available in which the Lipschitz and bounding constants in \Cref{assumption:pde_properties,assumption:qoi_properties} are allowed to depend on $m$ with appropriate moment assumptions.
    As a result, the gradient errors will depend on higher moments of the output and derivative errors.
    One can also consider gradient errors for generic risk measures of the form $f_2(\bE[f_1(\xi)])$ with appropriate growth conditions on $f_1$ and $f_2$. 
    This can similarly lead to gradient errors depending on higher moments of the output and derivative errors.
    We do not explore these extensions for conciseness.
\end{remark}

\section{Shape Derivative-Informed Operator Learning}\label{sec:shape_dino}
Motivated by the previous discussion, we seek formulations of operator learning for shape varying systems that provide theoretical and practical guarantees for controlling the derivative errors discussed in the previous section. 
In this section we proceed by deriving various universal approximation theorems for the simultaneous control of the operator error and the derivatives with respect to the shape parameters when using reduced basis architectures. 
This universal approximation theory is specifically targeted to the OUU problem, which does not specifically require control of the derivatives with respect to $m$. 
After this we discuss a practical formulation for training shape derivative-informed neural operators.
We also describe an effective strategy for computing reduced bases for the parameter and state spaces, with a discussion on the required adaptations in the context of shape optimization under uncertainty.

\subsection{Reduced Basis Architectures}
In this section we narrow our focus to reduced basis architectures which have been widely deployed in operator learning \cite{hesthaven2018pod,bhattacharya2021pca}, and specifically in the context of derivative-informed operator learning \cite{o2024derivative,luo2025dis} with applications to inverse problems and optimization \cite{cao2024lazydino,cao2025derivative,luo2023efficient}. 
As a reminder, we formulate the operator learning task on the reference domain $\Omega_0$ as in \eqref{eq:res_form_Omega0}; the solutions and the Fr\'echet derivatives can be mapped to the physical coordinates via the pushforward/pullback operators $T_{\diffeo}$ and $T_{\diffeo}^{-1}$.

Since $\cM = \cMref$ and $\cU = \cUref$ are assumed to be separable Hilbert spaces, we can represent the various functions in orthonormal bases $\{\phi_i\}_{i=1}^{\infty}$ of $\cUref$ and $\{\psi_i\}_{i=1}^{\infty}$ of $\cMref$, such that for any $u \in \cUref$ and $m \in \cMref$,
\begin{subequations}
    \begin{align}
        u &=  \sum_{i=1}^\infty \langle u, \phi_i\rangle_\mathcal{U}\phi_i \qquad &\langle \phi_i,\phi_j\rangle_\mathcal{U} = \delta_{ij},\\
        m &=  \sum_{i=1}^\infty \langle m,\psi_i\rangle_\mathcal{M}\psi_i \qquad &\langle \psi_i,\psi_j\rangle_\mathcal{M} = \delta_{ij},
    \end{align}
\end{subequations}
where $\delta_{ij}$ denotes the Kronecker delta.
Reduced basis architectures represent each of the formally infinite-dimensional fields via a truncated basis expansion. 
This corresponds to linear encoding and decoding operations, which we define in \Cref{tab:encoding_decoding}.
A neural network is then used to represent mapping between the reduced basis coefficients of the inputs and outputs.

\begin{table}[h!]
\centering
\begin{tabular}{@{}lll@{}}
\toprule
  & \textbf{Encoding } & \textbf{Decoding }\\
\midrule
$\mathcal{U}$ with basis $\{\phi_i\}_{i}$ & 
$\mathcal{U} \ni u \mapsto \mathbf{\Phi}_{\rru}^*u = \left( \langle u, \phi_i \rangle_{\mathcal{U}} \right)_{i=1}^{\rru} \in \mathbb{R}^{\rru}$& 
$\mathbb{R}^{\rru}\ni c \mapsto \mathbf{\Phi}_{\rru}c = \sum_{i=1}^{\rru} c_i \phi_i \in \mathcal{U}$ \\

$\mathcal{M}$ with basis $\{\psi_i\}_{i}$ & 
$\mathcal{M} \ni m \mapsto \mathbf{\Psi}_{\rrm}^*m = \left( \langle m, \psi_i \rangle_{\mathcal{M}} \right)_{i=1}^{\rrm} \in \mathbb{R}^{\rrm}$ &
$\mathbb{R}^{\rrm} \ni c \mapsto \mathbf{\Psi}_{\rrm}c = \sum_{i=1}^{\rrm} c_i \psi_i \in \mathcal{M}$  \\

$\mathcal{Z} = \bR^{d_z}$ &
Identity & 
Identity \\
\bottomrule
\end{tabular}
\caption{Encoding and decoding operations}
\label{tab:encoding_decoding}
\end{table}

For the output basis $\{\phi_i\}_{i=1}^{\infty}$, a natural choice is the proper orthogonal decomposition (POD) \cite{quarteroni2015reduced, hesthaven2018pod, bhattacharya2021pca}, which is also known as the principal component analysis (PCA).
For the input basis $\{\psi_i \}_{i=1}^{\infty}$, common choices used in operator learning include the POD (PCA), as well as active subspace (AS) bases that are constructed using the derivatives of the underlying mapping \cite{zahm2020gradient, oleary2022dipnet, luo2025dis}.
We elaborate on these approaches in \Cref{sec:practical_dimension_reduction}.
For the shape, we assume that $z$ is a finite parametrization of the diffeomorphism $\diffeo_z$, and we do not perform further dimension reduction.
Nevertheless, one can also consider analogous dimension reduction strategies when the hypothesized space of admissible designs is infinite or high dimensional.

In a reduced basis neural operator (RBNO), the mapping from the input coefficient space to the output coefficient space is typically approximated by a dense neural network, $g_\theta:\bR^{\rrm}\times \bR^{\dz}\rightarrow \bR^{\rru}$.
The network takes inputs  $a_0 = (m_r, z) \in \mathbb{R}^{\rrm + \dz}$ and maps to the output $a_{d_L + 1} \in \bR^{\rru}$, i.e., $a_{d_{L} + 1} = g_{\theta}(a_0)$,
through $d_L$ hidden layers with widths $d_{W,1}, d_{W,2}, \dots, d_{W,d_L}$, where each $a_i$ is 
\[
\begin{aligned}
a_i &= \sigma(W_{i-1} a_{i-1} + b_{i-1}), 
\quad i = 1,2,\dots,d_L, \\
a_{d_L+1} &= W_{d_L} a_{d_L} + b_{d_L}.
\end{aligned}
\]
The collection of network parameters is $\theta = \{(W_i, b_i)\}^{d_L}_{i = 0},$ where the weights are of dimensions 
\[
W_0 \in \mathbb{R}^{d_{W,1}\times (\rrm + \dz)}, 
\quad W_i \in \mathbb{R}^{d_{W, i+1}\times d_{W, i}}, i = 1, 2, \dots, d_{L-1}, \quad W_{d_L} \in \mathbb{R}^{\rru \times d_{W,d_L}},
\]
and the biases are of dimensions
\[b_i \in \mathbb{R}^{d_{W, i+1}}, 
\qquad i=0,
\dots,d_{L-1}, \quad b_{d_L} \in \mathbb{R}^{\rru}.\]
Here, we focus on the case where $g_{\theta}$ uses the Gaussian Error Linear Unit (GELU) as the activation function,
\begin{equation}
    \GELU(x) = x \Phi_{\text{Gauss}}(x),
\end{equation}
where $\Phi_{\text{Gauss}}(x)$ is the cumulative distribution function for the standard Gaussian distribution.
The full neural operator architecture then takes the form
\begin{equation} \label{eq:rbno}
    u_{\theta}(m,z) = \mathbf{\Phi}_{\rru} g_\theta(\mathbf{\Psi}_{\rrm}^*(m - \shiftm), z) + \shiftu,
\end{equation} 
where $\shiftu \in \cUref$ and $\shiftm \in \cMref$ are bias terms that can be selected based on the particular choices of reduced bases.

\subsection{Universal Approximation Results for Reduced Basis Neural Operators}
Under \Cref{assumption:pde_properties}, we can show that the reduced basis architectures \eqref{eq:rbno} can approximate the solution operator $u(m,z)$ and its $z$-derivative to arbitrary accuracy in the following sense:

\begin{theorem}[Universal approximation via multi-input neural operators]\label{theorem:rbno_ua_semibounded}
Suppose $u$ satisfies \Cref{assumption:pde_properties} 
with the Borel probability measure $\mu_m$, compact admissible set $\cZ_{ad}$, and open set $\openset \supset \cZ_{ad}$ for some $p \geq 2$.
Let $\{\phi_i\}_{i=1}^{\infty}$ and $\{\psi_i\}_{i=1}^{\infty}$ 
be orthonormal bases of $\cU$ and $\cM$, respectively. 
Additionally, let $\shiftu \in \cU$ and $\shiftm \in \cM$ be any two vectors.
Then, for any $\epsilon > 0$ 
there exist ranks $\rrm$ and $\rru$ along with a GELU neural network 
$g_{\theta} : \mathbb{R}^{\rrm} \times \bR^{\dz} \rightarrow \mathbb{R}^{\rru}$ 
such that the RBNO $u_{\theta}(m,z) = \mathbf{\Phi}_{\rru} g_{\theta}(\mathbf{\Psi}_{\rrm}^*(m - \shiftm), z) + \shiftu$ satisfies
\begin{equation}
    \sup_{z \in \cZ_{ad}} 
    \bE_{m \sim \mu_m} [ \| u(m,z) - u_{\theta}(m,z) \|_{\cU}^p ]
    + \bE_{m \sim \mu_m} [ \| D_z u(m,z) - D_z u_{\theta}(m,z) \|_{\HS(\cZ,\cU)}^p ]
    \leq \epsilon^p.
\end{equation}
\end{theorem}
The proof of \Cref{theorem:rbno_ua_semibounded} is presented in \Cref{appendix:ua_all}. Supporting lemmas are first developed in \Cref{appendix:ua_nn,appendix:dimension_red_errors}. These are then combined to complete the proof in \Cref{appendix:proof_rbno_ua_semibounded}.

Note that the approximation error of the neural operator is measured in a mixed norm, being taken as an $L^p_{\mu_m}$ norm in the $m$ direction and as a supremum in the $z$ direction over the compact admissible set.
This matches the form of the approximation error required to accurately represent the first-order optimality conditions of the OUU problem (e.g., \Cref{prop:surrogate_gradient_error}).
In particular, \Cref{theorem:rbno_ua_semibounded} implies that the RBNO architecture is capable of approximating the solution operator such that when used as a surrogate for expectation minimization (assuming the true cost functional is strongly convex), its stationary points can be arbitrarily close to the true minimizer. 

\begin{theorem}\label{theorem:rbno_ouu_approx}
    Suppose that $u$ satisfies \Cref{assumption:pde_properties} and $Q$ satisfies \Cref{assumption:qoi_properties}
    with the Borel probability measure $\mu_m$, a convex and compact admissible set $\cZ_{ad}$, and an open set $\openset \supset \cZ_{ad}$ for some $p \geq 2$.
    Consider now the optimization problem with the mean objective,
    \[
        \min_{z \in \cZ_{ad}} \cJ(z) := \bE_{m \sim \mu_m}[ Q(u(m,z),m,z)] + \cP(z),
    \]
    where $Q$, $u$, and $\cP$ are such that $\cJ \in C^1(\openset; \bR)$ is strongly convex. Let $\{ \phi_i \}_{i=1}^{\infty}$ and $\{ \psi_i \}_{i = 1}^{\infty}$ be orthonormal bases of $\cU$ and $\cM$, respectively, and let 
    $\shiftu \in \cU$ and $\shiftm \in \cM$ be any two vectors.
    Then, for any $\epsilon > 0$ 
    there exist ranks $\rrm$ and $\rru$ along with a GELU neural network 
    $g_{\theta} : \mathbb{R}^{\rrm} \times \bR^{\dz} \rightarrow \mathbb{R}^{\rru}$ 
    such that 
    any point $z^{\dagger}$ satisfying the first-order optimality conditions of the surrogate minimization problem 
    \[ 
        \min_{z \in \cZ_{ad}} \cJ_{\theta}(z) := \bE_{m \sim \mu_m}[Q(u_{\theta}(m,z),m,z)] + \cP(z)
    \] 
    with the RBNO $u_{\theta}(m,z) = \mathbf{\Phi}_{\rru} g_{\theta}(\mathbf{\Psi}_{\rrm}^*(m - \shiftm), z) + \shiftu$
    has its optimization error bounded by
    \begin{equation}
        \|z^{\star} - z^{\dagger} \|_{\cZ} \leq \epsilon
    \end{equation}
    and its optimality gap bounded by
    \begin{equation}
        \cJ(z^{\dagger}) - \cJ(z^{\star}) \leq \epsilon^2.
    \end{equation}
\end{theorem}
\Cref{theorem:rbno_ouu_approx} combines the optimization error bounds in \Cref{theorem:optimization_error}, the error expressions in \Cref{prop:surrogate_gradient_error}, with the universal approximation result \Cref{theorem:rbno_ua_semibounded}. We present its proof in \Cref{appendix:uaouu}.
While the strong convexity assumption is strict, we note that it can often be achieved in practice when using strongly convex penalization functionals $\cP$.
Alternatively, one can reinterpret this result in terms of a bounded neighborhood around the true minimizer $z^{\star}$ over which the objective is strongly convex. 
Then, derivative accuracy of the surrogate within such a neighborhood yields bounded optimization errors for any surrogate stationary points that fall within the neighborhood (see, for example, \cite{yao2025derivative}).

\subsection{Shape Derivative-Informed Operator Learning}

Although \Cref{theorem:rbno_ouu_approx} suggests that the approximation error can be controlled uniformly over $z \in \cZ_{ad}$ when $\cZ_{ad}$ is compact, 
in practice, it is typically not possible to ascertain supremum errors bounds via supervised learning. 
Instead, as a practically actionable operator learning formulation, we adopt mean-squared-error (MSE) losses for the output values and the Fr\'echet derivatives of the operator in order to improve accuracy in the derivatives, following \cite{o2024derivative,luo2023efficient}. To this end, we consider a joint distribution $\mu := \mu_m \otimes \mu_z$, where $\mu_z$ is a postulated sampling distribution over $\cZ$ whose support contains $\cZ_{ad}$.
We then solve the following problem in constructing our operator surrogate:
\begin{align} \label{eq:h1_shape_no}
\min_\theta \mathbb{E}_{(m,z) \sim \mu} \bigg[ \|u(m,z) - u_\theta(m,z)\|_{\cU}^2 &
    + \|D_m u(m,z) - D_m u_\theta(m,z)\|^2_{\HS(\cM,\cU)} \nonumber\\
& + \|D_z u(m,z) - D_z u_\theta(m,z)\|^2_{\HS(\cZ, \cU)}\bigg],
\end{align}
We refer to this Sobolev training formulation as the $H^1_{\mu}$ formulation, since the loss function resembles $H^1$ Sobolev norm in the finite dimensional context.
This is in contrast to the conventional $L^2_{\mu}$ training formulation, 
\begin{align} \label{eq:l2_shape_no}
\min_\theta \mathbb{E}_{(m,z) \sim \mu} [ \|u(m,z) - u_\theta(m,z)\|_{\cU}^2],
\end{align}
which does not involve the derivative error.

Note that in \eqref{eq:h1_shape_no} we have additionally included the derivative with respect to the uncertain parameter, $D_m u$, in the training loss. 
This differs from the approach presented in \cite{luo2023efficient}.
Although this is not required for approximation theory presented in \Cref{theorem:optimization_error,theorem:rbno_ouu_approx}, 
it does serve as additional training data that can be used to improve the generalization accuracy. 
Moreover, as we will discuss in \Cref{sec:computational_cost_theoretical}, the computational cost of generating this additional Jacobian data is negligible.

The formulation of \eqref{eq:h1_shape_no} is general and applies to any neural operator architecture. However, training on the full $\HS(\cM,\cU)$ and $\HS(\cZ,\cU)$ norms for the derivatives can lead to prohibitive computational costs, both in generating the derivative (Jacobian) data and in evaluating the loss during training.
Instead, to make $H^1_{\mu}$ learning tractable, we leverage the nature of reduced basis architectures \eqref{eq:rbno}
and project both the state and Jacobian directly in the span reduced bases. 
This leads to the latent space form of the training problem:
\begin{align} \label{eq:h1_shape_rbno}
\min_\theta \mathbb{E}_{(m,z) \sim \mu} \bigg[ 
& \|\mathbf{\Phi}_{\rru}^*(u(m,z) - \shiftu) - g_\theta(\mathbf{\Psi}_{\rrm}^*(m - \shiftm),z)\|_{2}^2  \nonumber 
\\
&+ \|\mathbf{\Phi}_{\rru}^*D_m u(m,z)\mathbf{\Psi}_{\rrm} - D_{1} g_\theta(\mathbf{\Psi}_{\rrm}^*(m-\shiftm),z)\|^2_{F}\nonumber\\
&+ \|\mathbf{\Phi}_{\rru}^*D_z u(m,z) - D_2 g_\theta(\mathbf{\Psi}_{\rrm}^*(m - \shiftm),z)\|^2_{F}\bigg],
\end{align}
where $D_i$ denotes the derivative with respect to the $i^{th}$ argument. Since the output and derivatives of the RBNO lie only in the reduced subspaces, \eqref{eq:h1_shape_rbno} is equivalent to \eqref{eq:h1_shape_no}.
Moreover, as the losses are computed only in the latent spaces of dimensions $\rrm$, $\dz$, and $\rru$, training can be performed efficiently.
This approach was developed in \cite{o2024derivative} and has been utilized in various other works \cite{cao2024lazydino,cao2025derivative,luo2023efficient}.
We note that the $L^2_{\mu}$ loss \eqref{eq:l2_shape_no} can analogously be projected into the latent space, amounting to just the first term in \eqref{eq:h1_shape_rbno}.

\subsection{Practical Dimension Reduction}
\label{sec:practical_dimension_reduction}
Although any choice of orthonormal bases suffices for the purpose of the universal approximation results of \Cref{theorem:rbno_ua_semibounded,theorem:rbno_ouu_approx}, 
an effective dimension reduction strategy allows smaller input and output ranks to be chosen thereby requiring a smaller latent space neural network.
To this end, we adopt the strategy considered for the DIPNet architecture \cite{oleary2022dipnet}.
Namely, we use POD for the output $u$ and AS for the uncertain parameter $m$,
which balances computational practicality with effectiveness of dimension reduction.
Both are constructed using the data-generating distribution $\mu := \mu_m \otimes \mu_z$.

Here, we present the dimension reduction procedures and highlight small modifications to the standard theoretical results that arise due to (1) the need to account for the changing domains in the state variable and (2) only performing dimension reduction on one of the two input variables (namely, the uncertain parameter).
 
\subsubsection{Proper Orthogonal Decomposition for the State Variable}
We compute the POD directly on the reference domain. 
That is, we use the covariance operator for the state variable, 
\begin{equation}
    \cC_{u} = \bE_{(m,z) \sim \mu}[ (u(m,z) - \bar{u}) (u(m,z) - \bar{u})^T],
\end{equation}
where $\bar{u} = \bE_{(m,z) \sim \mu}[u(m,z)]$ is the mean over the training distribution,
and the integrand consists of the outer product
$(u(m,z) - \bar{u}) (u(m,z) - \bar{u})^T v = (u(m,z) - \bar{u}) \langle u(m,z) - \bar{u}, v \rangle_{\cU}$.
The POD basis is then given by the first $\rru$ eigenvectors of $\cC_u$, i.e., $\{ \phi_i \}_{i=1}^{\rru} = \{ \phi_i^{\POD} \}_{i=1}^{\rru}$ satisfying
\begin{equation}
    \cC_u \phi_i^{\POD} = \lambda_i^{\POD} \phi_i^{\POD}, \qquad i \in \bN,
\end{equation}
where $\lambda_i^{\POD}$ are given in descending order.
When using the mean-shifted POD, we adopt $\shiftu = \bar{u}$ as the bias term. This essentially corresponds to the affine decoder $u_r \mapsto \mathbf{\Phi}_{\rru}^* u_r + \bar{u}$.
While not required, mean-shifting of the covariance ensures that the affine subspace $\bar{u} + \mathrm{span}(\phi_1^{\POD}, \dots, \phi_{\rru}^{\POD})$ corresponding to the above decoding operation satisfies any linear, nonparametric constraints (e.g., homogeneous boundary conditions, divergence free conditions) \cite{luo2025dis}.

Standard results for POD suggest that the truncation error associated with the POD basis is bounded by the sum of the trailing eigenvalues in the truncated eigendirections. This is summarized below (see \cite{Lanthaler23,luo2025dis} for a proof).
\begin{proposition}[POD reconstruction error bound]
    \label{prop:pod_error}
    Suppose $\mu$ is a probability measure on $\cM \times \cZ$ and $u \in L^2_{\mu}(\cM \times \cZ; \cU)$.
    Let $\{\phi_i\}_{i=1}^{\rru} = \{ \phi_i^{\POD} \}_{i=1}^{\rru}$ be the first $\rru$ POD basis vectors.
    Then, the reconstruction error for the projection $\projru{\rru} := \sum_{i=1}^{\rru} \phi_i \phi_i^T$
    is bounded by
    \begin{equation}
    \bE_{(m,z)\sim\mu}[\|(I - \projru{\rru}) (u(m,z) - \bar{u})\|_{\cU}^2] \leq \sum_{i > \rru} \lambda_i^{\POD}.
    \end{equation}
\end{proposition}

Since we compute the POD in the reference domain, 
this truncation error bound uses the reference domain function space $\cU$.
For $\cUref = L^2(\Omegaref)$ or $H^1(\Omegaref)$, we can additionally consider the error bound for the truncated state variable pushed forward to the spatial domain $\Omegaz$. 
\begin{proposition}[POD reconstruction error bound in spatial domain norms]
    Consider the setting of \Cref{prop:pod_error}.
    Additionally, assume that there exist constants $\gamma_{0}$ and $\gamma_{1}$ such that 
    \begin{subequations}
    \begin{align}
        \| \diffeo_z\|_{W^{1,\infty}(\Omegaref)} &\leq \gamma_0 \qquad \forall z \in \cZ_{ad}, \\
        \| \diffeo_z^{-1}\|_{W^{1,\infty}(\Omegaz)} &\leq \gamma_1 \qquad \forall z \in \cZ_{ad}.
    \end{align}
    \end{subequations} 
    We have the following:
    \begin{enumerate}
    \item If $\cUref = L^2(\Omegaref)$ and $\cUz = L^2(\Omegaz)$, then 
    \begin{equation}
        \bE_{(m,z)\sim \mu} [ \| T_{\diffeo_z}^{-1} ((I-\projru{\rru})(u(m,z) - \bar{u})) \|_{\cUz}^2 ]
        \leq \gamma_0^{d} \sum_{i > \rru} \lambda_i^{\POD}.
    \end{equation}
    \item If $\cUref = H^1(\Omegaref)$ and $\cUz = H^1(\Omegaz)$, then 
    \begin{equation}
        \bE_{(m,z)\sim \mu} [ \| T_{\diffeo_z}^{-1} ((I-\projru{\rru})(u(m,z) - \bar{u})) \|_{\cUz}^2 ]
        \leq (1 + \gamma_1^2) \gamma_0^{d} \sum_{i > \rru} \lambda_i^{\POD}.
    \end{equation}
    \end{enumerate}
\end{proposition}
\begin{proof}
    Consider first the case of $\cUref = H^1(\Omegaref)$ and $\cUz = H^1(\Omegaz)$.
    Then, by $\Cref{prop:pullback_pushforward_operators}$, we have that 
    \begin{align*}
        \| T_{\diffeo_z}^{-1}((I - \projru{\rru})(u(m,z) - \bar{u}))\|_{\cUz}^2 
        & \leq 
            \left( 1 + \|\diffeo_z^{-1}\|_{W^{1,\infty}(\Omegaz)}^2 \right) \| \diffeo_z \|_{W^{1,\infty}(\Omegaref)}^{d} \|(I - \projru{\rru})(u(m,z) - \bar{u})\|_{\cUref}^2 \\
        & \leq 
        (1 + \gamma_1^2) \gamma_0^{d} \|(I - \projru{\rru})(u(m,z) - \bar{u})\|_{\cUref}^2
    \end{align*}
    for any $z \in \cZ_{ad}$.
    Thus, taking expectations with respect to $(m,z) \sim \mu$ and applying the standard POD bound yields the desired result.
    The proof is analogous for $\cUref = L^2(\Omegaref)$.
\end{proof}

\subsubsection{Active Subspaces for the Uncertain Parameter} 
We then construct our dimension reduction for $m$ based on the truncated representation of the solution operator in the reduced POD coordinates.
That is, we compute the AS for the mapping $\projru{\rru} u(m,z) = \mathbf{\Phi}_{\rru} \mathbf{\Phi}_{\rru}^* u(m,z)$,
for which the derivative covariance is given by
\begin{equation}\label{eq:active_subspace_matrix}
    \cH_m = \bE_{(m,z) \sim \mu}[ D_m u(m,z)^* \mathbf{\Phi}_{\rru} \mathbf{\Phi}_{\rru}^* D_m u(m,z)].
\end{equation}
The subspace basis is then given by the first $\rrm$ eigenvectors of $\cH_m$, i.e., 
\begin{equation}\label{eq:active_subspace_eigenvalue_problem}
    \cH_m \psi_i^{\AS} = \lambda_i^{\AS} \psi_i^{\AS}, \qquad i \in \bN,
\end{equation}
where $\lambda_i^{\AS}$ are again given in descending order.
When then take $\{ \psi_i \}_{i=1}^{\rrm} = \{ \psi_i^{\AS} \}_{i=1}^{\rrm}$ and $\shiftm = \bar{m}$.
AS captures dominant directions in the parameter space that have the greatest impact on the output as measured by the size of the derivatives \cite{zahm2020gradient}. 
Since the dimension reduction explicitly takes into account sensitivity information,
it is typically more effective than POD for input dimension reduction,
which considers only the input variance and is agnostic to the mapping itself.
Here we note that our AS based on \eqref{eq:active_subspace_matrix} differs from the standard AS in that only the $m$-derivatives are considered.

Theoretical properties of AS can be analyzed when the underlying distribution $\mu_m$ satisfies the subspace Poincar\'e inequality. That is, let $\salg(\projrm{\rrm})$ be the sub sigma algebra generated by an orthogonal projection $\projrm{\rrm}$ of rank $r$.
Then, for any $\fun \in L^2_{\mu_m}$, its conditional expectation
$\bE_{\mu_m}[\fun | \salg(\projrm{\rrm})]$
is the best $L^2_{\mu}$ approximation of $\fun$ within the subspace of $\salg_{\projrm{\rrm}}$-measurable functions. 
Such functions can be characterized by the fact that $\fun(m) = g(\projrm{\rrm} m)$ for some measurable function $g$, i.e., these are ridge functions that do not depend on the null space of $\projrm{\rrm}$.
We can quantify the effectiveness of the input dimension reduction strategy by inspecting this $L^2_{\mu_m}$ error between $\fun$ and $\bE_{\mu_m}[\fun | \salg(\projrm{\rrm})]$.
To this end, let $\latent := \bR^{\rru}$ denote the latent space for the output. 
We then have the following definition.
\begin{definition}[Subspace Poincar\'e inequality]
    We say a Borel probability measure $\mu_m$ satisfies the \textit{Subspace Poincar\'e inequality} 
    if for any $\fun \in C^1(\cM;\latent)$ with $\bE_{\mu_m}[\|D_m \fun\|_{\HS(\cM,\latent)}^2] < \infty$,
    \begin{equation}
        \bE_{\mu_m}[ \| \fun - \bE_{\mu}[\fun | \salg(\projrm{\rrm})] \|_{\latent}^2 ] 
        \leq c_{\mu_m} \bE_{\mu_m}[\|D_m \fun (I - \projrm{\rrm}) \|_{\HS(\cM,\latent)}^2],
    \end{equation}
    where $c_{\mu_m} > 0$ is a Poincar\'e constant that is independent of $\fun$ and $\projrm{\rrm}$.
\end{definition}
The canonical example of a distribution satisfying a Subspace Poincar\'e inequality is the Gaussian distribution; we remark on this case below.
Other examples and characterizations are available, though these are typically given for finite-dimensional input spaces.

Now consider any differentiable $\fun : \cM \times \cZ \rightarrow \latent$.
Since we are only interested in dimension reduction of $m$,
we study the conditional expectation of $u$ with respect to $m$ taken pointwise in $z$,
\begin{equation}
    \condexp{\rrm}{\fun}(m,z) 
    := \bE_{\mu_m}[\fun(\cdot,z) | \salg(\projrm{\rrm}) ](m).
\end{equation}
and study the approximation error of $\condexp{r}{\fun}(m,z)$. 
Then, the product measure $\mu = \mu_m \otimes \mu_z$ satisfies a partial subspace Poincar\'e inequality whenever
$\mu_m$ satisfies the subspace Poincar\'e inequality.
In particular, since $z$ is not truncated, the result does not rely on $\mu_z$ satisfying the subspace Poincar\'e inequality, and the resulting Poincar\'e constant depends only on the subspace Poincar\'e constant for $\mu_m$.

\begin{proposition}[Subspace Poincar\'e inequality for partial dimension reduction]\label{prop:mixed_subspace_poincare}
    Suppose $\mu_m$ satisfies a subspace Poincar\'e inequality,
    $\mu_z$ is a Borel probability measure on $\cZ$, and $\fun \in C^1(\cM \times \cZ; \cW)$ satisfies the moment bounds for $\mu = \mu_m \otimes \mu_z$,
    \begin{equation}
        \bE_{\mu} [ \|\fun\|_{\latent}^2] 
        + \bE_{\mu} [ \|D_m \fun\|_{\HS(\cM,\latent)}^2]  
        < \infty.
    \end{equation}
    Then, 
    \begin{equation}
        \bE_{\mu} [\|\fun - \condexp{\rrm}{\fun}\|_{\latent}^2 ] \leq 
        c_{\mu} \bE_{\mu} [\| D_m \fun (I - \projrm{\rrm}) \|_{\HS(\cM, \latent)}^2].
    \end{equation}
\end{proposition}
\begin{proof}
    This follows directly from the subspace Poincar\'e inequality. 
    Recall that for any $z$, we have 
    \begin{equation}
        \bE_{m \sim \mu_m} [\|\fun(m,z) - \condexp{\rrm}{\fun}(m,z) \|_{\latent}^2]  
        \leq 
        c_{\mu} \bE_{m \sim \mu_m} [\|D_m \fun(m,z) (I - \projrm{\rrm}) \|_{\HS(\cM, \latent)}^2]
    \end{equation}
    Taking expectations now with respect to $z \sim \mu_z$ yields the desired result.
\end{proof}

This allows us to quantify the effectiveness of the partial AS given in \eqref{eq:active_subspace_matrix}--\eqref{eq:active_subspace_eigenvalue_problem}.
\begin{proposition}\label{prop:partial_as_result}
    Suppose $\mu_m$ satisfies a subspace Poincar\'e inequality and $\mu_z$ is a Borel probability measure over $\cZ$,
    and $u \in C^1(\cM \times \cZ; \cU)$ and satisfies the moment bounds for $\mu = \mu_m \otimes \mu_z$
    \begin{equation}
        \bE_{\mu} [ \|u\|_{\cU}^2] 
        + \bE_{\mu} [ \|D_m u\|_{\HS(\cM,\cU)}^2]  
        < \infty.
    \end{equation}
    Let $\{ \psi_i \}_{i=1}^{\rrm}$
    = $\{ \psi_i^{\AS} \}_{i=1}^{\rrm}$
    be the AS basis
    and $\projrm{\rrm} := \sum_{i=1}^{\rrm} \psi_i \psi_i^T$. Then, the conditional expectation given the projection $\projrm{\rrm}$
    satisfies
    \begin{equation}
        \bE_{(m,z) \sim \mu}[ \|\mathbf{\Phi}_{\rru}^*( u(m,z) - \condexp{\rrm}{u}(m,z) )\|_{\latent}^2] \leq c_{\mu} \sum_{i > \rrm} \lambda_i^{\AS}.
    \end{equation}
\end{proposition}
\begin{proof}
    We make use of the subspace Poincar\'e inequality from above applied to the mapping $\fun(m,z) = \mathbf{\Phi}_{\rru}^* u(m,z)$. 
    That is, for any $\projrm{\rrm}$, we have 
    \[
     \bE_{(m,z) \sim \mu} [ \|\fun(m,z) - \condexp{\rrm}{\fun}(m,z) \|_{\latent}^2] \leq  c_{\mu_m}
     \bE_{(m,z) \sim \mu} [ \|D_m \fun(m,z) (I - \projrm{\rrm}) \|_{\HS(\cM, \latent)}^2].
    \]
    In particular, the upper bound can also be written as 
    \begin{align*}
     \bE_{\mu} [ \|D_m \fun (I - \projrm{\rrm}) \|_{\HS(\cM, \latent)}^2]
     &= \tr ( (I - \projrm{\rrm}) \bE_{\mu} [ D_m \fun^* D_m\fun] (I- \projrm{\rrm}) ) \\
     &=  \tr ((I - \projrm{\rrm}) \cH_m (I - \projrm{\rrm})).
    \end{align*}
    For $\projrm{\rrm} = \sum_{i = 1}^{\rrm} \psi_i \psi_i^T$, the trace term simplifies to
    \[
    \tr ((I - \projrm{\rrm}) \cH_m (I - \projrm{\rrm})) = \sum_{i > \rrm} \lambda_i^{\AS},
    \]
    which yields the desired result.
\end{proof}

This result shows that there exists a ridge function, namely $\condexp{\rrm}{u}$, whose average errors are bounded by the sum over trailing eigenvalues of $\cH_m$.
While this is not immediately connected with the approximation theory developed in the previous section (which makes use of the ridge function $u_r(m,z) = \projru{\rru} u(\projrm{\rrm} m,z)$ instead of the conditional expectation), it does provide intuitive justification that the subspace represented by $\projrm{\rrm}$ captures the important directions of the input space. 
We do note that under further regularity assumptions on $u$ and $\mu_m$, such as Lipschitz continuity of $u$ and $Du$ and Gaussian $\mu_m$
we can show that $\bar{u}_r$ is also continuously differentiable with Lipschitz continuous derivatives,
and hence develop approximation errors of the RBNO based on a neural network approximation of $\bar{u}_r$,
analogous to the approach used in Section 2.
A similar approach is taken in \cite{luo2023efficient}.

\paragraph{Finite- and infinite-dimensional Gaussian inputs}
In our numerical examples, we will mostly consider discretizations of the Gaussian random fields for the uncertain parameter $\cM$, i.e., $\cM = \bR^{d_m}$ is finite dimensional and $\mu_m = \cN(0, \cC_m)$ is a Gaussian with a non-degenerate covariance $\cC_m$.
In such cases, it is customary to equip $\cM$ with the precision-weighted inner product. To avoid confusion, we will write $\cE$ to denote $\cM$ equipped with this inner product, such that
$\langle m_1, m_2 \rangle_{\cE} = \langle \cC_m^{-1/2} m_1, \cC_m^{-1/2} m_2 \rangle_{\cM}$.
For Gaussian $\mu_m$, conditional expectations with respect to $\salg(\projrm{\rrm})$ have an explicit expression, 
\begin{equation}
    \bE_{\mu_m}[\fun | \salg(\projrm{\rrm})] = \int \fun(\projrm{\rrm} m + (I - \projrm{\rrm})\eta) d\mu_m(\eta).
\end{equation}
Moreover, $\mu_m$ satisfies the subspace Poincar\'e inequality 
\begin{equation}\label{eq:gaussian_subspace_poincare}
    \bE_{\mu_m}[\|\fun - \bE_{\mu_m}[\fun | \salg(\projrm{\rrm})]\|_{\cU}^2] 
    \leq 
    \bE_{\mu_m}[\| D_m \fun (I - \projrm{\rrm})\|_{\HS(\cE, \cU)}^2],
\end{equation}
with $c_{\mu_m} = 1$.
In this case, the adjoint in \eqref{eq:active_subspace_matrix} is given by $D_m u(m,z)^* = (D_m u(m,z))^*_{\cE} = \cC_m (D_m u(m,z))^*_{\cM}$, 
where $(\cdot)^{*}_{\cE}$ refers to the adjoint with respect to the Cameron--Martin space inner product and $(\cdot)^{*}_{\cM}$ denotes the adjoint with respect to the original inner product $\langle \cdot, \cdot \rangle_{\cM}$.
The corresponding eigenvalue problem \eqref{eq:active_subspace_eigenvalue_problem} 
can then be written as a generalized eigenvalue problem, 
\begin{equation}\label{eq:gaussian_evp}
    \widetilde{\cH}_m \psi_i^{\AS} = \lambda_i^{\AS} \cC_m^{-1} \psi_i^{\AS}, \qquad i \in \bN,
\end{equation}
where 
\begin{equation}\label{eq:gaussian_as_matrix}
\widetilde{\cH}_m = \bE_{(m,z)\sim \mu} [ (D_m u(m,z))_{\cM}^* \mathbf{\Phi}_{\rru} \mathbf{\Phi}_{\rru}^* D_m u(m,z))],
\end{equation}
Importantly, the subspace Poincar\'e constant $c_{\mu_m} = 1$, and therefore the error bound in \Cref{prop:partial_as_result}, does not explicitly depend on the discretization dimension $d_m$.

Although $\cM$ and $\cE$ are equivalent in finite dimensions, 
this is no longer the case in infinite dimensions.
Specifically, $\cC_m^{-1/2}$ can be an unbounded operator on $\cM$, and hence the $\cE$ norm is only well-defined on the Cameron--Martin space $\mathrm{Range}(\cC_m^{1/2})$, which does not necessarily coincide with $\cM$. Nevertheless, the reduced bases constructed from \eqref{eq:active_subspace_eigenvalue_problem}
can still be a well-defined form of input dimension reduction for infinite-dimensional Gaussians (see \cite{luo2025dis,cao2024lazydino,cao2025derivative}).
This can be viewed as the limiting case of $d_m \rightarrow \infty$ provided \eqref{eq:gaussian_evp} and \eqref{eq:gaussian_as_matrix} have meaningful limits.

\subsubsection{Sample-Based Approximations for Reduced Basis Computation}

In practice, both the POD and AS need to be computed from a finite sample-based approximation of the expectations.
For the POD, we take the mean estimator
\begin{equation}\label{eq:empirical_mean_u}
    \widehat{u} := \frac{1}{n_{\POD}} \sum_{i=1}^{n_{\POD}} u(m^{(i)}, z^{(i)})
\end{equation}
and the unbiased covariance estimator
\begin{equation}\label{eq:empirical_covariance_u}
    \widehat{\cC}_u := \frac{1}{n_{\POD}-1}  \sum_{i=1}^{n_{\POD}} 
    (u(m^{(i)}, z^{(i)}) - \widehat{u})(u(m^{(i)}, z^{(i)}) - \widehat{u})^*
\end{equation}
where $(m^{(i)}, z^{(i)}) \sim \mu$, $i = 1, \dots, n_{\POD}$ 
are i.i.d.~samples from the input distribution.
Then, the empirical POD is given by the dominant eigenvectors $\widehat{\cC}_{u}$,
\begin{equation}\label{eq:empirical_pod_eigenvalue}
    \widehat{\cC}_u \widehat{\phi}_{j}^{\POD} = \widehat{\lambda}_{j}^{\POD} \widehat{\phi}_j^{\POD}, \qquad j \in \bN
\end{equation}
where we take $\{ \phi_j \}_{i=1}^{\rru} = \{\widehat{\phi}_{j}^{\POD}\}_{j=1}^{\rru}$
and use $\shiftu = \hat{u}$ for the bias in the RBNO.

For the AS, we take the derivative-covariance estimator
\begin{equation}\label{eq:empirical_active_subspace_matrix}
    \widehat{\cH}_m := \frac{1}{n_{\AS}} \sum_{i=1}^{n_{\AS}} 
    D_{m} u(m^{(i)}, z^{(i)})^* 
    {\mathbf{\Phi}}_{\rru} 
    {\mathbf{\Phi}}_{\rru}^* 
    D_{m} u(m^{(i)}, z^{(i)}),
\end{equation}
where $(m^{(i)}, z^{(i)}) \sim \mu$, $i = 1, \dots, n_{\AS}$ 
are again i.i.d.~samples from the input distribution.
The empirical AS is given by the dominant eigenvectors of $\widehat{\cH}_m$, 
\begin{equation}\label{eq:empirical_active_subspace_eigenvalue}
    \widehat{\cH}_m \widehat{\psi}_{j}^{\AS} = \widehat{\lambda}_{j}^{\AS} \widehat{\psi}_j^{\AS}, \qquad j \in \bN
\end{equation}
where we take $\{ \psi_j \}_{j=1}^{\rrm} = \{\widehat{\psi}_{j}^{\AS}\}_{j=1}^{\rrm}$.
Since we assume that the mean $\bar{m}$ of $\mu_m$ is known, we simply take $\shiftm = \bar{m}$ to be the bias.

\section{Implementation Details}\label{sec:implementation_details}
\subsection{Diffeomorphic Deformations via Basis Expansions and Elastic Extensions} \label{sec:deformations}
In this section, we discuss choices for $\diffeo_z$ and the associated distribution $z\sim \mu_z$ such that $\diffeo_z$ remains a diffeomorphism for all choice of $z$.
Specifically, we will consider representations of $z \mapsto \diffeo_z$ that itself utilizes a basis expansion, 
\begin{equation}\label{eq:motion_basis_expansion}
    \diffeo_z(X) = X + \sum_{i=1}^{d_z} z_i \displace_i(X),
\end{equation}
where $\{ \displace_i \}_{i=1}^{d_z}$ are a set of smooth displacement basis functions $\displace_i : \Omegaref \rightarrow \bR^{d}$.
We have that if 
\begin{equation} \label{eq:condition_diffeo}
    \left\| \sum_{i=1}^{d_Z}z_i \nabla_X \displace_i - I \right\|_{L^\infty(\Omegaref)} < 1,
\end{equation}
then $\diffeo_z$ given by \eqref{eq:motion_basis_expansion} is a diffeomorphism.
However, as noted in \cite{BallarinManzoniRozzaEtAl13}, the condition \eqref{eq:condition_diffeo} is often overly conservative and in practice, a larger admissible set $\cZ_{ad}$ can be permitted. 
In our numerical examples, we typically select $\cZ_{ad}$ to be constrained by box bounds $z_{l} \leq z_i \leq z_u$ that ensures the resulting deformations $\diffeo_z$ are diffeomorphic, and subsequently choose $\mu_z$ to be a uniform distribution over $\cZ_{ad}$.

The formulation of \eqref{eq:motion_basis_expansion} accommodates several parametric shape representations commonly used in shape optimization.
This includes domain deformations given by a Fourier basis expansion, which enables the design of boundaries as the graphs of arbitrary functions \cite{nicholson2021joint,CastrillonCandasNobileTempone21}. 
It also includes shape representations via free-form deformations \cite{sederberg_parry_1986}, 
in which $\displace_i$ are typically chosen from an expressive class of locally-supported basis functions such as Bernstein polynomials or B-splines.
In our numerical experiments, we consider shape optimizations using the Fourier representation in a Poisson PDE in \Cref{section:poisson}
and the free-form deformation approach in steady-state Navier--Stokes flow in \Cref{section:ns2d,section:ns3d}.
Details of the shape parameterizations are presented in \Cref{section:poisson_fourier,section:ffd} for the two approaches, respectively.
 
In many other design problems, the shape boundaries are prescribed as graphs of basis expansions without an explicit extension to the interior of the domain. 
This is the case when using domain-specific bases on non-trivial domains, 
such as the Hicks--Henne bases \cite{hicks1978wing, wu2003comparisons, sabater2020gradient} for airfoil design
and Lewis forms \cite{athanassoulis1985extended} for ship hull design.
In such cases, the displacement basis $\displace_{D}^{(i)}$ is prescribed only on the designed boundaries $\Gamma_{D} \subset \partial \Omega_0$. Nevertheless, we can still pose the problem using \eqref{eq:motion_basis_expansion} by extending the displacement basis into the entire domain $\Omegaref$ via elastic extensions, 
a commonly-used method for mesh updating in shape optimization 
\cite{radtke_bletsos_kühl_suchan_rung_düster_welker_2023, schulz_siebenborn_2016} 
and fluid structure interaction \cite{shamanskiy_simeon_2020, stein_tezduyar_benney_2003, diosady_murman}.  
To do so, we obtain $\displace_i$ as the solution of the linear elasticity PDE with 
Dirichlet boundary condition $\displace_i |_{\Gamma_D} = \displace_D^{(i)}$ and $\displace _i|_{\partial \Omega_0 \setminus \Gamma_D} = 0$,
\begin{subequations}
\label{eq:linear_elasticity}
    \begin{alignat}{3}
        -\nabla \cdot \left( 
            a_{\mu} (\nabla u + \nabla u^T) + a_{\lambda} (\nabla \cdot u) I 
            \right) &= 0 && \quad \text{in } \Omega_0\\
        \displace_i &= \displace_D^{(i)} && \quad \text{on } \Gamma_D \\
        \displace_i &= 0 && \quad \text{on } \partial \Omega_0 \setminus \Gamma_D,
    \end{alignat}
\end{subequations}
where $(\nabla u^T)_{ij} = (\nabla u)_{ji}$ corresponds to the matrix transpose of the displacement gradient.
The Lam\'e parameters $a_{\lambda}$ and $a_{\mu}$ can be defined in a spatially-varying manner (e.g., based on distance to boundaries) to improve the conditioning of the deformation gradient $\defgradz(X) = \nabla_X \diffeo_z(X)$.
The displacement bases are then numerically computed by solving \eqref{eq:linear_elasticity} using, say, the finite element method.

In fact, even when using free-form deformations where $\displace_i$ are prescribed on the entire $\Omega_0$, 
we will still use $\displace_i$ to generate a new set of displacement basis functions $\tilde{\displace}_i$ by using $\displace_i |_{\partial \Omega_0}$ as the boundary conditions.
This does not change the motion of the shape boundaries when used in \eqref{eq:motion_basis_expansion}, but can improve the conditioning of the deformation gradient and hence the quality of the deformed mesh in the resulting displacement basis functions.

\begin{remark}[Regularity of elastic extensions]
Theoretically, for a deformation map $\diffeo_z$ to be a diffeomorphism, it is required to be of class $C^1$. However, when such maps are generated by solving the equations of linear elasticity with the finite element method, the resulting weak solutions belong to Sobolev spaces such as $W^{1,\infty}$ as opposed to $C^1$.
This level of smoothness is often sufficient to define pushforward/pullback operations and ensure differentiability in shape optimization (e.g., \cite[Section 11]{ManzoniQuarteroniSalsa21}).
Moreover, invertibility of the numerical deformations can be verified by ensuring the deformation gradient has a positive determinant at the element level.
Nevertheless, we remark that elastic extensions are not strictly required for free-form deformations approaches, and can be omitted if smoothness of the deformation maps is of concern.
\end{remark}

\subsection{Computational Workflow for Operator Learning}
We now summarize the overall workflow for constructing shape-DINO surrogates for PDE-constrained OUU, which largely follows that described in \cite{luo2023efficient}.
We then elaborate on the details of the training data generation and neural operator training. 

\begin{enumerate}
\item \textbf{State data generation:}
Sample parameter–control pairs
$\{(\mref^{(i)}, z^{(i)})\}_{i=1}^{n_s}$ from the prescribed training distribution.
For each sample, solve the high-fidelity PDE to obtain the corresponding state
solutions $\{\uref^{(i)}\}_{i=1}^{n_s}$.

\item \textbf{State-space reduction via POD:}
Compute the sample mean $\widehat{u}$ over a subset of $n_{\POD}$ training samples and perform POD on the centered snapshots $\{u^{(i)} - \widehat{u}\}_{i=1}^{n_{\POD}}$
following \eqref{eq:empirical_mean_u}--\eqref{eq:empirical_pod_eigenvalue}. Then, select a reduced basis $\{\phi_j\}_{j=1}^{\rru}$ of rank $\rru$ to represent the state space.

\item \textbf{Sensitivity data generation:}
For each training sample, compute reduced-output parameter sensitivities 
$\{\mathbf{\Phi}_{\rru}^* D_m u^{(i)}\}_{i=1}^{n_s}$ 
and shape sensitivities
$\{\mathbf{\Phi}_{\rru}^* D_{z} u^{(i)}\}_{i=1}^{n_s}$ 
using the adjoint method.

\item \textbf{Parameter-space reduction via active subspaces:}
Compute an active subspace for the random parameters from 
$n_{\AS}$ reduced-output parameter sensitivities $\{\mathbf{\Phi}_{\rru}^* D_m u^{(i)}\}_{i=1}^{n_{\AS}}$
following \eqref{eq:empirical_active_subspace_matrix}--\eqref{eq:empirical_active_subspace_eigenvalue}
and select a reduced basis $\{\psi_i\}_{i=1}^{\rrm}$ of rank $\rrm$.

\item \textbf{Latent-space encoding:}
Project parameters, states, and Jacobians onto their respective reduced bases,
\[
m_r^{(i)} = \mathbf{\Psi}_{\rrm}^* \, m^{(i)}, \qquad
u_r^{(i)} = \mathbf{\Phi}_{\rru}^* \, (u^{(i)} - \widehat{u}), \qquad
J_{m,r}^{(i)} = \mathbf{\Phi}_{\rru}^* D_m u^{(i)} \mathbf{\Psi}_{\rrm}, \qquad
J_{z,r}^{(i)} = \mathbf{\Phi}_{\rru}^* D_z u^{(i)}.
\]

\item \textbf{Neural operator training.}
Train the neural network $g_\theta$ via empirical risk minimization using the
reduced training dataset
\[
\{(m_r^{(i)}, z^{(i)}, u_r^{(i)}, J_{m,r}^{(i)}, J_{z,r}^{(i)})\}_{i=1}^{n_s},
\]
thereby learning a surrogate model that maps parameters and shapes to reduced states and sensitivities in latent space.
\end{enumerate}

\subsubsection{Details and Computational Costs}\label{sec:computational_cost_theoretical}
\paragraph{State data generation}
State training data are generated by solving the underlying PDE $\Rref(\uref^{(i)}, \mref^{(i)}, z^{(i)}) = 0$ for $\uref^{(i)}$ given the sampled inputs $(\mref^{(i)}, z^{(i)})$.
This is typically the dominant cost in training data generation.
When $\Rref$ corresponds to a linear PDE (i.e., $\Rref$ is linear in $\uref$), the resulting computational cost is solving a linear system of equations that scales with the discretization dimension $\du$. 
On the other hand, when $\Rref$ corresponds to a nonlinear PDE, the solution $\uref^{(i)}$ may require successively solving several linear PDEs, as in Newton or Picard iteration. 

\paragraph{POD computation} 
Once the state data $\{\uref^{(i)}\}_{i=1}^{n_s}$ is generated, 
the POD can be computed by solving the eigenvalue problem using 
a sample-based estimate of the covariance using $n_{\POD}$ of the $n_s$ training samples,
as in \eqref{eq:empirical_mean_u}--\eqref{eq:empirical_pod_eigenvalue}. In particular, the eigenvalue problem \eqref{eq:empirical_pod_eigenvalue} 
can be solved efficiently in a matrix-free manner via iterative methods such as Lanczos \cite{saad2003iterative} or randomized methods \cite{halko2011finding}, and requires no additional PDE solves.

\paragraph{Derivative data computation}
Recall that VJPs can be computed by the adjoint sensitivity equations \eqref{eq:linearized_adjoint}.
Thus, for a given input sample $(m^{(i)}, z^{(i)})$, once the state $u^{(i)}$ has been computed, the cost of generating the VJPs along each of the $\{\phi_j\}_{j=1}^{n_\POD}$ POD basis directions involves solving $\rru$ linear systems of equations. 
Notably, the $\rru$ linear systems share the same operator (left-hand side), 
and only differ in their right-hand sides.
This allows amortization of the solution costs, either by re-using the matrix factorizations when using direct methods (e.g., LU/Cholesky factors), 
or by re-using expensive preconditioners when using iterative methods. 
In our numerical examples, we consider problems for which direct methods are tractable. 
In this case, the cost of generating the $\rru$ VJPs (including both $m$ and $z$ Jacobians) is asymptotically equivalent to that of solving just the first linear system, since cost of applying the matrix factors to the remaining $\rru - 1$ right-hand sides is negligible compared to the cost of performing the matrix factorization.

\paragraph{Active subspace computation}
In our workflow, the AS can be directly computed from the sampled VJPs $\mathbf{\Phi}_{\rru}^* D_m u(m^{(i)},z^{(i)})$.
Again, the eigenvalue problem in \eqref{eq:active_subspace_eigenvalue_problem} can be solved 
through randomized methods
without explicitly forming $\widehat{\cH}_m$,
and requires no additional PDE solves once the sensitivity data has been generated.
{Here, we note when using the conventional $L^2_{\mu}$ loss \eqref{eq:l2_shape_no}, 
Jacobian data only needs to be generated for the first $n_{\AS}$ input samples to compute the active subspace since it is not used elsewhere in training.
However, in this setting, 
it may be preferable in practice to use an input PCA for dimension reduction of $\cM$ to avoid computing derivatives altogether. 
This approach is investigated in \cite{luo2023efficient}, 
but we do not adopt it in order to retain consistency in our comparisons between Shape-DINO and Shape-NO.
}

\subsubsection{Neural Network Training}
For neural network training, with given reduced bases, we can approximate \eqref{eq:h1_shape_rbno} via Monte Carlo using a finite training data corpus 
\[  \left\{(m_r^{(i)}, z^{(i)}, u_r^{(i)}, J_{m,r}^{(i)}, J_{z,r}^{(i)} \right\}_{i=1}^{n_s}. \]
In particular, we will adopt the loss function using a normalized form of \eqref{eq:h1_shape_rbno}
\begin{equation}
\label{eq:finite_h1_rnbo}
    \min_\theta 
    \sum_{i=1}^{n_s} 
    \frac{\| u_r^{(i)} -    g_\theta(m_r^{(i)}, z^{(i)})\|^2_2}{\| u_r^{(i)} \|_2^2}
    + 
    \alpha_{D} \sum_{i=1}^{n_s}  \left(
    + \frac{\| J_{m,r}^{(i)} - D_1 g_\theta(m_r^{(i)}, z^{(i)})\|^2_2}{\| J_{m,r}^{(i)}\|^2_F} 
    + \frac{\| J_{z,r}^{(i)} - D_2 g_\theta(m_r^{(i)}, z^{(i)})\|^2_2}{\| J_{z,r}^{(i)}\|^2_F}  \right).
\end{equation}
Here, we take $\alpha_{D} = 1$ for derivative-informed training (Shape-DINO) and $\alpha_{D} = 0$ for conventional training (Shape-NO).

\section{Numerical Results}\label{sec:numerical_results}
In this section, we evaluate the accuracy and sample efficiency of shape-DINO across three distinct PDE-constrained shape OUU problems:
(i) a flux-tracking problem for a Poisson equation with a shape-parametrized top boundary and an uncertain log-permeability field, (ii)
dissipation reduction for 2D steady-state Navier--Stokes 
flow around an object under uncertain inflow conditions,
and (iii) bending-moment and force reduction for 3D steady-state Navier–Stokes flow around a tower under uncertain inflow conditions. 
In each example, our goal is to assess the effectiveness of Shape-DINOs (and Shape-NOs) through 
(1) their generalization accuracies for output and derivative predictions and
(2) their accuracies in producing solutions to the OUU problems compared to PDE-based solutions using similar computational budgets.
Note that for the 2D problems, we can afford to compare the approximate OUU solutions to an accurate PDE-based reference solution computed using a large sample size.
On the other hand, 
for the 3D case, 
producing an accurate reference solution is computationally prohibitive,
and instead the purpose of the experiment is to demonstrate scalability of our surrogate-based approach. 

For the numerical results, we use FEniCS \cite{logg2012automated} 
to implement finite element discretizations for the PDEs
and
PyTorch \cite{paszke_gross_massa_lerer_bradbury_chanan_killeen_lin_gimelshein_antiga_2019} to construct the neural network approximations. 
Additionally, data generation is implemented using hIPPYlib \cite{10.1145/3428447} and hIPPYflow \cite{olearyroseberry2021hippyflow}, and the OUU problems are formulated and solved using SOUPy \cite{luo_chen_o’leary-roseberry_villa_ghattas_2024}.
CPU-based timings are carried out on a single compute node of the TACC VISTA system using an NVIDIA Grace CPU Superchip,
while neural network training and deployment are performed with an NVIDIA H200 GPU.

\subsection{Latent-space Neural Network Architecture}
Across all experiments, we utilize fully connected feedforward architectures to represent the latent-space neural network $g_\theta$. We use the Adam optimizer \cite{KingmaBa14} to train the neural networks on either the conventional $L^2_{\mu}$ loss (Shape-NO) 
or the derivative-informed $H^1_{\mu}$ loss (Shape-DINO).
In each case, we use a learning-rate scheduler that halves the learning rate at specified epoch milestones. 
The architectural and training hyperparameters, including $d_L$, $d_{W,i}$, 
the initial learning rate, 
the scheduler stages, 
and the number training epochs are all problem-dependent. 
These are summarized in \Cref{tab:nn_hyperparams} and are also discussed within each numerical example.
We note that the same hyperparameters are used for Shape-DINO and Shape-NO for consistency.
\begin{table}[H]
\centering
\resizebox{0.85\textwidth}{!}{%
\begin{tabular}{lccccccc}
\hline
Example 
& $d_{W,i}$ 
& Initial LR
& LR Scheduler,  $\gamma=0.5$
& Epochs 
& $\rru$
& $\rrm$
& $\dz$
\\
\hline
Poisson 
& $[512, 512, 512, 512]$  
& $0.0005$ 
& $\{800,1200,1800\}$ 
& $2000$ 
& $400$
& $256$
& $11$
\\

2D NS  
& $[512, 1024, 512]$ 
& $0.0001$ 
& $\{500, 750\}$ 
& $1000$
& $256$
& $50$
& $162$
\\

3D NS 
& $[512, 512]$ 
& $0.0001$ 
& $\{1000, 1500\}$ 
& $2000$ 
& $200$
& $128$
& $160$
\\
\hline
\end{tabular}}
\caption{Neural network architectures and training hyperparameters for all experiments.
For each experiment, Shape-DINO and Shape-NO use identical settings, except for the loss function. Initial LR denotes the initial learning rate used by the Adam optimizer. LR Scheduler specifies the learning-rate stages, where the set of integers denotes the epoch milestones at which the learning rate is updated by a factor $0.5$.}
\label{tab:nn_hyperparams}
\end{table}

\subsection{Cost accuracy comparison of the neural operator}
We train Shape-DINOs and Shape-NOs for a range of sample sizes. The trained neural operators are then used to solve the OUU problems via SAA of the objective function using a large sample size. We denote the resulting optimal solutions by the neural operators by $z^\dagger_\text{NN}$. For comparison, we also solve the OUU problems directly using the governing PDE and SAA with varying sample sizes, chosen to reflect computational budgets comparable to those required for training the neural operators. The corresponding PDE-based optimal solutions are denoted by $z^\dagger_\text{PDE}$. All optimization problems are solved using gradient-based methods. For the linear elliptic Poisson example with box constraints, we employ the L-BFGS-B algorithm. For the 2D and 3D Navier–Stokes shape OUU problems, we use a trust-region constrained method with box constraints. 

The accuracies of the approximations are assessed relative to the reference solutions obtained by solving the PDE-based OUU problem using a large sample size. The resulting reference optimal solution is denoted by  $z^\star_\text{ref}$. Approximate optimal solutions are compared against $z^\star_\text{ref}$ using the relative optimal cost error defined as 
\begin{equation}
    \frac{|\mathcal{J}(z^\dagger_\text{approx}) - \mathcal{J}(z^\star_\text{ref})|}{|\mathcal{J}(z^\star_\text{ref})|},
\end{equation}
where $\mathcal{J}(z^\dagger_\text{approx})$ and $\mathcal{J}(z^\star_\text{ref})$ are evaluated using the true objective function, estimated via an accurate Monte Carlo estimator with a large number of samples. Here, $z^\dagger_\text{approx}$ denotes either the neural-operator-based solution $z^\dagger_\text{NN}$ or the PDE-based solution $z^{\dagger}_\text{PDE}$, depending on the method under consideration.

Computational efficiency is measured in terms of the total number of state PDE solves, that is, the evaluation of $u(m, z)$ required to obtain the optimal solution. This quantity represents the dominant computational cost of the optimization process. For neural-operator-based optimization, this cost corresponds to the total number of training samples used. For PDE-based optimization, it is given by the product of the total number of objective function evaluations and the SAA sample size. 

In \Cref{section:numerical_timings}, we compare the computational cost of solving for the forward state and its sensitivity with respect to the shape parameter using PDE-based solvers and neural operators. Because these two operations dominate the overall optimization cost, we focus exclusively on them. 
Across our numerical experiments, we show that neural operators can provide orders-of-magnitude speedups over high-fidelity PDE solves, making large-scale OUU problems computationally feasible.

\subsection{Flux-tracking design in a Poisson Problem under Uncertain Permeability Fields}
\label{section:poisson}
\subsubsection{Poisson Problem with Varying Top Boundary}
For our first numerical example, we consider a Poisson PDE on a domain $\Omegaz$ in the form of a rectangle with a top boundary that we are free to design.
We assume that $\Omegaz$ is related to the reference domain $\Omegaref = (0, 4) \times (0, 1) \subset \bR^2$ through the family of diffeomorphisms $\diffeo_z$, the form of which is described below.
This example is analogous to the setup considered in \cite{nicholson2021joint}.
Here, the PDE is given in spatial coordinates as
\begin{equation}
\label{eq:possion}
    \begin{aligned}
        -\nabla_x \cdot( \exp(\mz) \nabla_x \uz) &= f \qquad x \in \Omegaz \\
        \nabla_x \uz \cdot n &= 0 \qquad x \in \Gamma_{L,z} \cup \Gamma_{R,z} \\
        \uz &= 0\qquad x \in \Gamma_{B,z} \\
        \uz &= 1 \qquad x \in \Gamma_{T,z} \\
    \end{aligned}
    \end{equation}
where $x = (x_1, x_2)^T \in \mathbb{R}^2$, and $\Gamma_{L,z}, \Gamma_{R,z}, \Gamma_{B,z}, \Gamma_{T,z}$ are the left, right, bottom, and top boundaries of $\partial \Omegaz$, respectively. 
Here, $f(x) = 0.1\sin(x_1)\cos(x_2)$ is a source term that is prescribed in spatial coordinates. 
The model is parametrized by the log permeability field $\mz$, 
whose uncertainty is characterized by a Gaussian random field on the reference domain, 
$\mref \sim \mu_m = \cN(0, \cC_m)$, 
such that $\mz = T_{\diffeo_z}^{-1} \mref$.
The covariance of the Gaussian is given by a Bi-Laplacian operator on the reference domain 
\begin{equation}\label{eq:bilplacian_covariance}
    \cC_m = (\delta - \gamma \nabla_X \cdot( \Theta \nabla_X))^{-2},
\end{equation}
which corresponds to a Mat\'ern covariance kernel \cite{LindgrenRueLindstroem11, villa2024note}. 
Here we take $\gamma = 1, \delta = 25$, 
and $\Theta$ to be a symmetric positive definite matrix with eigenvalues $\theta_1 = 2$ 
along the eigendirection $[\sin(\pi/4), \cos(\pi/4)]$ and $\theta_2 = 0.5$ along the eigendirection $[\cos(\pi/4), \sin(\pi/4)]$.

\subsubsection{Shape Parametrization using Fourier Basis}
\label{section:poisson_fourier}
We represent the shape of the top boundary using a Fourier expansion with respect to the horizontal coordinates, 
\begin{equation}
    x_2(x_1) = a_0 + \sum_{k=1}^{n_z} a_k \cos(k \pi x_1) + b_k \sin(k \pi x_1),
\end{equation}
with a maximum wavenumber $n_z \pi$.
This can be formulated through deformations of the form \eqref{eq:motion_basis_expansion} with the displacement basis functions
\begin{align*}
    \displace_{2k}(X) &= (0, X_2 \cos(k \pi X_1))^T, \qquad k = 0, \dots, n_z, \\
    \displace_{2k + 1}(X) &= (0, X_2 \sin(k \pi X_1))^T, \qquad k = 1, \dots, n_z.
\end{align*}
and $z=[a_0,a_1,b_1, a_2, b_2, \dots, a_{n_z}, b_{n_z}]^T\in\mathbb{R}^{2{n_z}+1}$.
In particular, we will take $n_{z} = 5$ such that we have 11 displacement basis functions
and take the admissible set to be $\cZ_{ad} = [-0.2, 0.2]^{\dz}$.

\subsubsection{Reference Domain Formulation}
The PDE can be formulated on the reference domain through a change of coordinates in its weak form. On the spatial domain, we have the trial space
\[
    \cUz := 
    \{\uz \in H^1(\Omegaz) : \uz = 0 \text{ on } \Gamma_{B,z} \text{ and } \uz = 1 \text{ on } \Gamma_{T,z} \}
\]
and the test space 
\[
    \cVz := 
    \{\vz \in H^1(\Omegaz) : \vz = 0 \text{ on } \Gamma_{B,z} \text{ and } \vz = 0 \text{ on } \Gamma_{T,z} \}.
\] 
The weak form is then given for any $\vz \in \cVz$ as
\begin{equation}
    \langle R_z(\uz, \mz), \vz \rangle_{\cVz' \times \cVz} = 
    \int_{\Omegaz} \exp(\mz) \nabla_{x} \uz \cdot \nabla_x \vz \; dx 
    - \int_{\Omegaz} f \vz \; dx.
\end{equation}

We then apply the coordinate transformations from \ref{appendix:coordinate_transforms} to the integrals, 
which yields the weak form of the PDE residual expressed in terms of the reference domain quantities.
To this end, we consider reference domain state and test spaces,
\begin{align*}
    \cUref &:= \{u \in H^1(\Omegaref) : u = 0 \text{ on } \Gamma_B \text{ and } u = 1 \text{ on } \Gamma_T \}, \\
    \cVref &:= \{v \in H^1(\Omegaref) : v = 0 \text{ on } \Gamma_B \text{ and } v = 0 \text{ on } \Gamma_T \}.
\end{align*}
where $\Gamma_{T}$ and $\Gamma_{B}$ are the top and bottom boundaries of $\Omegaref$, respectively.
We then have
\begin{align} \label{eq:poisson_weak_reference}
    \langle \Rref(\uref, \mref, z), \vref \rangle_{\cVref' \times \cVref} 
    = \int_{\Omegaref} 
        \exp(\mref) (\defgradz^{-T} \nabla_{X} \uref ) \cdot (\defgradz^{-T} \nabla_{X} \vref) \det \defgradz dX 
        - \int_{\Omegaref} (T_{\diffeo_z} f) \vref \det \defgradz dX.
\end{align}
We observe that the shape variable $z$ appears explicitly in the residual \eqref{eq:poisson_weak_reference} through the deformation gradient and $\defgradz = \nabla_X \diffeo_z$ and the pullback $T_{\diffeo_z}$.

As previously discussed, we treat \eqref{eq:poisson_weak_reference} as the more fundamental definition of the PDE problem for shape optimization. 
This is discretized using a Galerkin finite element method with continuous piecewise-linear elements for bot the state and random parameter space, leading to $\dm = \du = 4,257$. 

\subsubsection{Flux-tracking Shape Optimization under Uncertainty}
For the OUU problem, we consider a tracking-type objective 
for the flux across the bottom boundary. 
That is, given a target flux profile
$\flux_{\target} \in L^2(\Gamma_{B,z})$, 
we seek to minimize the $L^2(\Gamma_{B,z})$ misfit between the 
normal flux $-e^{m} \nabla_x u \cdot n$ and the target, 
\begin{equation}\label{eq:poisson_flux_tracking_omegaz}
    Q_z (u,m) = \frac{1}{2}
        \int_{\Gamma_{B,z}} (- e^{m} \nabla_x u \cdot n - \flux_{\target})^2 ds.
\end{equation}
This transforms to the following integral in reference coordinates
\begin{equation}\label{eq:poisson_flux_tracking_omega0}
    Q (u,m,z) = \frac{1}{2}
        \int_{\Gamma_B} 
        \left(- \exp(m) \defgradz^{-T}
        \nabla_X \uref
        \cdot N
        - 
        \flux_{\target} \right)^2 dS,
\end{equation}
where we have made the simplification using the fact that
$\diffeo_z(X) = X$ on $\Gamma_{B}$ 
and therefore the unit normal vector remains unchanged, i.e.,
\[
n(X) = \frac{\det \defgradz(X) \defgradz^{-T}(X) N(X)}{|\det \defgradz(X) \defgradz^{-T}(X) N(X)|} = N(X).
\]

For the penalization, we consider an $\ell_2$ norm on $z$, i.e.,
\begin{equation}
    \label{eq:L2penalty}
    \cP_{\ell_2}(z; \alpha) := \alpha |\Omegaref| \| z\|_{2}^2,
\end{equation}
where $\alpha > 0$ is a penalization weight 
and the factor $|\Omegaref|$, which denotes the Lebesgue measure of $\Omegaref$, is used for ease of implementation in FEniCS.

We consider two example risk measures for our objectives. In the first example, we consider the risk neutral case using the expectation as the optimization objective, 
\begin{equation}
\min_{z \in \cZ_{ad}} \rho^{\mean}_{m \sim \mu_m}(Q(u(m,z),m,z) + \cP_{\ell_2}(z; \alpha).
\end{equation}
Our second example considers the risk averse case using CVaR as the optimization objective,
\begin{equation}
\min_{z \in \cZ_{ad}} \rho^{\CVaR,\beta}_{m \sim \mu_m} Q(u(m,z),m,z) + \cP_{\ell_2}(z;\alpha),
\end{equation}
with $\beta = 0.95$,
which we solve using the smoothed CVaR formulation \eqref{eq:smoothed_cvar_saa}. 
In both cases, the objective functions include the 
$\ell_2$ penalization on the shape variables,
where we take $\alpha = 0.001$ for the mean objective
and examine both $\alpha = 0.001$ and $\alpha = 0.1$ for the CVaR objective.

\subsubsection{Neural Operator Construction}
The training data for the neural operator is generated by sampling from the input distributions of $m$ and $z$, 
where we use the uncertain parameter distribution $\mu_m = \cN(0, \cC_m)$
and uniform distribution for the shape variables 
$\mu_z = \unif([-0.2, 0.2]^{d_z})$.
We train the neural networks using different training
datasets of sizes 512, 1,024, 2,048, and 4,096, 
both with and without Jacobian information. 
In each case, to generate the reduced basis, 
we use a subset of $n_{\POD} = n_{\AS} = 512$ samples 
from the training set to compute both the POD and the AS,
from which we take $r_u = 400$ POD basis vectors 
and $r_m = 256$ AS basis vectors to define the dimension reduction. 
The latent-space mappings are represented by multi-input dense neural networks with 4 hidden layers of width 512 and the GELU activation function.
These are then trained using Adam for 2,000 epochs with an initial learning rate of $5 \times 10^{-4}$ that is halved after 800 epochs, 1,200 epochs, and 1,600 epochs. 

In Figure~\ref{fig:poisson_nn}, we plot the mean relative errors in $\cU = L^2(\Omegaref)$ 
of the neural operator output 
and the mean relative errors of its reduced Jacobians 
$J_{m,r}(m,z) = \mathbf{\Phi}_{\rru}^* D_m u(m,z) \mathbf{\Psi}_{\rrm}$
and
$J_{z,r}(m,z) = \mathbf{\Phi}_{\rru}^* D_z u(m,z)$.
The errors are computed on a test set of 1,024 samples from the input distribution, and are averaged 10 different runs of the training process with different initializations and training/validation splits. 
We observe that training with Jacobian information yields consistently lower errors in both the solution approximation and its associated sensitivities.

Figure~\ref{fig:poisson_nn_error} illustrates a representative test example. We visualize the state $u$ on the deformed domain, showing (top-left) the reference high-fidelity PDE solution $u_\text{true}$, 
(top-right) the Shape-DINO prediction $u_\text{DINO}$ trained with 512 samples, (bottom-left) the pointwise difference $u_\text{DINO} - u_\text{true}$, and (bottom-right) the parameter $m$ used in the computation. 
Visually, we see that the prediction errors for the state remain small despite the limited training dataset. 

\begin{figure}
\centering
\includegraphics[width=0.95\linewidth]
{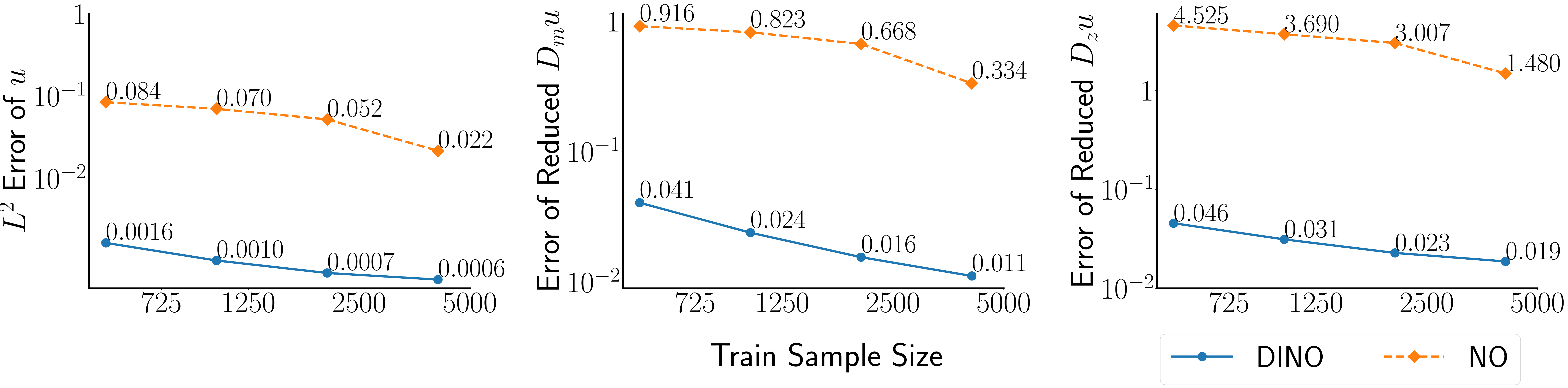}
\caption{Mean relative testing errors for full order state ($L^2(\Omega_0)$) (left) and reduced order Jacobians (middle and right) of Poisson
PDE neural operators trained with (Shape-DINO) and without derivative information (Shape-NO). Training with Jacobian information consistently yields lower errors in both the solution and its sensitivities, demonstrating improved accuracy and robustness of shape-DINO.}
\label{fig:poisson_nn}
\end{figure}

\begin{figure}
\centering
\includegraphics[width=0.67\linewidth]{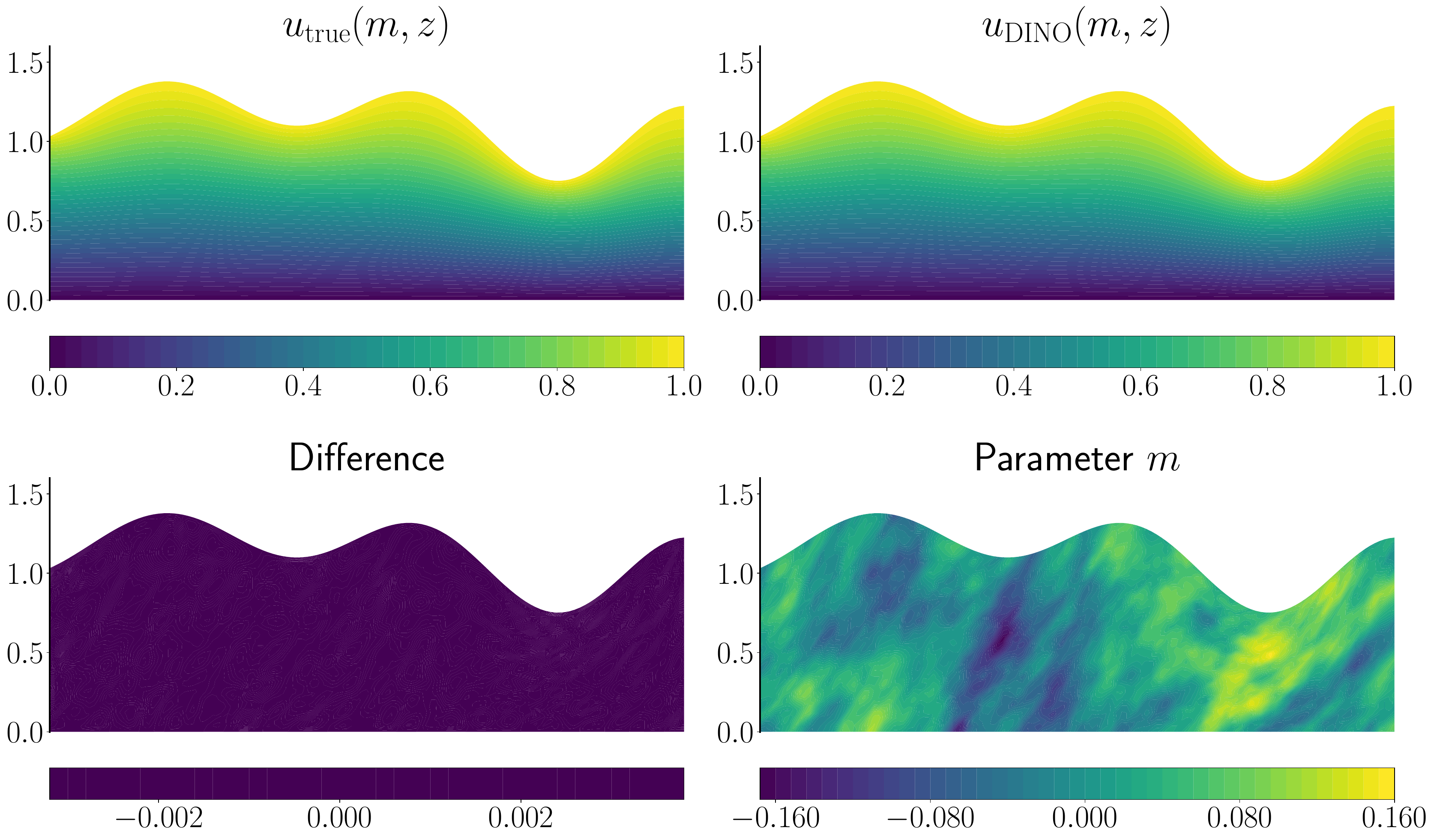}
\captionof{figure}{Example of state $u$ plotted on deformed domain. Top-left: true PDE state $u_\text{true}$. Top-right: predicted state $u_\text{DINO}$ using Shape-DINO, with 512 training samples. Bottom-left: difference: $u_\text{DINO} - u_\text{true}$. Bottom-right: parameter $m$ used for the calculation. Visually, the prediction error is near-zero, confirming the accuracy of Shape-DINO.}
\label{fig:poisson_nn_error}
\end{figure}

\subsubsection{Shape Optimization and Cost-Accuracy Comparison}
To evaluate the predictive and optimization accuracy of Shape-DINO and Shape-NO, 
we compare their optimized objective values against those obtained from PDE-based optimization solutions. For the PDE-based approach, the objective is evaluated via SAA with increasing samples sizes per iteration: 16, 32, 64, 128, 256, and 512, and a reference solution is obtained by a high-fidelity PDE SAA with 16,384 samples.

For the surrogate-based methods (Shape-DINO and Shape-NO), the optimization is performed using a fixed batch size of 2,048 samples per iteration. Since these samples are computed and evaluated using the trained neural surrogates, their computational cost is negligible compared to the PDE solves used for training data generation. Consequently, when comparing computational efficiency across methods, we report surrogates' results based on the training sample size, which directly reflects the number of PDE solves required to construct each surrogate and provides a fair measure of offline computational cost relative to the function-evaluation cost of PDE-based optimization. 

The comparison is summarized in three sets of plots \Cref{fig:poisson_ouu_mean_errors,fig:poisson_ouu_cvar_errors_small_pen,fig:poisson_ouu_cvar_errors_large_pen}, each showing the optimized objective function value with respect to the number of state PDE solves (for PDE-based SAA) or, equivalently, the neural network training size (for surrogate-based methods). Along with the mean relative error lines, we also display the associated $25\%–75\%$ quantile bands. Each set corresponds to a different objective configuration: (i) mean objective with penalty weight $\alpha = 0.001$, (ii) CVaR objective with $\alpha = 0.001$, and (iii) CVaR objective with $\alpha = 0.1$. 

We make several key observations. Across all risk settings, Shape-DINO consistently achieves significantly smaller optimization errors than the PDE-based SAA for the same or even orders-of-magnitude fewer PDE solves. In particular, Shape-DINO attains near-reference accuracy with as few as $512$ training samples, demonstrating significant improvements in sample efficiency. This advantage is most evident for the CVaR objective with a small penalty weight, where accurate risk estimation is highly sample-intensive. On the other hand, for the risk-neutral objective or CVaR with much larger penalty weights, the performance gap between Shape-DINO and PDE is noticeable but less substantial. 
In such cases, the objective functions are either relatively easy to estimate (as in the mean objective) or are dominated by the penalization term (as in the CVaR objective with $\alpha = 0.1$). In comparison, Shape-NO struggles to produce competitive solutions even with $4{,}096$ training samples in all cases, indicating that Jacobian information can be essential for training accurate and practically useful surrogates that amortize the offline training cost and deliver reliable performance across a wide range of use cases. 

\begin{figure}[!htp]
    \centering
    \includegraphics[width=0.98\linewidth]{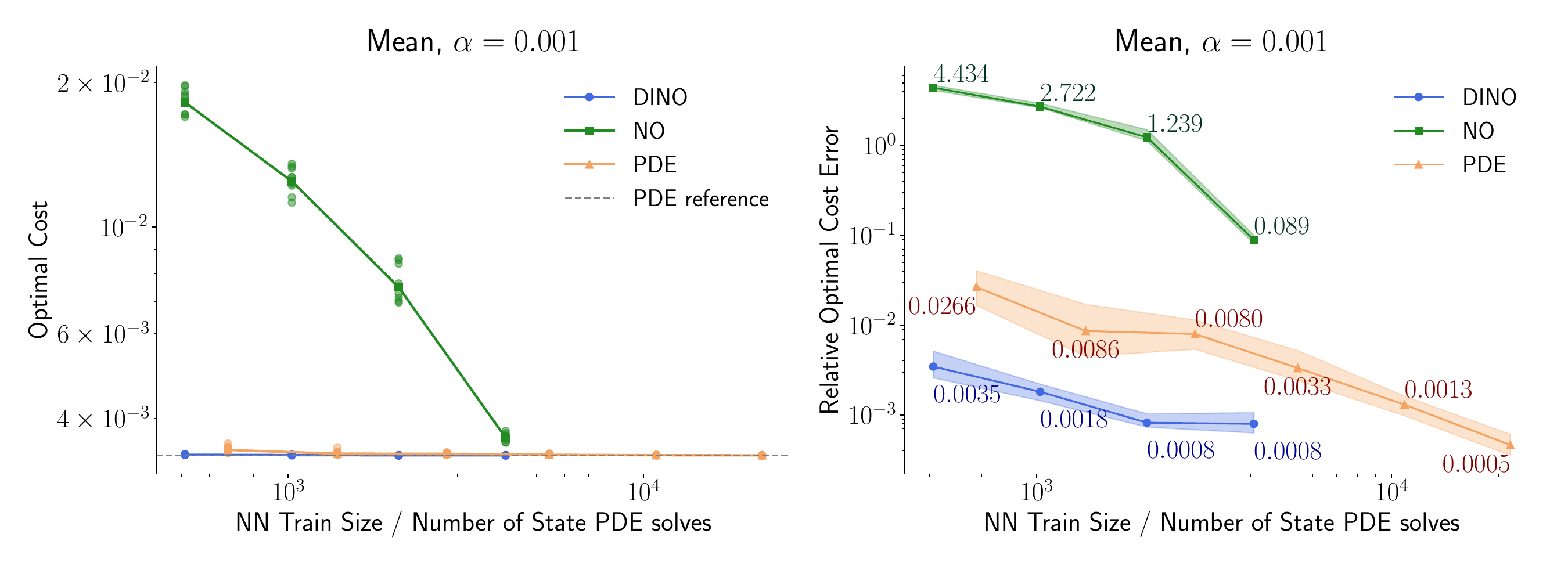}
    \caption{Optimal cost values at optimal solutions (left) and relative error of the objective function (right) versus the number of state PDE solves
required for the optimization, both with mean SAA and penalty weight $\alpha = 0.001$. Results are shown for shape-DINO, shape-NO, and the PDE-based SAA. Solid lines denote the mean relative error, and shaded regions indicate the $25\%–75\% $quantile bands across multiple runs. Even when the objective function is easy to estimate (as in the mean objective), Shape-DINO converges rapidly to the reference solution with substantially fewer PDE solves than the PDE-based method, while Shape-NO exhibits larger errors and slower convergence.}
    \label{fig:poisson_ouu_mean_errors}
\end{figure}

\begin{figure}[!htp]
    \centering
    \includegraphics[width=0.98\linewidth]{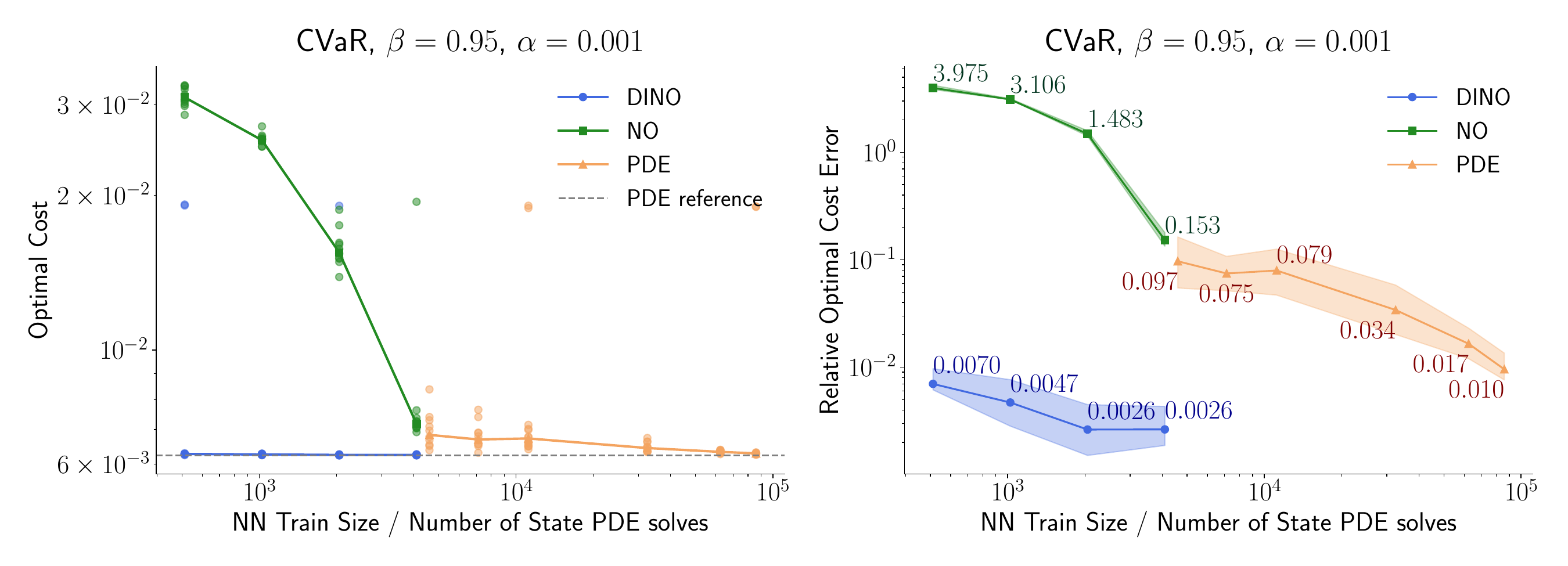}
    \caption{Optimal cost values at optimal solutions (left) and relative error of the objective function (right) versus the number of state PDE solves
required for the optimization, both with quantile $\beta = 0.95$ and penalty weight $\alpha = 0.001$. Solid lines show the mean relative error and shaded regions represent the $25\%–75\%$ quantile bands. In this highly risk-sensitive setting, Shape-DINO achieves near-reference accuracy with orders of magnitude fewer PDE solves than the PDE-based SAA, whereas Shape-NO performs poorly for small and moderate training sizes, highlighting the importance of derivative information for accurate risk estimation.}
    \label{fig:poisson_ouu_cvar_errors_small_pen}
\end{figure}

\begin{figure}[!htp]
    \centering
    \includegraphics[width=0.98\linewidth]{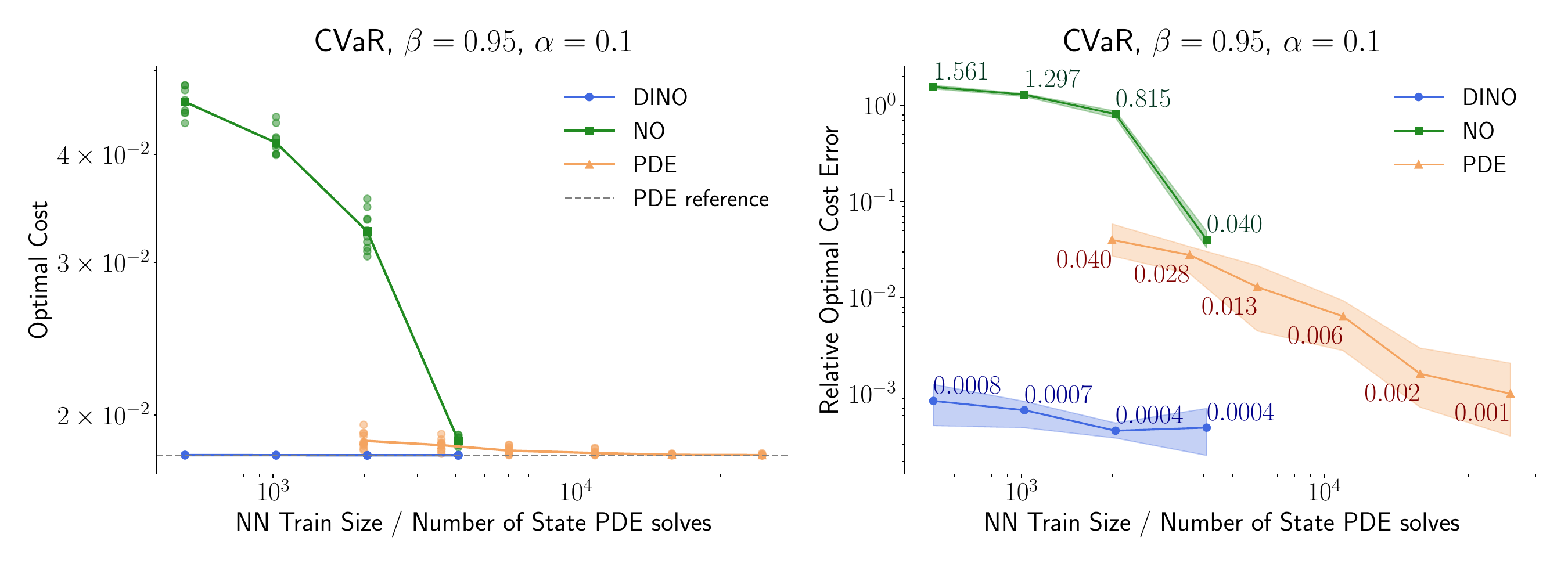}
    \caption{Optimal cost values at optimal solutions (left) and relative error of the objective function (right) versus the number of state PDE solves
required for the optimization, both with quantile $\beta = 0.95$ and penalty weight $\alpha = 0.1$. Solid lines denote the mean relative error and shaded regions show the $25\%–75\%$ quantile bands. As the penalization term becomes dominant, the performance gap between Shape-DINO and the PDE-based method narrows, though Shape-DINO still provides the best trade-off between accuracy and offline computational cost, while Shape-NO remains significantly less reliable.}
    \label{fig:poisson_ouu_cvar_errors_large_pen}
\end{figure}

Figure~\ref{fig:poisson_flux} illustrates representative optimal designs and their corresponding state solutions obtained from CVaR–SAA optimization with quantile $\beta = 0.95$ and penalty weight $\alpha = 0.001$. The top row shows the optimal upper-boundary designs obtained by Shape-DINO ($z^{\dagger}_{\text{DINO}}$) and Shape-NO ($z_{\text{NO}}^{\dagger})$ compared to the PDE reference ($z^{\star}_{\text{ref}}$), together with an associated sample of the PDE state plotted on the deformed domains. The reference solution $z^\star_\text{ref}$ is computed using a high-fidelity PDE-based SAA with 16{,}284 samples. The bottom row reports sample flux profiles evaluated at each optimal design. The black curve denotes the target flux $\flux_\text{target}$, and the colored curves correspond to flux profiles computed using $u(m^{(i)}, z^\dagger_\text{DINO})$, $u(m^{(i)}, z^\dagger_\text{NO})$ and $u(m^{(i)}, z^\star_\text{ref})$ for $256$ independent parameter samples $m^{(i)} \sim \mu_m$. 
We note that $z^\dagger_\text{DINO}$ yields a comparable top-boundary shape to the reference solution $z^\star_{\text{ref}}$ and 
the resulting flux profiles align much more closely with the target flux than those produced by Shape-NO, 
achieving performance very close to that of the reference.

To demonstrate the generalization capabilities of the operator surrogate, 
we solve two additional OUU problems using the 
trained Shape-DINOs and Shape-NOs each for a
different target flux,
focusing on the neural operators trained with 512 samples.
We then compare the flux profile samples corresponding to the optimal shape designs in 
Figures~\ref{fig:bump_flux_poisson}--\ref{fig:tanh_flux_poisson}.
Here, the first case corresponds to a smooth target flux that is consistent the training shape distribution, 
while the second introduces a discontinuous target flux that is not achievable by the shapes in the admissible set. 
Together with Figure~\ref{fig:poisson_flux}, Shape-DINO consistently yields flux responses that more closely match the prescribed target and display reduced variability across parameter realizations across all cases. Notably, even in the out-of-distribution setting, Shape-DINO yields flux profiles that are comparable to those obtained from a large-sample PDE-based SAA, with an offline cost that is orders of magnitude smaller. The improved robustness observed in both in-distribution and out-of-distribution settings highlights the importance of incorporating derivative information when training neural surrogates for such OUU problem. 

\begin{figure}[!htp]
    \centering
    \begin{subfigure}{0.95\linewidth}
        \centering
        \hspace*{0.15cm}
        \includegraphics[width=\linewidth]{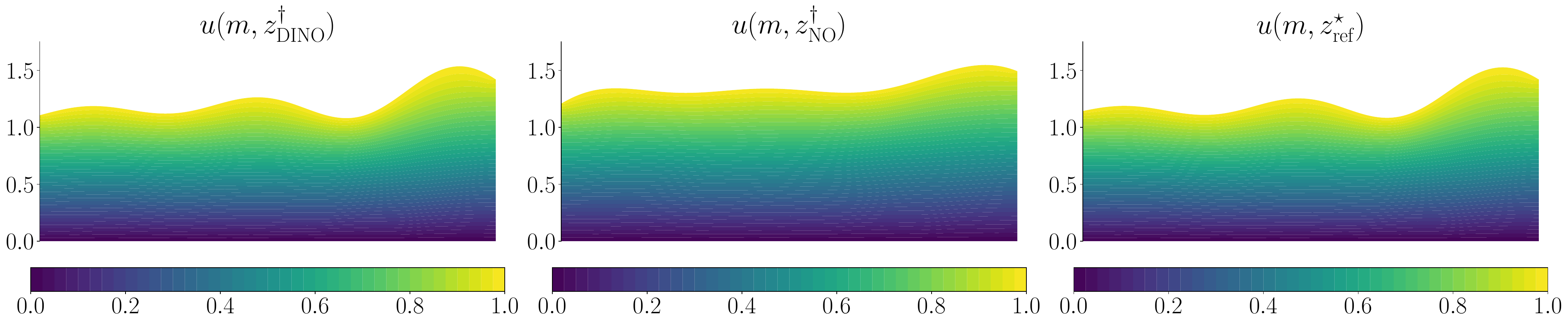}
        \caption{Optimal designs}
    \end{subfigure}

    \vspace{0.5em}

    \begin{subfigure}{0.95\linewidth}
        \centering
        \includegraphics[width=\linewidth]{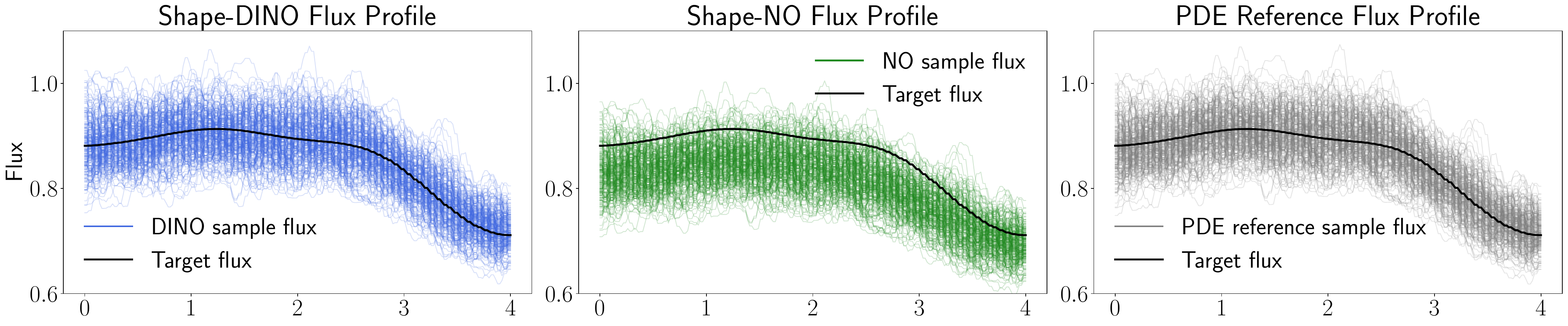}
        \caption{Flux profiles at optimal designs}
    \end{subfigure}

    \caption{Top: optimal upper boundary designs from CVaR-SAA optimization with quantile $\beta = 0.95$ and penalty weight $\alpha = 0.001$. The three subplots show the optimal designs by Shape-DINO, Shape-NO, and PDE reference, respectively, together with their corresponding PDE state solutions $u(m, z)$ on the deformed domain. Here, $z^\star_\text{ref}$ is obtained from a large-sampled PDE-based SAA with 16{,}284 samples. Bottom: sample flux profile at optimal designs. The black line denotes the target flux $\flux_\text{target}$. Colored lines correspond to flux profiles computed with $u(m^{(i)}, z^\dagger_\text{DINO})$, $u(m^{(i)}, z^\dagger_\text{NO})$ and $u(m^{(i)}, z^\star_\text{ref})$ for $256$ independent parameter samples $m^{(i)} \sim \mu_m$. Shape-DINO optimal design produces flux profiles that aligns more closely with the target flux $\sigma_\text{target}$ compared to shape-NO and achieve performance very close to the PDE reference. 
    }
    \label{fig:poisson_flux}
\end{figure}

\begin{figure}[!htp]
    \centering
    \includegraphics[width=0.95\linewidth]{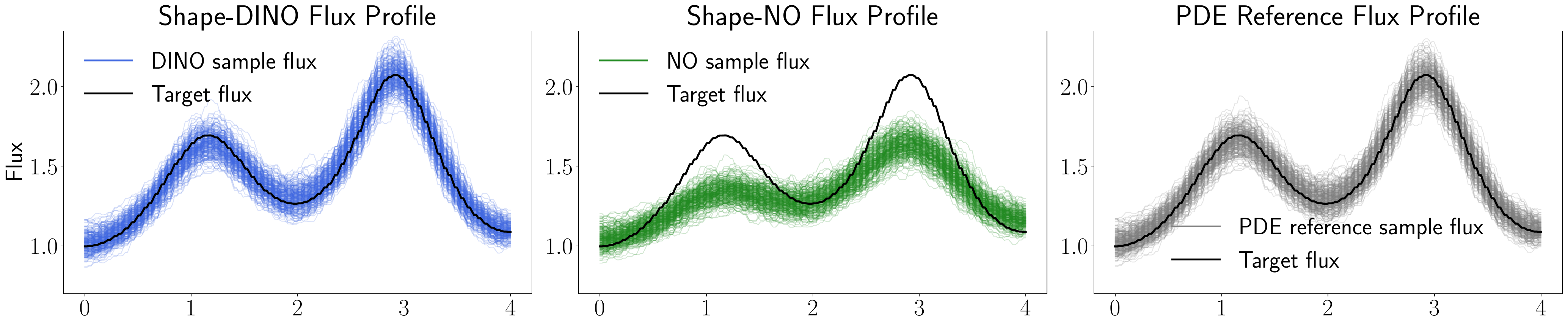}
    \caption{Sample flux profile at the optimal designs obtained from CVaR optimization with $\beta = 0.95$ and $\alpha = 0.001$, using a smooth target flux profile different from Figure~\ref{fig:poisson_flux}. Both DINO and NO are trained with 512 samples, while $z^\star_\text{ref}$ is obtained from a large-sampled PDE-based SAA with 16{,}284 samples. Despite using orders of magnitude fewer PDE solves in the offline stage, shape-DINO produces flux profiles that are nearly indistinguishable from those of the large-sample PDE reference, closely tracking the target flux across parameter realizations and exhibiting reduced variability, whereas shape-NO shows much larger deviations.} \label{fig:bump_flux_poisson}
\end{figure}

\begin{figure}[!htp]
    \centering
    \includegraphics[width=0.95\linewidth]{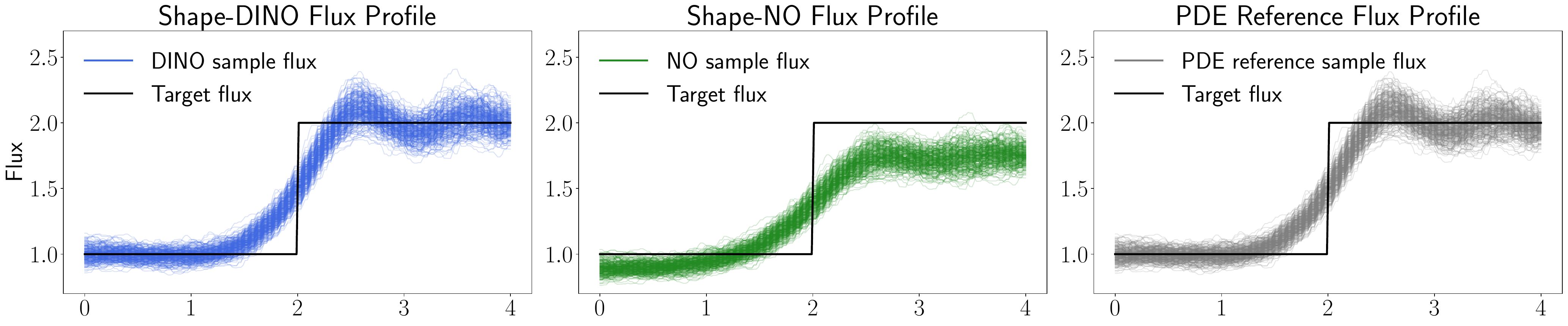}
    \caption{Sample flux profile at the optimal designs obtained from CVaR optimization, with $\beta = 0.95$ and $\alpha = 0.001$, using a discontinuous target flux profile defined by $x_2=1$ for $x_1 \in [0,2]$ and $x_2=2$ for $x_1 \in [2,4]$. Both DINO and NO are trained with 512 samples, while the PDE reference is obtained from a large-sample PDE-based SAA. Despite the target flux lying outside of the distribution used during training, shape-DINO produces flux profiles that better match the target and remain stable across uncertainty, whereas shape-NO produces flux profiles that significantly deviate from the target.}
    \label{fig:tanh_flux_poisson}
\end{figure}

\subsection{Design for Dissipation Reduction in 2D Flow under Multiple Uncertain Inflow Conditions} 
\label{section:ns2d}

\subsubsection{Navier--Stokes Equations with a Variable-Shape Object}
Next we consider a steady incompressible Navier–Stokes problem on the channel domain 
    $\Omega_{\text{channel}} = (0,2) \times (0,1) \subset \mathbb{R}^2$
in which there is an initially circular object 
    $\Omega_{\text{object}} = \{ x \in \bR^{2} : (x_1-0.5)^2 + (x_2 - 0.5)^2 < 0.1^2\}$
with radius $r = 0.1$ 
centered at $(0.5, 0.5)^T$.
The goal is to design the shape of the object to minimize the viscous dissipation through free-form deformations,
where the shape of the object is transformed through a family of diffeomorphic deformations $\diffeo_z$ 
to obtain the design $\Omega_{\object,z} = \diffeo_z(\Omega_{\object})$.
This setting is often used as a benchmark configuration to assess shape-induced drag reduction and dissipation control in confined flows \cite{quarteroni2014numerical}.

Here, we introduce an additional difficulty by
considering two independent inflow scenarios, 
in which the object faces either a horizontally-directed inflow (eastward flow) or a vertically-directed inflow (northward flow),
serving as a prototype for shape design problems under different wind loading directions.
We formulate this problem by considering two sets of PDEs whose domains differ by a reflection of the object along the line $x_1 = x_2$.
For the eastward flow case, we take the object in its default orientation, $\Omega^{\eastward}_{\object} = \Omega_{\object,z}$,
for which the flow domain becomes
\[ 
\Omegaz^{\eastward} = \Omega_{\channel} \setminus \closure{\Omega_{\object}^{\eastward}},
\]
where $\closure{\cdot}$ denotes the closure.
For the northward flow, we reflect the orientation of the object, $\Omega_{\object}^{\northward} = \reflect(\Omega_{\object,z})$
such that the flow domain is 
\[
    \Omegaz^{\northward} = \Omega_{\channel} \setminus \closure{\Omega_{\object}^{\northward}},
\]
where $\reflect(\cdot)$ denotes a reflection about the line $x_1 = x_2$.
Note that this operation amounts to a 90-degree clockwise rotation about the center $(0.5, 0.5)^T$ followed by a reflection along the line $(x_1, 0.5)$. 
This additional reflection will not affect our problem definition, since we will subsequently consider inflow distributions that are invariant to this reflection.

We now formulate the PDEs for the two scenarios on the domains $\Omegaz^{s}, s \in \{\eastward, \northward\}$, 
for which their states in spatial coordinates are given by the velocity-pressure pairs $\uz^{s} = (\vz^{s}, \pz^{s})$.
Using the Cauchy stress tensor,
\begin{equation}\label{eq:ns_stress_tensor}
    \stressz(\uz^s) := 2\nu \strainz(\vz^s) - \pz^s I,
\end{equation}
where
\begin{equation}\label{eq:ns_strain_rate_tensor}
    \strainz(\vz^s) := \symm(\nabla_x \vz^{s}) = \frac{\nabla_x \vz^{s} + (\nabla_x \vz^{s})^T}{2}
\end{equation}
is the strain rate tensor, we can write the governing equations as 
\begin{subequations}
\begin{alignat}{3}
\label{eq:ns_pellet}
    (\vz^s \cdot \nabla_x) \vz^s - \nabla_x \cdot \stressz (\uz^s) &= 0 && \quad \text{in } \Omegaz^{s},\\
    \nabla_x \cdot \vz^{s} &= 0 && \quad \text{in } \Omegaz^{s},\\
    \vz^{s} -  (\vz)_{\text{in}}^{s} &= 0 && \quad \text{on } \Gamma_{\text{in}}, \\ 
    \stressz(\uz^{s}) n &= 0 && \quad \text{on } \Gamma_{\text{out}}, \\
    \vz^{s} &= 0 && \quad \text{on } \Gamma_{\text{no-slip}}, \\ 
    \stressz(\uz^{s}) n \cdot e_t &= 0  && \quad \text{on } \Gamma_{\text{free-slip}}, \\ 
    \vz^{s} \cdot n &= 0  && \quad \text{on } \Gamma_{\text{free-slip}},
\end{alignat}
\end{subequations}
where $e_t$ is the unit tangent vector on the boundary and 
\begin{alignat*}{2}
    \Gamma_{\text{in}} &= \{ x \in \partial \Omega_z^{s} : x_1 = 0 \}
    ,  \qquad & 
    \Gamma_{\text{no-slip}} &= \partial \Omega_{\object}^{s}
    , \\
    \Gamma_{\text{out}} &= \{ x \in \partial \Omega_z^{s} : x_1 = 2 \}
    , \qquad & 
    \Gamma_{\text{free-slip}} &= \{x \in \partial \Omega_z^{s} : x_2 = 0 \text{ or } x_2 = 1 \}
    .
\end{alignat*}
The inflow profile is determined by the uncertain model parameter through
\[
(\vz)_{\text{in}}^{s}(x) = \exp(\mz^{s}(x)) e_{1}, \qquad e_1 = (1, 0)^T.
\]
We define the distributions of $\mz^{s}$ Gaussian random fields on the reference domain 
$\Omegaref := \Omega_{\channel} \setminus \closure{\Omega_{\object}}$, 
such that $\mref^{s} \sim \mu_m^{s} = \cN(0, \cC_{m}^{s})$ with a different Bi-Laplacian covariance operator 
$\cC_{m}^{s}$ 
for each scenario $s \in \{\eastward, \northward\}$.
In particular, $\cC_m^{s}$ is given by \eqref{eq:bilplacian_covariance}, where we take $\Theta = I$ to be the identity matrix, 
and $\delta^{s}, \gamma^{s}$ 
such that the correlation length is $0.25$ for both flow scenarios, 
while the pointwise variance is $0.2$ for $s = \eastward$ and $0.1$ for $s = \northward$.

\subsubsection{Shape Parametrization using Free-Form Deformations}
\label{section:ffd}
The admissible shapes variations are represented using a
free-form deformation approach where $\diffeo_z$ are given by Bernstein polynomials. 
This free-form deformation strategy has been widely applied in PDE-constrained optimization due to its flexibility, smoothness, and relatively low-dimensional parametrization of boundary motion \cite{sederberg_parry_1986, Samareh04, rozza2013free, salmoiraghi_scardigli_telib_rozza_2018}.
We first present the approach based on \cite{Manzoni_Quarteroni_Rozza_2011} and then show how this parametrization can be interpreted in the form of \eqref{eq:motion_basis_expansion}.

To this end, 
we consider the unit square $\Omega_{\control} := (0, 1)^2$ containing the object $\Omega_{\object}$
along with a lattice of control points $\xi_{k,l} = (k/K, l/L)^T$, for $k = 0, \dots, K$ and $l = 0, \dots L$, where $K, L$ are integers.
We then define the Bernstein polynomials, which for a maximum degree $n_p$ are given by 
\begin{equation}
    \label{eq:2d_bernstein}
    b_{j}^{n_p}(t) := \binom {n_p}{j} t^{j} \left(1 - t\right)^{n_p-j}, \quad j = 0, \ldots, n_p.
\end{equation}
In 2D, we can construct the tensor-product Bernstein polynomials from the 1D Bernstein polynomials 
\begin{equation}
    b_{k,l}^{K,L} (t_1, t_2) = b_{k}^{K}(t_1) b_{l}^{L}(t_2), \qquad k = 0, \dots, K, \qquad l = 0, \dots, L.
\end{equation}
Next, we let $\delta_{k,l} = ((\delta_{k,l})_{1}, (\delta_{k,l})_2 )^T\in \bR^{2}$ 
be the displacement of the control point $\xi_{k,l}$, 
such that its displaced location is given by $\widetilde{\xi}_{k,l} = \xi_{k,l} + \delta_{k,l}$.
We let the motion of the control points define the deformation map, such that 
\begin{equation}
    \diffeo(X) := \sum_{k=0}^{K} \sum_{l=0}^{L} b_{k,l}^{K,L}(X_1, X_2) \widetilde{\xi}_{k,l}.
\end{equation}
Evidently, the control point displacements $\delta_{k,l}$ parametrize the deformations. 
To connect this with deformations of the form \eqref{eq:motion_basis_expansion}, we recognize that by the property $\sum_{k=0}^{K} k b_{k}^{K}(t) = K t$ of Bernstein polynomials,
we have that 
\begin{equation}
    \diffeo(X) = \sum_{k=0}^{K} \sum_{l=0}^{L} b_{k,l}^{K,L}(X_1, X_2) (\xi_{k,l} + \delta_{k,l}) = X + \sum_{k=0}^{K} \sum_{l = 0}^{L} b_{k,l}^{K,L} \delta_{k,l}.
\end{equation}
Thus, we can equivalently define $\diffeo_z$ in the form of $\eqref{eq:motion_basis_expansion}$
using the displacement basis functions
\[
    \displace^{1}_{k,l}(X) = b_{k,l}^{K,L}(X_1, X_2) e_1, \quad \displace^{2}_{k,l}(X) = b_{k,l}^{K,L}(X_1, X_2) e_2,
\]
where $e_1 = (1,0)^T$ and $e_2 = (0,1)^T$, along with the shape variables $z_{k,l}^j = (\delta_{k,l})_{j}$ for $k = 0, \dots K$, $l = 0, \dots L$, and $j = 1, 2$.

In our problem, we will adopt a $K = 10$ by $L = 10$ lattice of control points over the region $\Omega_{\control}$.
However, we exclude the first and last control point from our optimization variables to ensure that the outer boundaries do not move,
which also allows us to extend all remaining displacement basis functions by zero beyond $X_1 \geq 1$.
Thus, our resulting admissible deformations consists of the displacement bases $\displace^{j}_{k,l}(X)$ and shape variables $z_{k,l}^{j}$ for $k = 1, \dots K-1$, $l = 1, \dots L-1$, and $j = 1, 2$, giving a total of $\dz = 162$ degrees of freedom.
These displacements are also preprocessed using the linear elasticity equations \eqref{eq:linear_elasticity} to better condition the deformations within the fluid domain.

To generate the geometry variation for the northward configuration, we maintain the same reference domain $\Omegaref$ 
and construct a mirrored displacement basis by swapping the parametric coordinates in the displacement basis functions, such that 
\[
    \displace^{1}_{k,l}(X) = b_{k,l}^{K,L}(X_2, X_1) e_2, \quad \displace^{2}_{k,l}(X) = b_{k,l}^{K,L}(X_2, X_1) e_1,
\]
which effectively reflects the displacement field about the diagonal line $X_1 = X_2$, changing the object orientation relative to the fixed inflow boundary. 
We present example deformations for both inflow cases along with the reference channel domain in Figure~\ref{fig:2d_ns_domain}. 

\begin{figure}[H]
    \centering
    \includegraphics[width=0.99\linewidth]{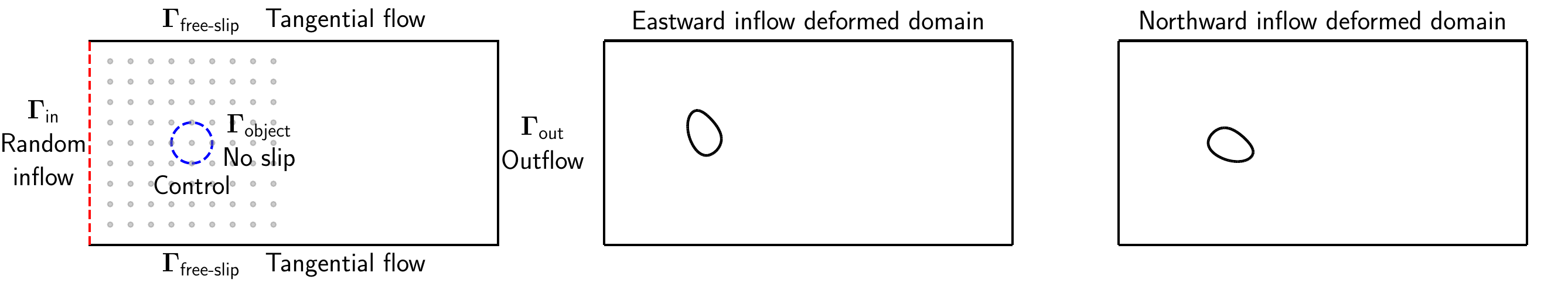}
    \caption{ Panel (left) shows the reference channel domain  $\Omega_{\text{channel}} = (0,2) \times (0,1)$ with the labeled boundaries. The obstacle is a circle with radius $0.1$ centered at $(0.5, 0.5)$. Grey dots are the Bernstein control points $\xi_{k,l}$.
Panels (middle) and (right) show the corresponding deformed domains under eastward and northward displacement bases, respectively, with only domain boundaries displayed for clarity.}
\label{fig:2d_ns_domain}
\end{figure}

\subsubsection{Reference Domain Formulation}
We formulate the Navier--Stokes equations \eqref{eq:ns_pellet} in weak form, where we use Nitsche's method to weakly impose the inflow boundary conditions involving the model parameters. 
In the spatial domain, we define the function spaces for the velocity and pressure as 
\begin{align*}
    \cUvz := \{ v \in H^1(\Omegaz^{s}) : v = 0 \text{ on } \Gamma_{\text{no-slip}}, \, v \cdot n = 0 \text{ on } \Gamma_{\text{free-slip}} \}, \qquad 
    \cUpz := L^2(\Omegaz^{s}),
\end{align*}
such that the trial space is given by 
$\cUuz := \cUvz \times \cUpz$
and the test space is given by $\cVuz := \cUuz$.
For any $\utestz^s = (\wz^s, \qz^s) \in \cVuz$, 
we have 
\begin{equation}
    r_z^{s}(\uz^{s} , \mz^{s}, \utestz^{s})
        := \langle \Rz(\uz^{s}, \mz^{s}), \utestz^{s} \rangle_{\cVuz' \times \cVuz}
        = r_{\flow}^{s}(\uz^{s}, \mz^{s}, \utestz)
        + r_{\nitsche}^{s}(\uz^{s}, \mz^{s}, \utestz)
\end{equation}
where 
\begin{align*}
    r_{\flow}^{s}(\uz^{s}, \mz^{s}, \utestz^{s}) &= 
    \int_{\Omegaz^s} \left( (\vz^{s} \cdot \nabla_x) \vz^{s} \right) \cdot \wz^{s} dx 
    + \int_{\Omegaz^s} 2 \nu \; \strainz(\vz^{s}) : \strainz(\wz^{s}) dx  \\
    & \qquad - \int_{\Omegaz^s} \pz^{s} \left( \nabla_x \cdot \wz^{s} \right) dx
    + \int_{\Omegaz^s} \left( \nabla_x \cdot \vz^{s} \right) \qz^{s} dx
\end{align*}
and 
\begin{align*}
    r_{\nitsche}^{s}(\uz^{s}, \mz^{s}, \utestz^{s}) 
    &= 
    -\int_{\Gammain} \wz^{s} \cdot (\stressz(\uz^{s}) n) ds 
    -  \int_{\Gammain}  ( \vz^{s} - (\vz)^{s}_{\inflow} ) \cdot (2 \nu \strainz(\wz^{s}) n + \qz^{s} n) ds  \\
    & \qquad - \int_{\Gammain} \wz^{s} \cdot (\vz^{s} - (\vz)_{\inflow}^{s}) (\vz^{s} \cdot n) ds 
    + \int_{\Gammain} c_d \nu (\vz^{s} - (\vz)_{\inflow}^{s}) \cdot \wz^{s} ds.
\end{align*}
Here, $c_d$ is the constant that is chosen to ensure the correct convergence rate of the Dirichlet boundary condition. 
We discuss its choice below.

We then transform the weak forms above into the reference domain forms. 
Since the reference domain is the same for the two flow scenarios, we can use the same trial and test spaces for both, i.e.,
$\cUuref = \cVuref = \cUvref \times \cUpref$, where
\begin{align*}
    \cUvref := \{ v \in H^1(\Omegaref) : v = 0 \text{ on } \partial \Omega_{\object}, \, v \cdot n = 0 \text{ on } \Gamma_{\text{free-slip}}\}, 
    \qquad 
    \cUpref := L^2(\Omegaref)
\end{align*}
Using the variable transformations for the strain rate and stress tensors, 
\begin{align}  \strain(\vref^s) &:= \symm(\nabla_X \vref^s F^{-1}) ,\\\stress(\uref^s) &:= 2 \nu \strain(\vref^s) - \pref^s I.
\end{align}
we have that for any $\utest \in \cVref$,
\begin{equation}\label{eq:ns_ref_2d}
    r(\uref^s, \mref^s, z, \utest)
    := 
    \langle \Rref (\uref^s, \mref^s, z), \utest \rangle_{\cV' \times \cV}
    = r_{\flow}(\uref^s, \mref^s, z, \utest)
    + r_{\nitsche}(\uref^s, \mref^s, z, \utest),
\end{equation}
where
\begin{align}
    r_{\flow}(\uref^s, \mref^s, z, \utest)
    &=
    \int_{\Omegaref} \left( \nabla_X \vref^s F^{-1} \vref^s \right) \cdot w^{s} \det F dX 
    + \int_{\Omegaref} 2 \nu \strain (\vref^s) : \strain(w^s) \det F dX  \nonumber \\
    & \qquad - \int_{\Omegaref} p^s \tr(\nabla_X w^s F^{-1}) \det F dX 
    + \int_{\Omegaref} (\tr(\nabla_X v^s F^{-1}) q^s \det F dX,
\end{align}
and 
\begin{align}
    & r_{\nitsche}(\uref^s, \mref^s, z, \utest) 
    = -\int_{\Gammain} w^{s} \cdot \stress(\uref^s) N dS
    - \int_{\Gammain} (v^s - v^s_{\inflow}) \cdot (2 \nu \strain(w^s) + q^s I ) 
    N dS  \nonumber \\
    & \qquad 
    - \int_{\Gammain} w^s \cdot (v^s - v^s_{\inflow}) 
    (v^s \cdot N) dS 
    + \int_{\Gammain} c_d \nu (v^s - v^s_{\inflow}) \cdot w^{s} 
    dS
\end{align}
Note that we have again made simplifications based on the fact that the left boundary remains fixed under the free-form deformations such that the unit normal does not change. 
Moreover, under this representation, the inflow profile is simply given by $v^s_{\inflow}(X) = \exp(m^{s}(X)) e_1$.

The PDEs are discretized in the reference domain using a mixed Galerkin formulation on a triangular mesh. We employ Taylor--Hood elements for the state $u = (v, p)$, consisting of continuous piecewise-quadratic elements for the velocity $v$ and continuous piecewise-linear for the pressure $p$, leading to a state dimension $\du = 86{,}205$. Continuous piecewise-quadratic elements are also used for the random parameter field, yielding 
$\dm = 201$. For the penalty term in the Nitsche method, we choose $c_d = 10^{5}/h_e$, 
where $h_e$ is the element size corresponding to each triangular element of the reference domain.
The control variable is represented in a finite-dimensional Euclidean space with dimension $d_{\mathcal Z} = 162$.

\subsubsection{Dissipation-reducing 2D Shape Design under Uncertain Inflows}

Our optimization problem seeks to simultaneously minimize the viscous dissipation rate induced by each of the two distinct inflow configurations $s \in \{ \eastward, \northward \}$, 
a measure that is physically related to the drag forces acting on the object. We write the dissipation QoI in spatial coordinates as 
\begin{equation}
    Q_{z}^{s}(\uz^{s})
    :=
    \int_{\Omegaz^{s}}
    2 \nu 
    \strainz(\vz^{s}) : \strainz(\vz^{s})
    \, dx.
\end{equation}
We can pose this in the reference domain by the change of variables for the strain tensor,
\begin{equation}
    Q(\uref^s, \mref^s,  z)
   = 
   \int_{\Omegaref} 2 \nu \strain(\vref^s) : \strain (\vref^s) 
   \det F dX.
\end{equation}

To quantify performance under uncertainty, we define a weighted entropic risk objective that combines the two flow scenarios in a weighted sum,
\begin{equation}
\cJ^{\mathrm{weighted}}(z) := \cJ^{\eastward}(z) + 0.3 \cJ^{\northward}(z) + \cP_{\ell_2}(z ;\alpha),
\end{equation}
where we take
\begin{equation}
    \cJ^{s}(z)  := \rho^{\entropic,\beta}_{\mref^{s} \sim \mu_m^{s}}(Q(\uref^{s}(\mref^{s}, z), \mref^{s}, z)), 
\end{equation}
to be the entropic risk measure for each inflow scenario $s \in \{\eastward, \northward\} $
and $\cP_{\ell_2}$ to be the $\ell_2$ penalization functional \eqref{eq:L2penalty}.
Here, we use $\alpha = 0.001$ for the penalty weight
and $\beta = 1$ for a moderate level of risk aversion.

We then consider the shape optimization problem subject to the governing PDE constraints and box constraints on the shape variables $z$ \eqref{eq:L2penalty}. In addition, we impose a volume constraint to prevent unrealistic shrinkage or growth of the object by requiring the volume to remain between $30\%$ and $120\%$ of its initial value $V_0 := |\Omega_{\object}| = \pi \cdot 0.1^2$.
That is, we solve the minimization problem
\begin{equation}
    \min_{z \in \cZ_{ad}} \cJ^{\mathrm{weighted}}(z),
    \qquad \cZ_{ad} =  \{ z \in [-1, 1]^{\dz} : \vol(z) \in [0.3 V_0, 1.2 V_0]\},
\end{equation}
where the volume functional is implemented as 
$
    \vol(z) = | \Omega_{\channel} | -  \int_{\Omegaref} \det F  dX.
$

\subsubsection{Neural Operator Construction}
For this problem, we construct two separate neural operators $u_{\theta}^{s}(m^{s}, z)$ to approximate the solution operators $u^{s}(m^{s}, z)$ 
for each of the two flow scenarios $s \in \{\eastward, \northward\}$.
Training data for each neural operator is generated by sampling the input distributions $m^{s} \sim \mu_m^{s}$ and $z \sim \mu_z = \unif([-1,1]^{\dz})$.
We treat the training problem for each flow scenario completely separate from the other.
For each flow scenario, we train neural networks using different training dataset of sizes 512, 1,024, 2,048, 4,096, both with and without Jacobian information.
In each case, we use a subset of $n_{\POD} = n_{\AS} = 512$ samples to compute the state POD basis of rank $\rru = 256$ and the parameter AS basis of rank $\rrm = 50$.
We represent latent space mapping by multi-input dense neural networks with 3 hidden layers of widths 512, 1,024, 512, each using the GELU activation function.
The neural networks are then trained using Adam for 1,000 epochs with an initial learning rate of $1 \times 10^{-4}$ that is halved after 500 epochs and 750 epochs. 

We present the generalization errors obtained by the neural networks for the eastward flow scenario in \Cref{fig:left_inflow_nn} and for the northward flow scenario in \Cref{fig:top_inflow_nn}, as a function of the number of training samples.
The reported test errors represent the average performance across these $10$ runs.
Similar to the Poisson problem, we again observe that the neural operators trained with the derivative information yield significantly lower generalization errors in both the outputs and the derivatives. This is true for both flow scenarios.
\begin{figure}[H]
\includegraphics[width=0.95
\linewidth]{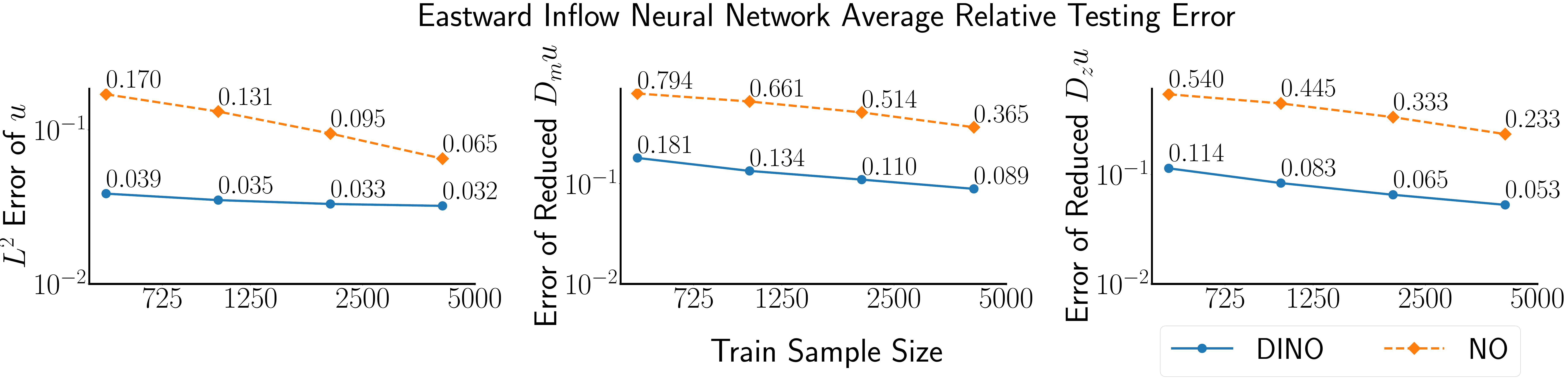}
    \caption{State (left) and Jacobians (middle and right) testing errors for eastward inflow case 2D Navier-Stokes neural operators with (Shape-DINO) and without (Shape-NO) Jacobian training. Training sample sizes are 512, 1{,}024, 2{,}048, 4{,}096 and test sample size is 1{,}024. Shape-DINO yields
significantly lower generalization errors in both the outputs and the derivatives compared to Shape-NO.}
    \label{fig:left_inflow_nn}
\end{figure}

\begin{figure}[H]
\includegraphics[width=0.95\linewidth]{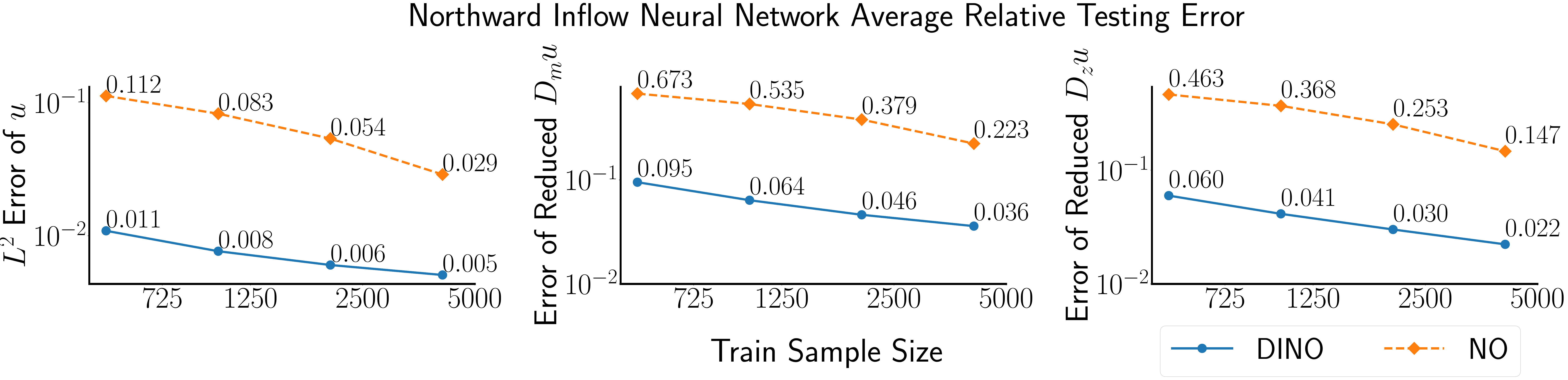}
    \caption{State (left) and Jacobians (middle and right) testing errors for northward inflow case 2D Navier-Stokes neural operators with (Shape-DINO) and without (Shape-NO) Jacobian training. Training sample sizes are 512, 1024, 2048, 4096 and test sample size is 1024. Similar to the eastward inflow case, Shape-DINO again yields markedly lower errors in both the predicted state and its derivatives compared to Shape-NO.}
    \label{fig:top_inflow_nn}
\end{figure}

\subsubsection{Shape Optimization and Cost-Accuracy Comparison}
We subsequent deploy the trained neural operators to solve the dissipation minimization problem.
Figure~\ref{fig:channel_pellet_optimal} compares representative flow fields
before and after optimization for the two inflow configurations.  Figures~\ref{fig:casea} and \ref{fig:casec} display the velocity
fields around the initial object geometry for the left and top inflow scenarios, respectively, obtained from high-fidelity PDE solves. In contrast,
Figures~\ref{fig:caseb} and~\ref{fig:cased} show the corresponding optimized
flow fields evaluated at the shape DINO-derived optimal shape $z^\dagger_{\text{DINO}}$,
using the shape DINO model trained with $1024$ samples. The optimized geometry yields a markedly reduced recirculation behind the object, 
which in turn lowers the overall viscous
dissipation rate compared to the initial shape.
Moreover object is predominantly streamlined in the eastward direction, yet is not perfectly symmetric since it additionally needs to account for flow in the northward direction at a smaller weighting.

\begin{figure}[H]
\centering
\begin{tabular}{cc}
\subcaptionbox{\label{fig:casea}}{\includegraphics[width=0.45\linewidth]{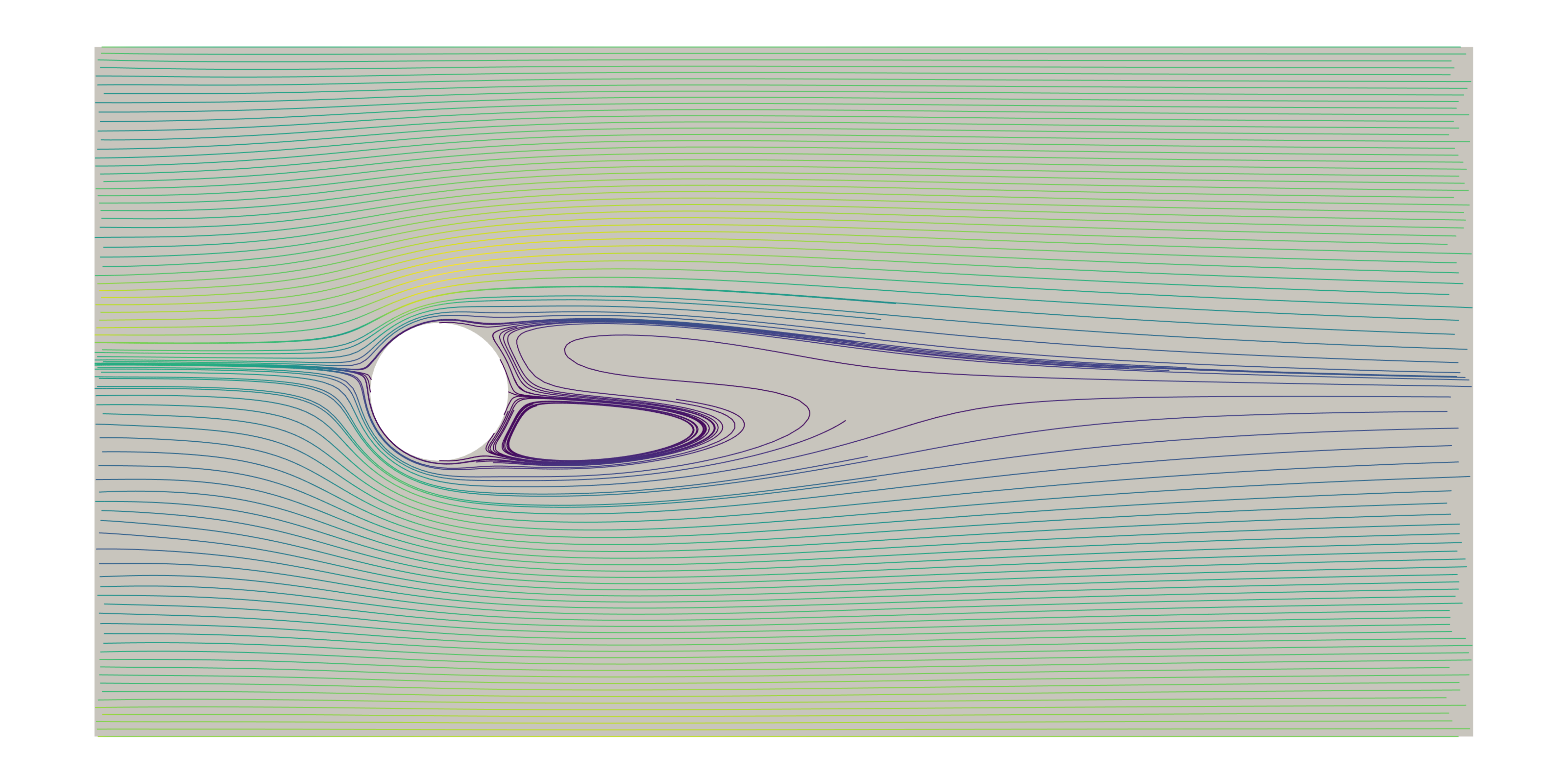}} &
\subcaptionbox{\label{fig:caseb}}{\includegraphics[width=0.45\linewidth]{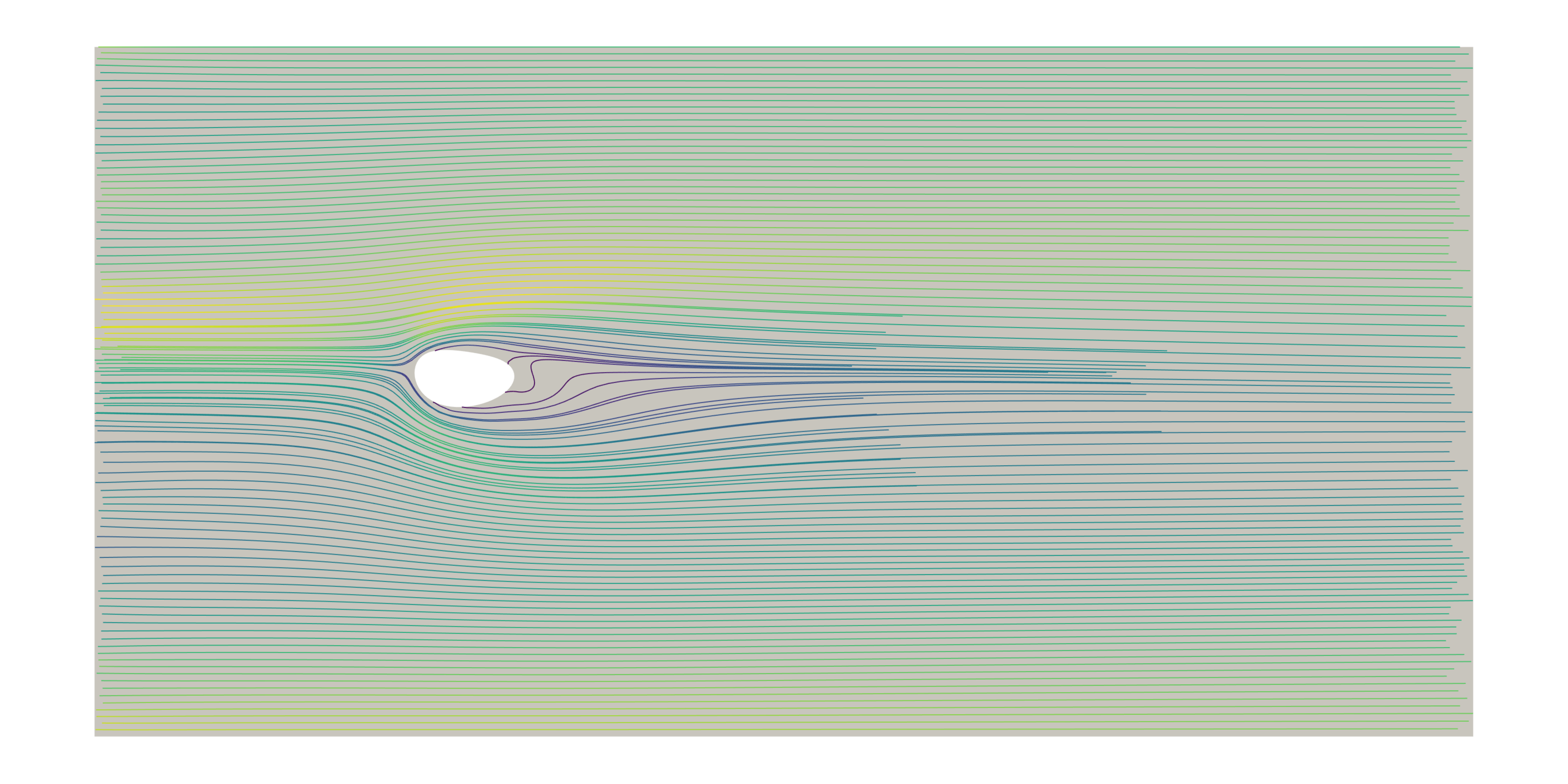}} \\
\subcaptionbox{\label{fig:casec}}{\includegraphics[width=0.45\linewidth]{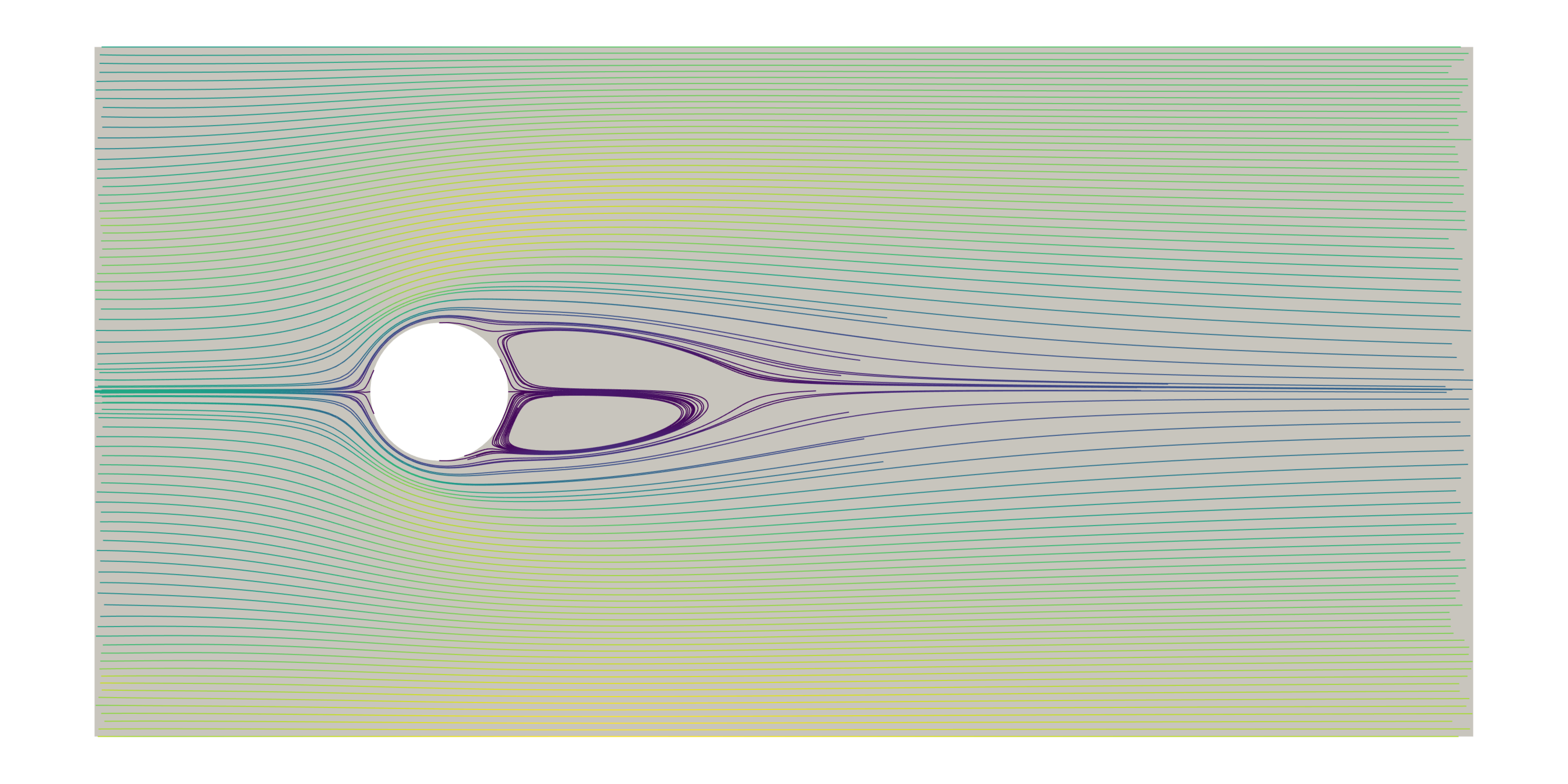}} &
\subcaptionbox{\label{fig:cased}}{\includegraphics[width=0.45\linewidth]{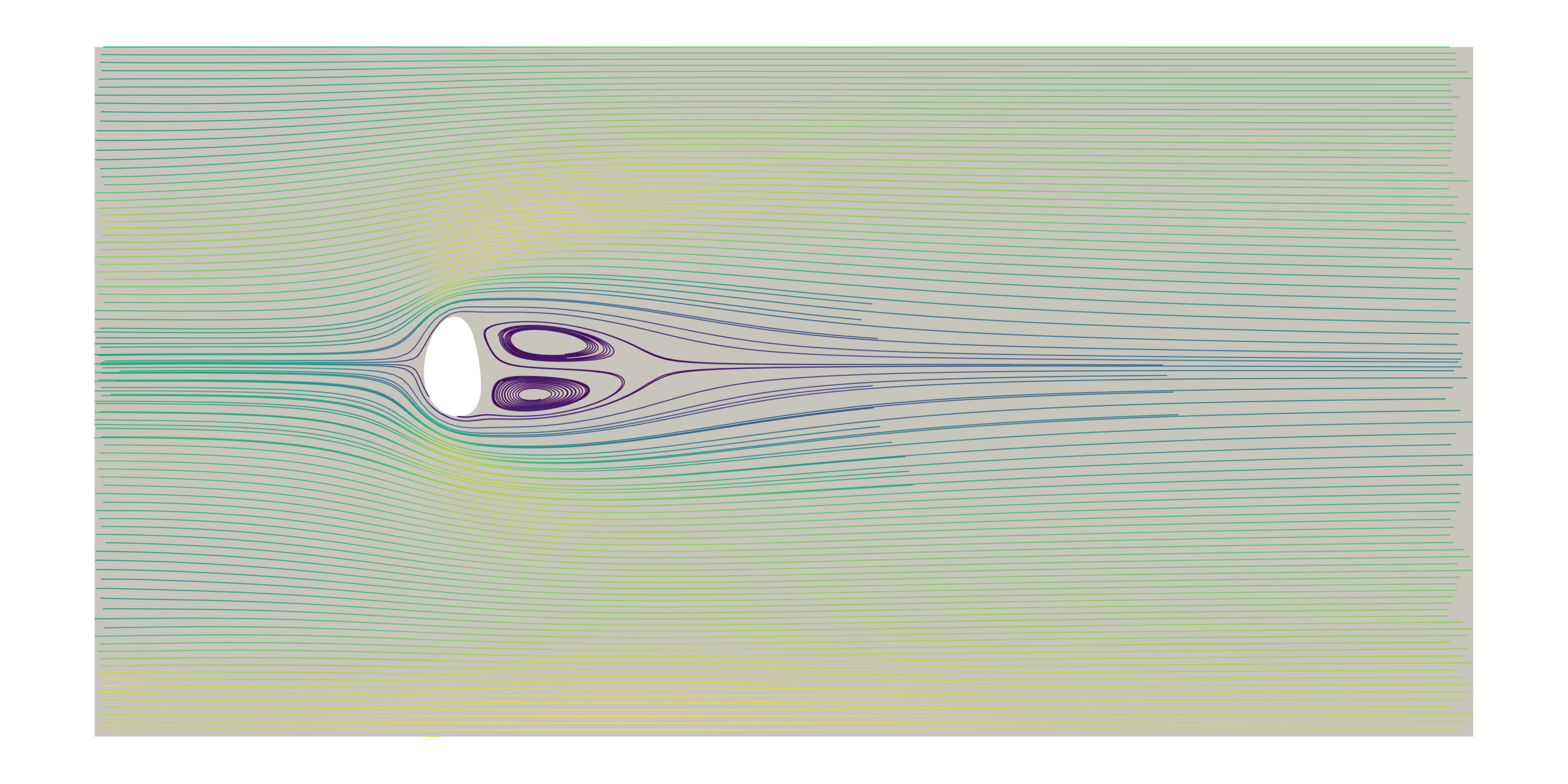}} 
\end{tabular}
\caption{Comparison of representative flow fields for the initial object geometry and
the optimized geometry obtained using DINO trained with $1024$ samples. Each
row corresponds to one of the two inflow scenarios (eastward and northward flow),
and each column shows the velocity fields before (left column) and after (right column)
optimization. All flow fields are evaluated using high-fidelity PDE solves. The optimized geometry $z^\dagger_\text{DINO}$, though not symmetric, significantly reduces
recirculation behind the object and lowers the overall viscous dissipation rate compared to the
initial shape.}
\label{fig:channel_pellet_optimal}
\end{figure}

These geometry-induced flow improvements motivate a quantitative assessment
of the risk-averse design performance. To this end, we present the evolution of the optimal objective function value as a function of the number of state PDE solves,
together with the associated $(25\%$–$75\%)$ quantile bands in Figure~\ref{fig:channel_ouu_eval}. For the neural-operator surrogates, both DINO and NO were trained with 512, 1{,}024, 2{,}048, and 4{,}096 samples, while all surrogate-based OUU runs used $8{,}192$ samples per iteration. In comparison, the PDE-based OUU problems were solved with 16, 32, 64, 128, 256, and 512 samples per iteration, and the PDE reference OUU problem was carried out with $8{,}192$ samples per iteration. All optimal designs $z^\star$ were subsequently assessed using the identical set of $16{,}384$ samples employed in the reference evaluation. As in the Poisson PDE examples, the errors are averaged across 10 runs with different random inflow samples corresponds to the same random seeds used in neural network training.

We make several key observations from these results. First, even with relative small sample sizes in the SAA formulation, the PDE-based approach yields reasonably accurate OUU solutions for the viscous dissipation objective. Nevertheless, the DINO surrogate achieves comparable or
lower optimal dissipation values using nearly an order of magnitude fewer
state PDE solves, demonstrating a substantial gain in sample efficiency. By
contrast, the NO surrogate without Jacobian learning does not attain the same cost-effectiveness as its performance does not exceed that of the PDE-based approach under identical computational budgets. With the same budget of state PDE solves, the DINO-based OUU achieves lower
optimal dissipation values than both the NO-based OUU and the direct PDE-based OUU, indicating superior sample efficiency of DINO in the risk-averse
design setting.

\begin{figure}[H]
    \centering
    \includegraphics[width=0.98\linewidth]{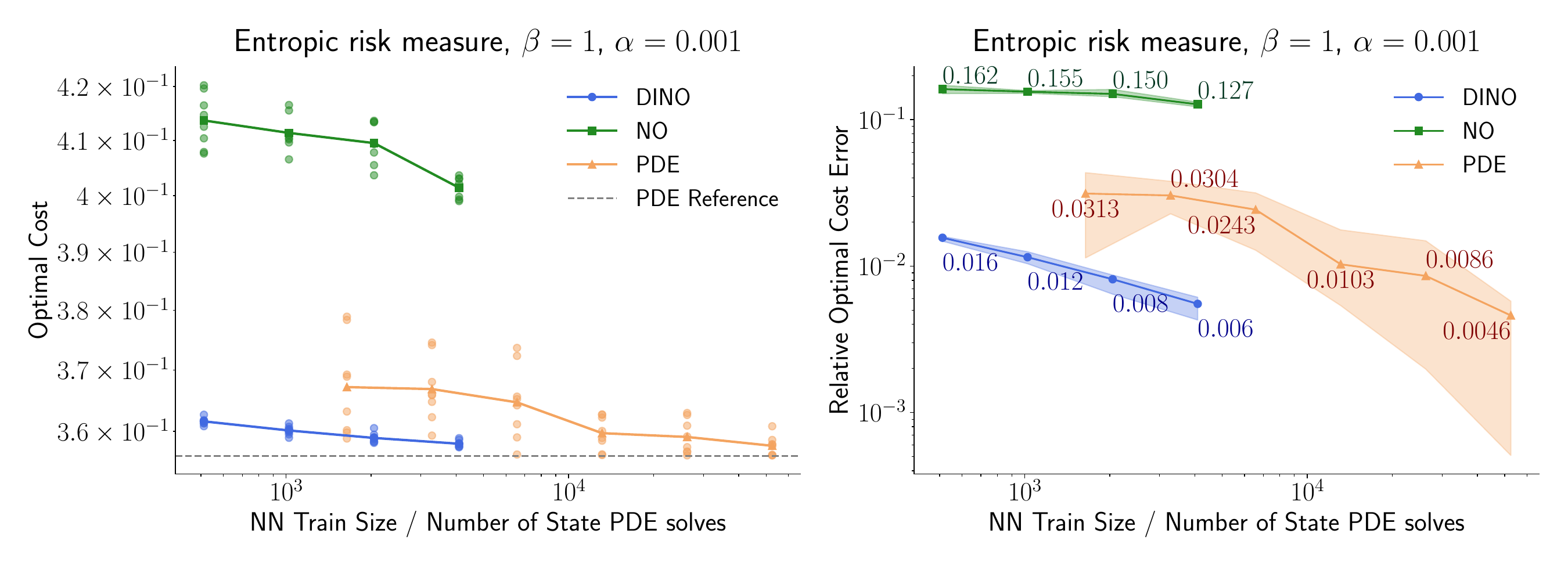}
    \caption{Optimal cost values at optimal solutions (left) and relative error of the objective function (right) versus the number of state PDE solves
required for the optimization, both with risk-aversion parameter $\beta = 1$ for entropic risk measure and penalty weight $\alpha = 0.001$. NN train sizes are 512, 1{,}024, 2{,}048, and 4{,}096, and all used 8{,}192 samples per iteration for ouu. PDE ouu per iteration with 16, 32, 64, 128, 512 samples. PDE reference ouu is done with 8{,}192 samples for iteration. The results are evaluated with the same 16{,}384 samples.
Note that the counts reported above refer to the number of samples \textit{per flow scenario}. The actual number of PDE solves is twice the reported number since we evaluate the QoI for both the eastward flow and northward flow scenarios.
}
\label{fig:channel_ouu_eval}
\end{figure}

\subsection{Design for Drag Reduction in 3D under Multiple Uncertain Inflow Conditions}
\label{section:ns3d}

We now extend the 2D problem considered in Section to a 3D example.
That is, we consider the flow in a channel $\Omega_{\channel} = (0, 2) \times (0, 1) \times (0, 1) \subset \mathbb{R}^3$ past a tower-shaped object $\Omega_{\object}$, whose geometry will be designed using free-form deformations.
In this case, the boundary of the channel, $\partial \Omega_{\channel}$ can be decomposed into its six planar faces, 
\begin{alignat*}{2}
\Gamma_{\text{left}} &= \{(x_1,x_2,x_3) \in \partial\Omega_{\channel}: x_1 = 0\}, 
    & \quad \Gamma_{\text{right}} & = \{(x_1,x_2,x_3) \in \partial\Omega_{\channel}: x_1 = 2\},\\
\Gamma_{\text{front}} &= \{(x_1,x_2,x_3) \in \partial\Omega_{\channel}: x_3 = 0\}, 
    & \quad \Gamma_{\text{back}} & = \{(x_1,x_2,x_3) \in \partial\Omega_{\channel}: x_3 = 1\},\\ 
\Gamma_{\text{base}} &= \{(x_1,x_2,x_3) \in \partial\Omega_{\channel}: x_2 = 0\}, 
    & \quad \Gamma_{\text{top}} & = \{(x_1,x_2,x_3) \in \partial\Omega_{\channel}: x_2 = 1\}.
\end{alignat*}

We again consider two flow scenarios corresponding to inflow along the eastward and northward directions respectively.
We formulate this in the same manner as \Cref{section:ns2d} using two separate orientations of the object geometry, i.e.,
$\Omega_{\object}^{\eastward} = \diffeo_z(\Omega_{\object})$ in the eastward flow and $\Omega_{\object}^{\northward} = \reflect(\Omega_{\object}^{\eastward})$
in the northward flow, where $\reflect(\cdot)$ now denotes a reflection about the $X_1 = X_3$ plane.
The spatial domains are then given by $\Omegaz^{s} := \Omega_{\channel} \setminus \closure{\Omega_{\object}^{s}}$.
The solution for each flow scenario $u^{s} = (v^{s}, p^{s}), s \in \{\eastward, \northward\}$ 
is then obtained by solving \eqref{eq:ns_pellet} with the boundaries
\begin{alignat*}{2}
\Gamma_{\text{in}} &= \Gamma_{\text{left}}
, & \quad 
\Gamma_\text{no-slip} &= \Gamma_{\text{base}} \cup \Gamma_{\object}^{s}
, \\
\Gamma_{\text{out}} &= \Gamma_{\text{right}}
, & \quad 
\Gamma_\text{free-slip} &= \Gamma_\text{front} \cup \Gamma_\text{top}\cup \Gamma_\text{back}
.
\end{alignat*}
That is, 
inflow conditions are imposed at the left surface, 
free-flow conditions are imposed at the right surface,
no-slip conditions are imposed at the base (ground) $x_2 = 0$ and the surface of the object, 
and free-slip conditions are imposed on the remaining surfaces. 
In particular, the inflow flow profile is defined by the model parameter, which scales a power law profile with exponent equal to $1/5$,
\begin{equation}
    (v_{z}^{s})_{\inflow} = \exp(\mz^{s}(x)) x_2^{1/5} e_{1}, \qquad e_1 = (1, 0, 0)^{T}.
\end{equation}
This specific power law is common in modeling boundary layer phenomena in an atmospheric setting \cite{schlichting2016boundary}.
We again take $\mz^{s}$ to be Gaussian random fields defined on the reference domain, i.e.,
$\mz^{s} = T_{\diffeo_z}^{-1} \mref^{s}$, $\mref^{s} \sim \cN(0, \cC_m^{s})$, where the covariance operators $\cC_m^{s}$ are again the Bi-Laplacian operator with correlation lengths and pointwise variances chosen to match that of the 2D case.

Figure~\ref{fig:left-inflow} shows the computational domain and the initial object geometry, which is a frustum-shaped tower, with streamlines of resulting velocity field obtained from the PDE solution. This configuration serves as the reference geometry for the following numerical experiments.
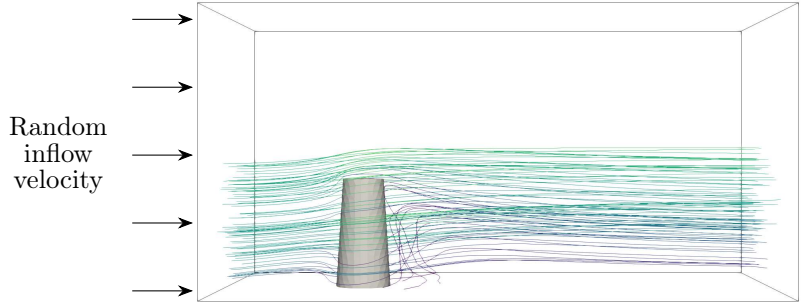
\begin{figure}[hbt!]
\centering
\begin{tikzpicture}
  \node[anchor=south west,inner sep=0] (img) at (0,0)
    {\includegraphics[width=8cm]{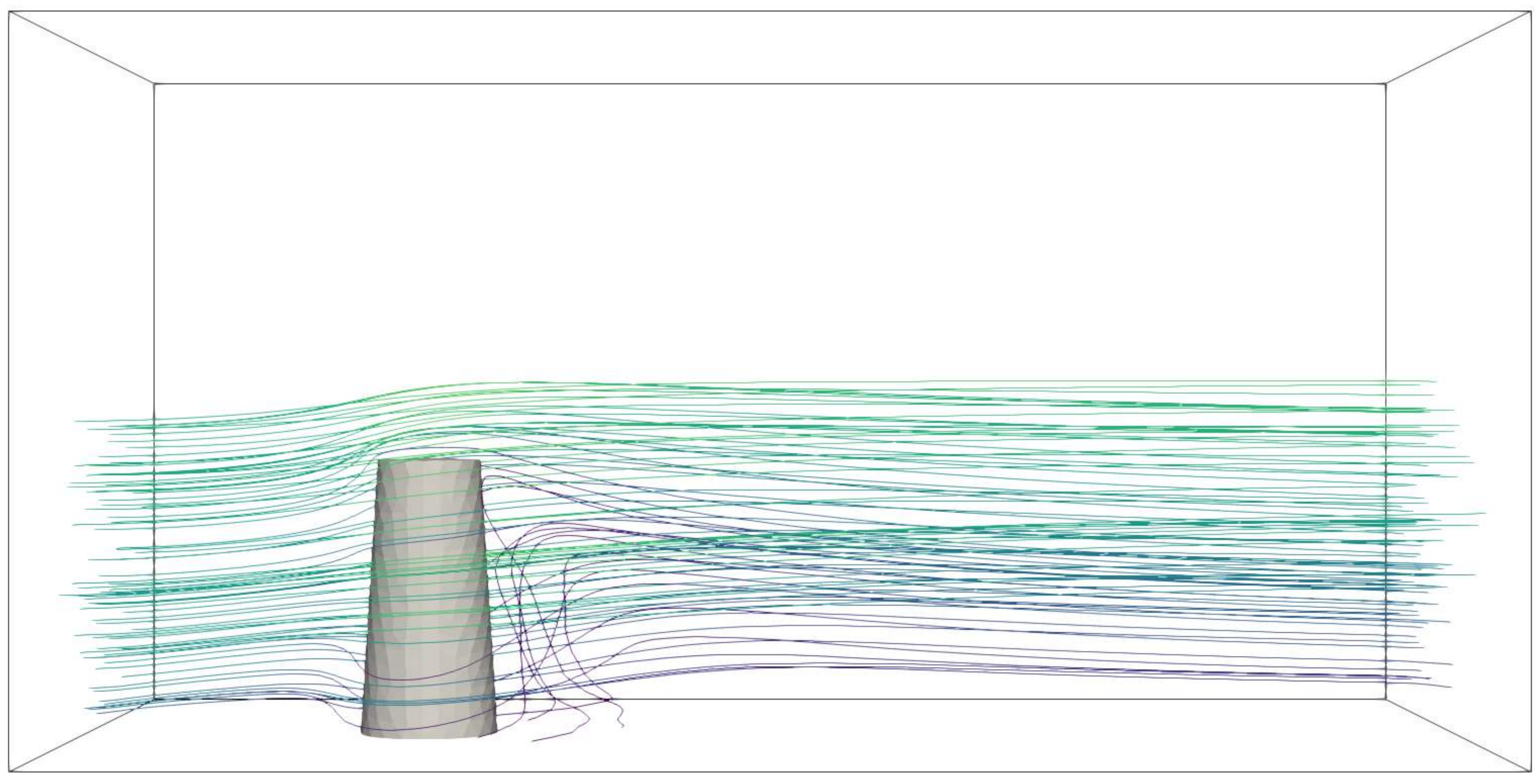}};

  \draw[-{Stealth[length=2mm]}, thin] (-0.8,2) -- (0,2);
  \draw[-{Stealth[length=2mm]}, thin] (-0.8,3.8) -- (0,3.8);
  \draw[-{Stealth[length=2mm]}, thin] (-0.8,2.9) -- (0,2.9);
  \draw[-{Stealth[length=2mm]}, thin] (-0.8,0.2) -- (0,0.2);
  \draw[-{Stealth[length=2mm]}, thin] (-0.8,1.1) -- (0,1.1);
  \node[anchor=east] at (-1,2) {\shortstack[c]{Random\\inflow\\ velocity}};
\end{tikzpicture}

\caption{Geometry of the domain and initial frustum-shaped tower together with streamlines of the velocity field obtained from the PDE Navier–Stokes solution. The case shown corresponds to the left inflow condition with inflow on $\Gamma_{\text{left}}$
    and outflow on $\Gamma_{\text{right}}$. The remaining faces $\Gamma_\text{free-slip} = \Gamma_\text{front} \cup \Gamma_\text{top}\cup \Gamma_\text{back}$ act as impermeable walls, while no-slip conditions are imposed on $\Gamma_\text{object} \cup \Gamma_\text{base}$.}
\label{fig:left-inflow}
\end{figure}

The shape parametrization is adapted from the 2D free-form deformation framework described above \eqref{eq:2d_bernstein} to include Bernstein polynomial bases along all three axes. Specifically, the Bernstein control points lie on a $6 \times 6 \times 6$ uniform lattice inside the unit cube $[0,1]^3$, 
where the points are located at $\{0, 0.2, 0.4, 0.6, 0.8, 1\}$ along each axis. Similar to the 2D problem, we exclude the first and last control points on the $X_1$ and $X_3$ axes, as well as the first point on the $X_2$ axis,
such that the left, base, front, and back boundaries are not moved. 
Moreover, we only allow the control points to move in the horizontal plane $X_1-X_3$, such that the height of object does not change. 
As a result, we have $\dz = 160$ shape degrees of freedom, 
coming from a total of 80 control points, each contributing $2$ degrees of freedom.

To generate the geometry variation for the northward flow scenario, 
we construct a mirrored displacement basis by swapping the parametric coordinates
\[
(X_1, X_2, X_3) \mapsto (X_3, X_2, X_1), \qquad (e_1, e_2, e_3) \mapsto (e_3, e_2, e_1)
\]
which reflects the displacement fields about the plane $X_1 = X_3$. 

The PDE is then formulated in the reference domain following \eqref{eq:ns_ref_2d}.
We discretize the domain using a tetrahedral mesh with Taylor-Hood elements for the state and quadratic elements for the parameter, 
resulting $\du = 393{,}331$
and $d_{\mathcal M} = 3{,}593$.

\subsubsection{Problem Formulation}
In this problem, we consider two different optimization QoIs that both capture notion of the drag forces on the object. 
These are (i) the total horizontal force on the object and (ii) the total bending moment on the object with respect to the center of its base. These QoIs are defined separately for each inflow scenario indexed by $s \in \{ \eastward, \northward \}$. 
For each case, let $u^{s}=(v^{s},p^{s})$ 
denote the velocity–pressure solution of the steady incompressible Navier--Stokes equations in spatial coordinates
The 
force and bending moments acting on the tower boundary 
$\Gamma_{\object}$ with outward unit normal $n$ are defined as
\begin{align}
  \mathbf{F}_z^{s}(\uz^{s}, \uz^{s})
  &= \int_{\Gamma_{\object}^{s}} \stressz(\uz^{s}) \, n \, ds,
  \label{eq:hydro-force-i} \\
  \mathbf{M}_z^{s}(\uz^{s}, \mz^{s})
  &= \int_{\Gamma_{\object}^{s}} (x - O) \times \left( \stressz (\uz^{s}) n \right)
  \, 
  ds,
  \label{eq:hydro-moment-i}
\end{align}
where $x=(x_1,x_2,x_3)^T$ are the spatial coordinates, $O = (0.5, 0, 0.5)^T$ is the center of the base of the structure, and $\times$ denotes the cross product.
We let 
$\mathbf{F}^{s}(u^{s}, m^{s}, z)$ 
and 
$\mathbf{M}^{s}(u^{s}, m^{s}, z)$ 
denote their respective forms when the input coordinates are in the reference domain (i.e., the force and moment vector are still expressed in spatial coordinates). 
Writing $\mathbf{F}^{s}=(F_1^{s}, F_2^{s}, F_3^{s})$
and $\mathbf{M}^{s}=(M_1^{s}, M_2^{s}, M_3^{s})^T$, 
the two scalar QoIs that we will consider are the total horizontal force
\begin{equation}\label{eq:total_force_qoi}
    Q_{\mathrm{force}}^{s}(u^{s}, m^{s}, z) := (F_1^{s})^2 + (F_2^{s})^2,
\end{equation}
where the vertical component, $F_2^{s}$, is omitted, 
and the total bending moment
\begin{equation}\label{eq:bending_moment_qoi}
    Q_{\mathrm{bend}}^{s}(u^{s}, m^{s}, z) := \|\mathbf{M}^{s}(u^{s},m^{s}, z) \|_2^2 = (M_1^{s})^2 + (M_2^{s})^2 + (M_3^{s})^2.
\end{equation}

In addition to the optimization QoIs, we incorporate penalization terms to regularize the design and enforce physical consistency. Again, two types of penalization are considered: (i) an $\ell_2$ penalty on the shape control coefficients \eqref{eq:L2penalty}, and (ii) a gravity moments penalization.
For the gravity moment, we assume the structure itself has a uniform solid density of $\rho_{\object} = 2{,}000$.
The moment about the center of its base is then given by 
\begin{equation}
    \mathbf{M}^{g}(z) = \int_{\Omega_{\object}} (x - O) \times \rho_{\object} g \; e_2 \; dx,
\end{equation}
where $g = 9.81$ is the gravitational acceleration constant and $e_2 = (0, -1, 0)^T$ is the unit vector in the vertical direction.
Similar to the integrals above, we can formulate this gravitational moment integral in the reference domain. 
We then take 
\begin{equation}
    \label{eq:grav_penalty}
    \cP_{\mathrm{grav}}(z;\alpha) := \frac{\alpha}{2} \|\mathbf{M}^{g}(z)\|^2_2.
\end{equation}

The objective functions for design under uncertainty are then given 
by the weighted sum of entropic risk measures across the two flow scenarios
supplemented by an appropriate penalization term.

For the horizontal force reduction, we adopt the objective function
\begin{equation}
    \cJ_{\force}(z) = 0.75 \cJ_{\force}^{\eastward}(z) + 0.25 \cJ_{\force}^{\northward}(z) + \cP_{\ell_2}(z;\alpha)
\end{equation}
where for each flow scenario $s \in \{\eastward, \northward\}$, we have 
\begin{equation}
    \cJ_{\force}^{s}(z) = \rho^{\entropic,\beta}_{m^{s} \sim \mu_m^{s}}(Q^{s}(u^{s}, m^{s}, z)).
\end{equation}
Here, the penalization $\mathcal{P}_{\ell_2}$ is the $\ell_2$ penalization used in \eqref{eq:L2penalty}, for which we take $\alpha = 2\times 10^{-5}$.

For bending-moment reduction, we adopt the objective function
\begin{equation}
    \cJ_{\bend}(z) = 0.75 \cJ_{\bend}^{\eastward}(z) + 0.25 \cJ_{\bend}^{\northward} (z) + \cP_{\grav}(z;\alpha)
\end{equation}
with $\alpha = 1$, where for each flow scenario $s \in \{\eastward, \northward\}$, we have 
\begin{equation}
    \cJ_{\bend}^{s}(z) := \rho^{\entropic,\beta}_{m^{s} \sim \mu_m^{s}}(Q_{\bend}^{s}(u^{s}, m^{s}, z)).
\end{equation}
This formulation can be thought of as minimizing the sum of the flow-induced moments and gravity induced moments.

\subsubsection{Neural Operator Construction}

The training data for the neural operator are generated by sampling from the input distributions of $m^{s}$ and $z$, where we use the uniform distribution $\mu_z = \unif([-0.5, 0.5]^{\dz})$ for the shape variables.
Since the PDE is much more expensive to solve in the 3D setting, we train the neural networks on a small training dataset consisting of size $384$ both with and without Jacobian information. 
In each case, to generate the reduced basis, 
we use a subset of $n_{\POD} = n_{\AS} = 256$ samples 
from the training set to compute both the POD and the AS,
from which we take $r_u = 200$ POD basis vectors 
and $r_m = 128$ AS basis vectors.
Unlike the previous two examples, we only perform a single training run for each case rather than multiple independent runs
due to the limited dataset.
Multi-input dense neural networks with 2 hidden layers of width 512 are used to represent the latent space mapping. 
The neural networks are then trained using Adam for 2,000 epochs with an initial
learning rate of $1 \times 10^{-4}$ that is halved after 1,000 epochs and 1,500 epochs. For the eastward flow scenario, the resulting test accuracies on full order $u$ over a testing dataset of size 128 are $7.22\%$ for Shape-DINO and $18.91\%$ for Shape-NO. In the northward flow scenario, 
the corresponding test accuracies are $4.13\%$ for Shape-DINO and $12.55\%$ for Shape-NO.

\subsubsection{Shape Optimization using Neural Operators}
We then deploy the trained neural operators to solve both the force and bending-moment minimization problems (without retraining). 
Since the computational cost associated with the 3D Navier--Stokes PDE largely precludes PDE-based risk-averse optimization without extensive high-performance computing resources, 
we instead compare our surrogate-based solutions to the optimal design computed from a deterministic PDE-based optimization $z^\dagger_{\mathrm{PDE}}$ using a nominal value of $m^{s} = \bar{m}^{s}$.
This corresponds to a single-point approximation of the OUU objective, which can be obtained at more modest computational costs.
Moreover, we then assess the quality of the PDE-based and surrogate-based solutions by inspecting the resulting QoI distributions at the optimal designs computed on 200 samples of the inputs $\mref^{s}$.

The resulting distributions are shown in \Cref{fig:tf_qoi_dist,fig:bm_qoi_dist} for the force and bending-moment objectives.
From the resulting distribution plots, we observe that for both QoIs, the distributions induced by the shape-DINO optimal design are more tightly concentrated and shifted closer to zero compared to those obtained from the deterministic PDE solution, and significantly closer than those produced by shape-NO. This observation indicates that shape-DINO yields designs that are simultaneously robust with respect to multiple performance metrics, reducing both the magnitude and variability of the responses under uncertainty. In contrast, the shape-NO solutions exhibit broader distributions with heavier tails for both QoIs, emphasizing the importance of incorporating derivative information during surrogate training. 

We further note that the deterministic PDE-based optimization is performed using single realizations of the uncertain parameter $\mref^{s}$ at its mean values and therefore does not account for variability induced by uncertainty. As a result, while the deterministic solution may perform well at $\bar{m}^{s}$, 
it provides no guarantees of robustness with respect to the inflow distributions. 
On the other hand, surrogate-based models enable efficient large-batch evaluation of the SAA objectives and specifically shape-DINO provides accurate objective gradients, which would be prohibitively expensive using high-fidelity PDE solves.

\begin{figure}[!htpb]
    \centering
    \includegraphics[width=0.98\linewidth]{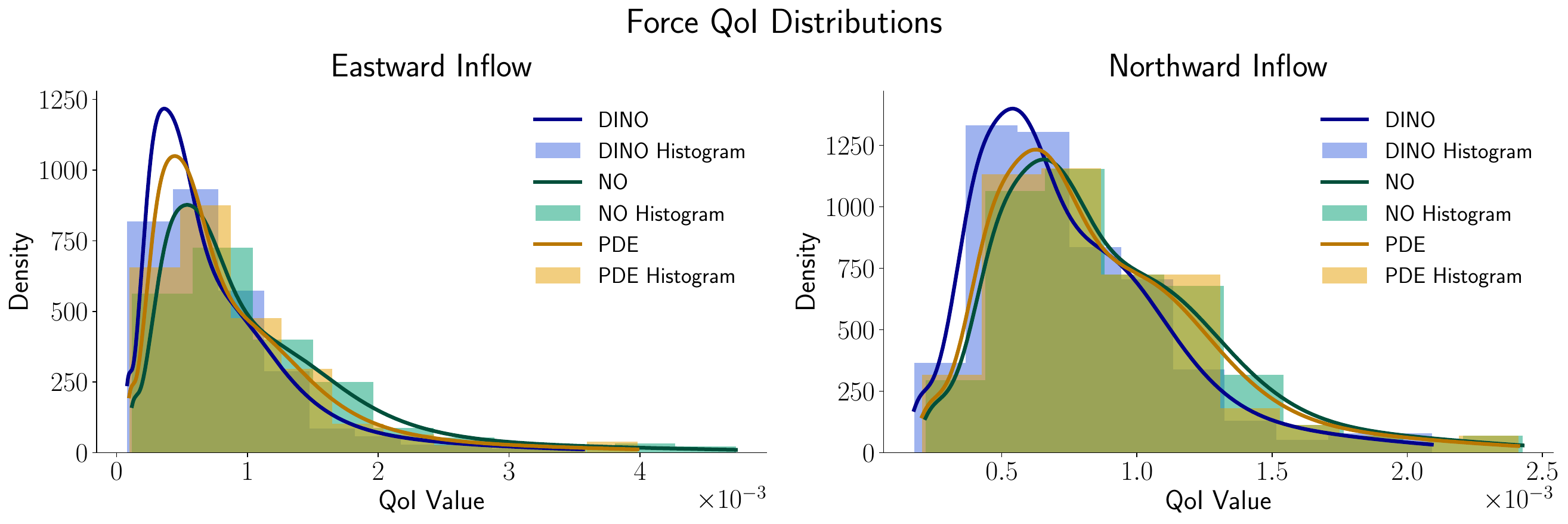}
    \caption{Distribution of horizontal-force QoIs for the optimal designs 
$z^\dagger_{\mathrm{DINO}}$, $z^\dagger_{\mathrm{NO}}$, and $z^\dagger_{\mathrm{PDE}}$. 
All designs are evaluated with 200 Monte Carlo samples. The DINO and NO solutions 
are obtained with 1,024 samples per iteration, while the PDE solution is computed 
deterministically at $\bar{m}^{s}$. The Shape-DINO design yields a distribution that is both more concentrated and closer to zero, while shape-NO exhibits broader distributions with heavier tails, highlighting the benefit of incorporating derivative information for achieving robust designs under uncertainty.}
    \label{fig:tf_qoi_dist}
\end{figure}

\begin{figure}[!htpb]
    \centering
    \includegraphics[width=0.98\linewidth]{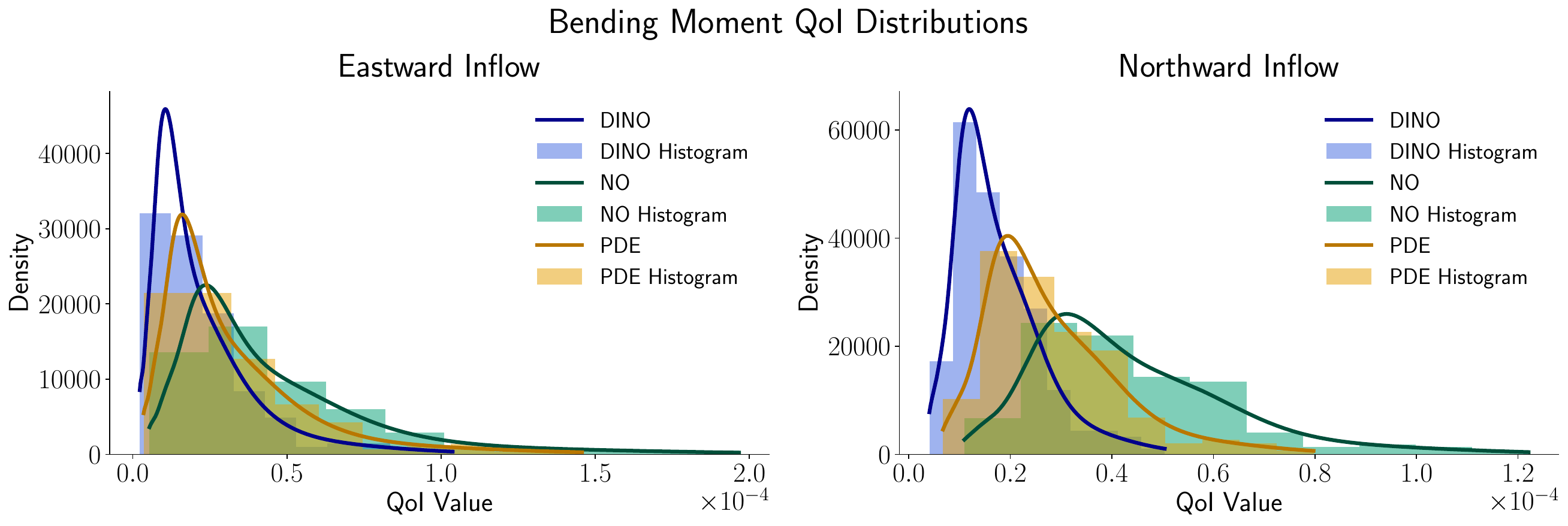}
    \caption{Distribution of bending-moment QoIs for the optimal designs 
$z^\dagger_{\mathrm{DINO}}$, $z^\dagger_{\mathrm{NO}}$, and $z^\dagger_{\mathrm{PDE}}$. 
All designs are evaluated with 200 Monte Carlo samples. The DINO and NO solutions 
are obtained with 1,024 samples per iteration, while the PDE solution is computed 
deterministically at $\bar{m}^{s}$. The Shape-DINO distribution is much more tightly concentrated and shifted closer to zero, indicating reduced magnitude and variability of the bending-moment QoI compared to both the deterministic PDE and shape-NO designs.}
    \label{fig:bm_qoi_dist}
\end{figure}

We present the corresponding optimal tower designs $z^\dagger$ 
obtained for both quantities of interest in Figure~ \ref{fig:opt_tower}. To illustrate the resulting flow behavior, streamlines are visualized under the mean parameter configuration $\bar{m}^{s}$.
Consistent with the QoI distribution results discussed above, the optimal designs obtained using shape-DINO are thinner compared to deterministic PDE, which allows the flow to pass smoothly around the tower with limited obstruction. In contrast, the shape-NO designs give bulkier and even twisting geometries that increase flow deflection, manifested by denser streamlines along the sides of the tower and a larger wake region downstream. These differences indicate the superiority of shape-DINO in producing optimal designs that effectively reduce horizontal forces and bending moments and provide a physical explanation for the broader and heavier-tailed QoI distribution for shape-NO as we discussed before.
\begin{figure}[htpb!]
   \centering
    \begin{subfigure}{0.49\textwidth}
    \centering
    \includegraphics[width=\linewidth]{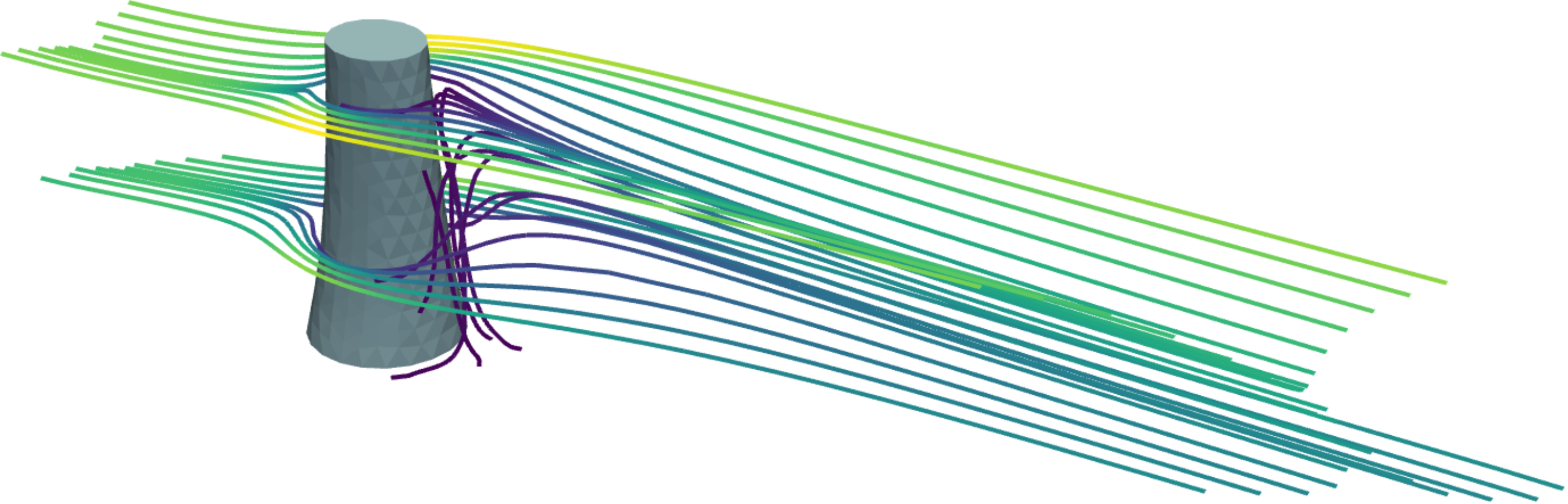}
    \caption{Horizontal Force $z^\dagger_\text{PDE}$}
    \label{fig:tf_pde}
    \end{subfigure}
    \begin{subfigure}{0.49\textwidth}
        \centering
        \includegraphics[width=\linewidth]{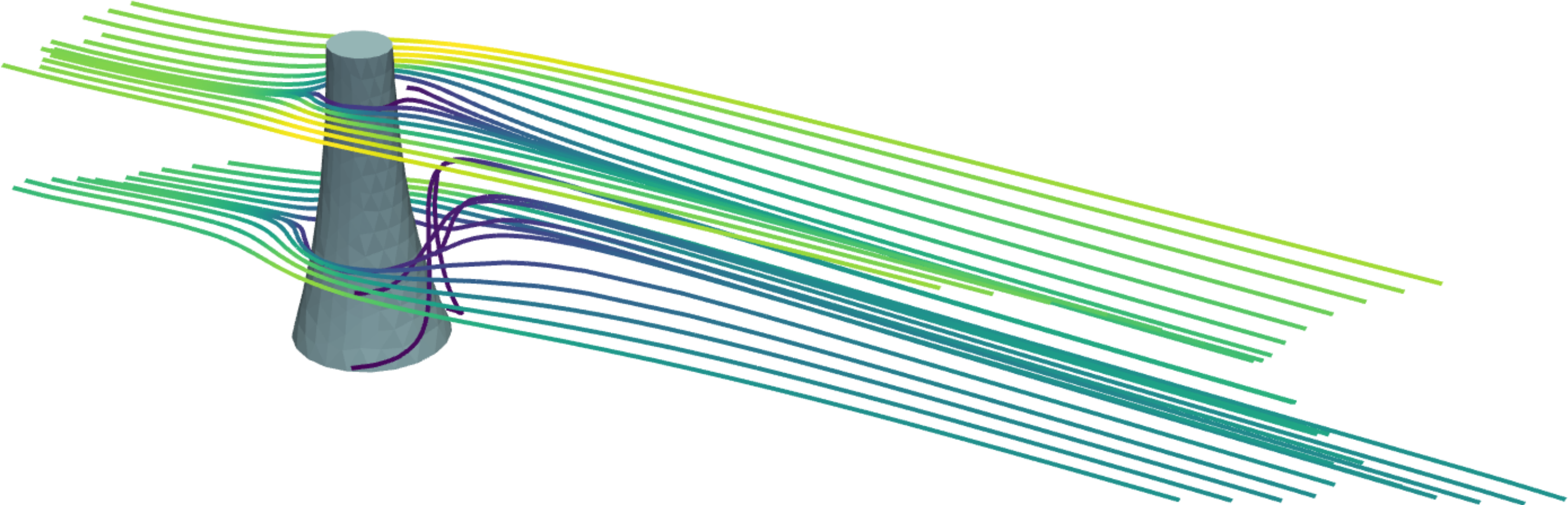}
        \caption{Bending Moment $z^\dagger_\text{PDE}$}
        \label{fig:bm_pde}
    \end{subfigure}
    
    \begin{subfigure}{0.49\textwidth}
        \centering
        \includegraphics[width=\linewidth]{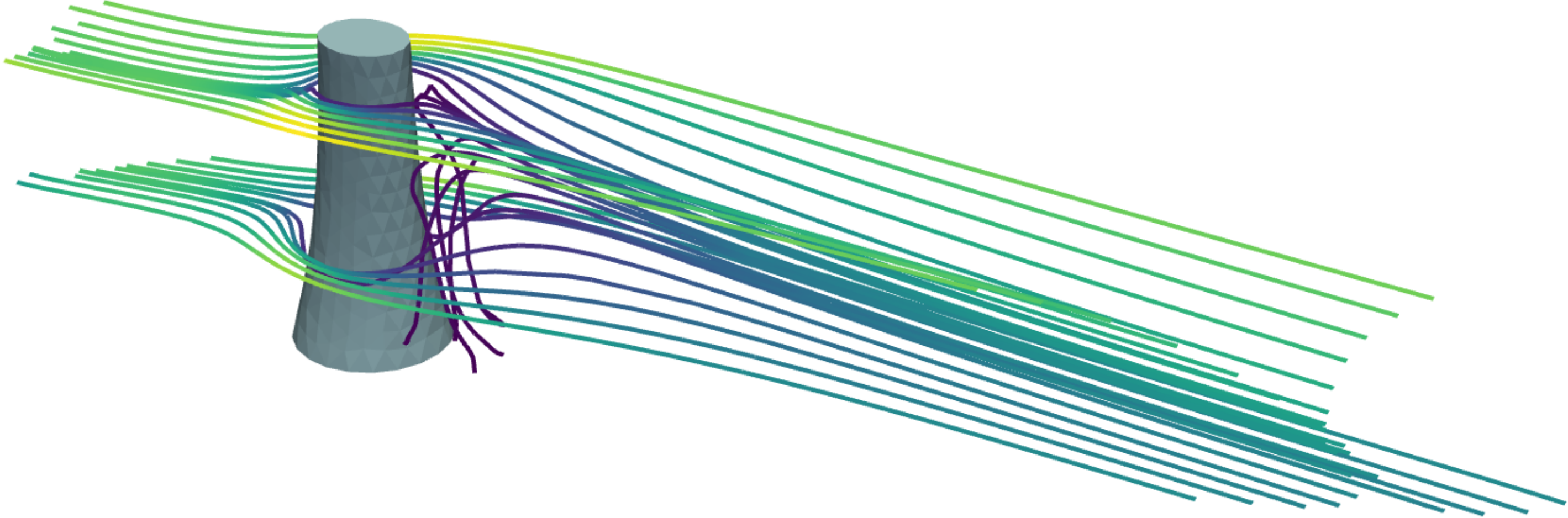}
        \caption{Horizontal Force $z^\dagger_\text{DINO}$}
        \label{fig:tf_dino}
    \end{subfigure}
    \begin{subfigure}{0.49\textwidth}
        \centering
        \includegraphics[width=\linewidth]{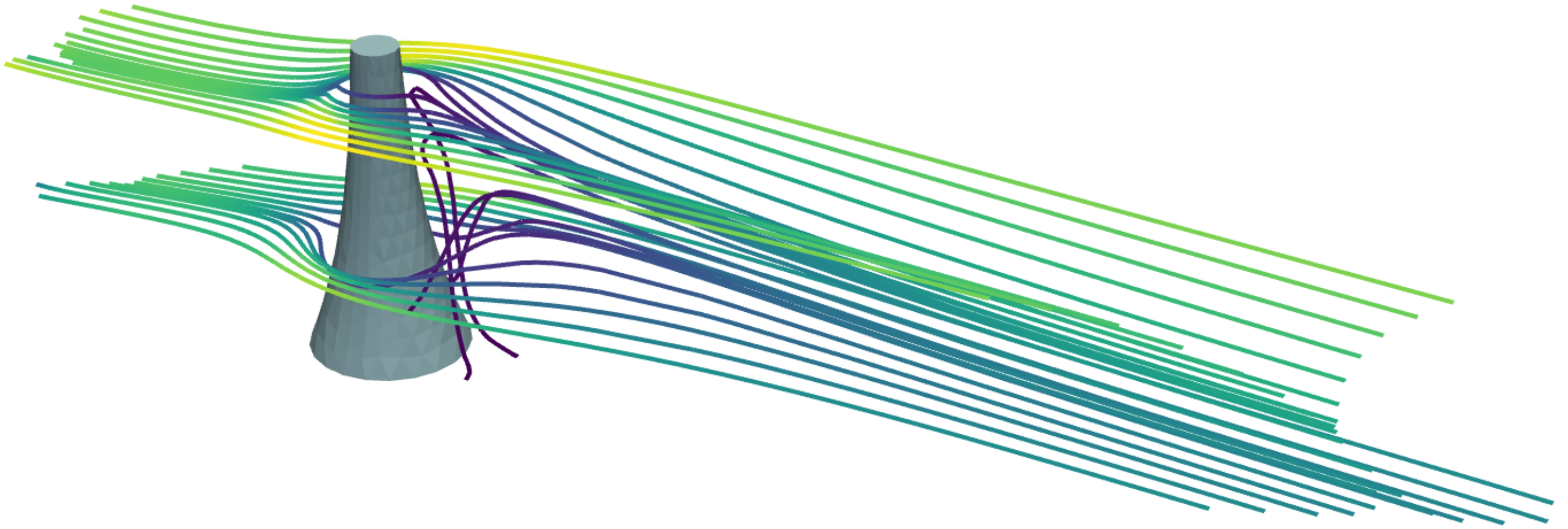}
        \caption{Bending Moment $z^\dagger_\text{DINO}$}
        \label{fig:bm_dino}
    \end{subfigure}

    \begin{subfigure}{0.49\textwidth}
        \centering
        \includegraphics[width=\linewidth]{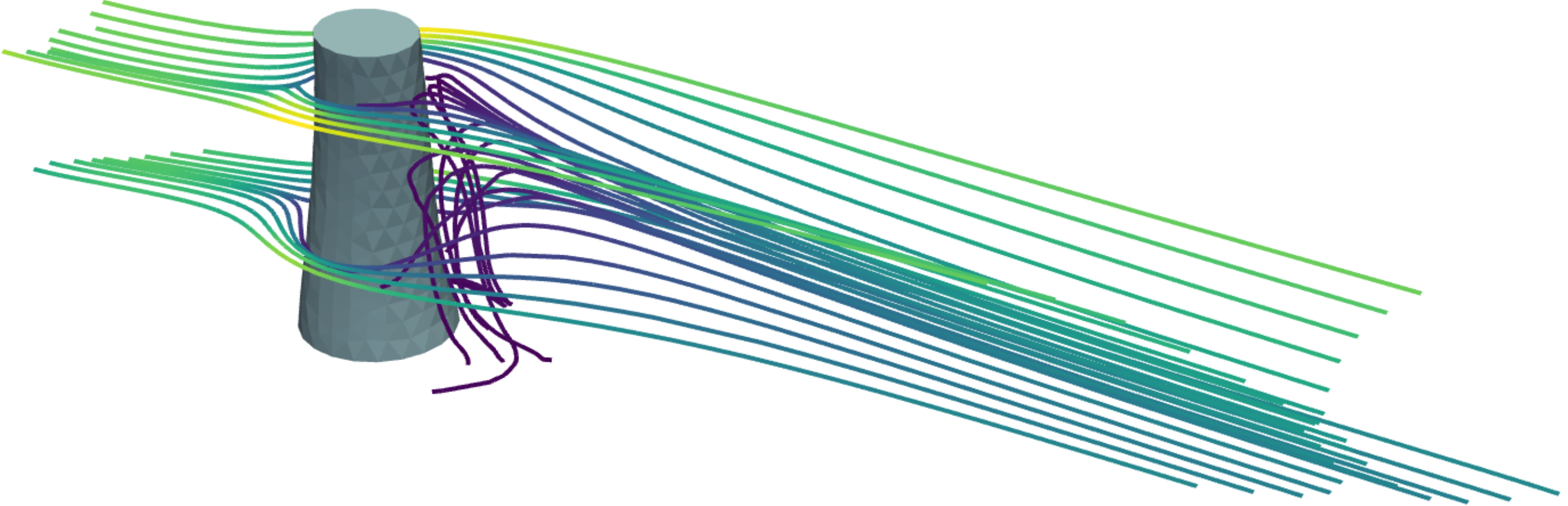}
        \caption{Horizontal Force $z^\dagger_\text{NO}$}
        \label{fig:tf_no}
    \end{subfigure}
    \begin{subfigure}{0.49\textwidth}
        \centering
        \includegraphics[width=\linewidth]{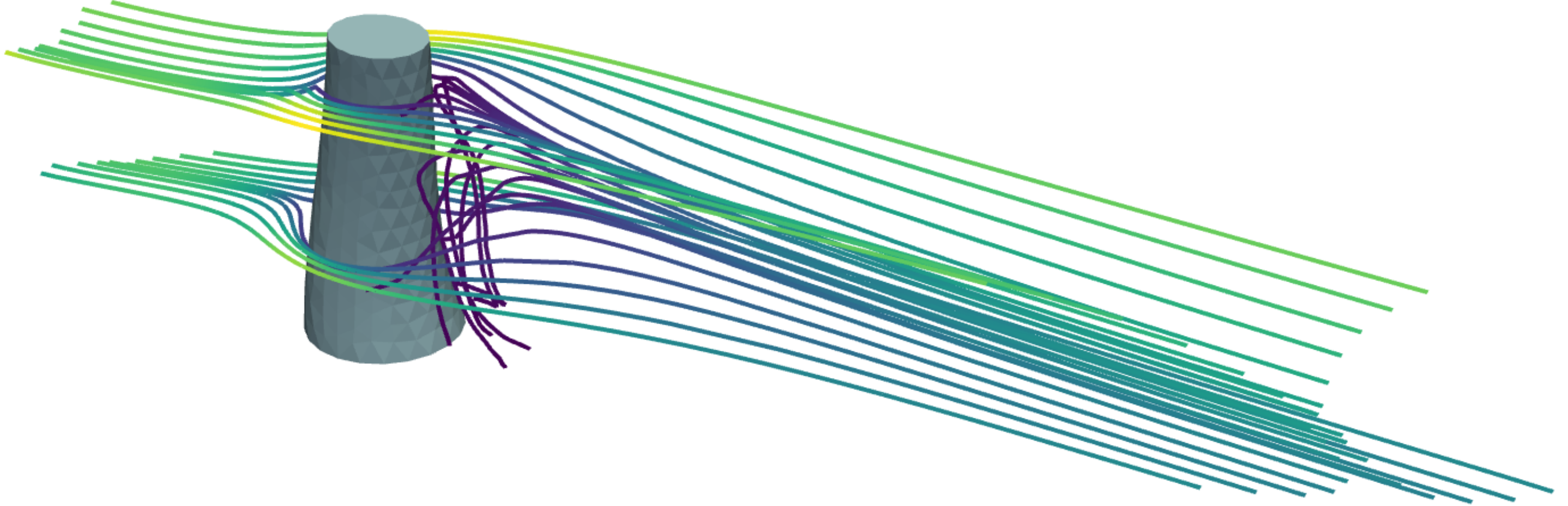}
        \caption{Bending Moment $z^\dagger_\text{NO}$}
        \label{fig:bm_no}
    \end{subfigure}
    
    \caption{Optimal tower shapes $z^\dagger$
    with streamlines evaluated under eastward flow at the mean parameter values $\mref^{\eastward} = \bar{\mref}^{\eastward}$.
    Panels 
    \ref{fig:tf_pde}, \ref{fig:tf_dino}, and \ref{fig:tf_no} correspond to the horizontal force QoI,
    while panels 
    \ref{fig:bm_pde},  \ref{fig:bm_dino}, and \ref{fig:bm_no} show results for the bending moment QoI. The Shape-DINO designs are thinner and allow smoother flow passage compared to deterministic PDE, whereas shape-NO yields bulkier and sometimes twisted geometries that induce stronger flow deflection and larger downstream recirculation regions. This comparison provides a physical explanation for the reduced forces and bending moment achieved by Shape-DINO.
    }
    \label{fig:opt_tower}

\end{figure}

\subsection{Comparison of timings}\label{section:numerical_timings}
In this section, we report the computational cost of evaluating the state and subsequent VJPs using either the PDE solver or the neural operator approximation. Execution times are measured using device-appropriate timers: CPU timings are measured as wall-clock elapsed time using Python’s \texttt{time.perf\_counter}, while GPU timings are measured as device execution times using CUDA events. Since the state and VJP evaluations account for the majority of the computational cost in both training-data generation, neural operator training, and the evaluation of the objective function and its gradients in OUU, we use these timing results to contextualize the sample-efficiency comparisons from the previous sections and to highlight the online speedups achieved by the Shape-DINO approach.

Recall that for the PDE-based approach, state evaluation corresponds to solving the state PDE, which itself may consist of solving a sequence of linear systems when the state PDE is nonlinear, while the VJP requires solving a single linear system corresponding to the adjoint PDE.
We therefore report the single-core timings of the solving the state and adjoint PDEs, where we use the direct solver MUMPS \cite{amestoy2000mumps} as the underlying linear solver.
For surrogate-based methods, we simply report the timing for the forward (state) and backward (VJP) passes of the neural operator; we present per-sample timings for both single sample and batched (1000 sample) evaluations.
Because Shape-NO and Shape-DINO share the same network architecture, these results are reported as one single time.  
The results for the Poisson, 2D and 3D Navier-Stokes examples are shown in Table \ref{tab:timings}. 

We first focus on the PDE-based timings. In the Poisson example, where the state PDE is linear, the costs of the state and adjoint solves are indeed comparable. 
In fact, when direct solvers are used, this additional adjoint solve can actually be computed at negligible cost by reusing the transpose of the matrix factorization for the state PDE.
On the other hand, in the Navier--Stokes examples, 
the cost of the adjoint solves are on average $0.2$--$0.25\times$ that of the state PDE solves due to the Newton iterations required in the state PDE solves.
Moreover, since direct solvers allow multiple VJPs to be computed at the cost of a single VJP,
the cost of generating the additional Jacobian data for DINO training using direct solvers is indeed a fraction 
of the cost of the state data (negligible amount in the linear Poisson case and $0.2$--$0.25\times$ in the Navier--Stokes case).
Thus, compared to the orders-of-magnitude gain in sample efficiency seen in the previous sections, the benefits of derivative-based training far outweigh this small additional cost per sample.

In terms of strictly online costs in optimization, neural operators provide substantial speedups over PDE solvers when used to evaluate the QoI $(m,z) \mapsto Q(u(m,z)m,z)$ and its gradients with respect to the shape variable $z$. 
For state solves, single sample neural operator evaluations are more than two orders of magnitude faster for the Poisson problem and more than four orders of magnitude faster for the 2D Navier--Stokes problem. For batched evaluations the speedups exceed five and six orders of magnitude for the Poisson and 2D Navier--Stokes problems, respectively. 
This speedup is particularly noteworthy because batched execution is native to single-GPU systems. In contrast, achieving comparable throughput for CPU-based PDE solves in a bulk-synchronous setting typically requires additional hardware, whereas neural operator evaluation typically does not. Similar advantages are observed for VJP/adjoint solve used to evaluate the $z$ gradient of the QoI. In the single sample case, the neural operator yields nearly four orders of magnitude speedup in the Poisson problem and three orders of magnitude speedup for the 2D Navier--Stokes problem. In the batched case the speedups exceed four orders of magnitude for both problems. 

The advantage becomes even more pronounced in the 3D Navier--Stokes setting, where the computational cost of high-fidelity PDE solves increases dramatically due to both the higher discretization dimension and the increased complexity of sparse factorizations in 3D relative to 2D.
For this problem the neural operators state solve speedups exceed six and eight orders of magnitude for the single sample and batched evaluation respectively. The VJP speedups likewise exceed five and seven orders of magnitude for the single sample and batched evaluations, respectively.
These results demonstrate that replacing repeated PDE solves with neural operator surrogates yields significant computational savings, making large-scale optimization problem feasible at a cost that can be prohibitive using only PDE solvers.

\begin{table}[H]
\centering
\renewcommand{\arraystretch}{1.15}
\begin{tabular}{l||ccc||ccc}
\hline
\multicolumn{7}{c}{\textbf{Execution time per sample (seconds)}} \\
\hline

& \multicolumn{3}{c||}{\textbf{State Solve}} &
\multicolumn{3}{c}{\textbf{VJP / Adjoint}} \\
\cline{2-4} \cline{5-7}

\textbf{Problem} & \textbf{PDE} & \textbf{NN (1)} & \textbf{NN (1000)}
& \textbf{PDE} & \textbf{NN (1)} & \textbf{NN (1000)} \\
\hline

Poisson
& 1.06e$-$1 & 4.12e$-$4 & 5.50e$-$7
& 8.84e$-$2 & 1.25e$-$5 & 1.35e$-$6 \\

2D Navier--Stokes
& 5.08e$+$0  & 3.86e$-$4 & 1.70e$-$6
& 1.07e$+$1  & 1.15e$-$4 & 3.16e$-$6 \\

3D Navier--Stokes
& 7.18e$+$2  & 3.85e$-$4 & 5.46e$-$6
& 1.62e$+$2  & 9.68e$-$4 & 1.07e$-$5 \\

\hline
\end{tabular}
\caption{Time (in seconds) for solving the state PDE and computing vector--Jacobian products (VJPs) using neural operators and PDE-based solvers. PDE-based solver timings are averages over 100 independent random samples of $(m,z)$. For neural operators, NN (1) denotes inference performed on a single sample, while NN (1000) corresponds to batched inference over 1000 samples evaluated simultaneously, with the time normalized per sample. The reported times for NN (1) and NN (1000) are both averages over 100 independent trials. For neural operators, we report the time to compute the VJP by a backward pass, while for PDE-based method, we report the time taken to solve the adjoint PDE, which is the dominant cost of the VJP. Neural operators achieve orders-of-magnitude reductions in computational time compared to PDE-based solvers for both state and VJP evaluations, with particularly pronounced gains in the 2D and 3D Navier--Stokes cases. These gains are further amplified under batched inference (NN (1000)), where native GPU parallel evaluation leads to additional per-sample reductions in execution time.}
\label{tab:timings}
\end{table}

\section{Conclusions}\label{sec:conclusions}
In this work, we have presented a novel surrogate-based framework, Shape-DINO, for accelerating PDE-constrained shape OUU problems. By representing geometric variability through diffeomorphic mappings to a fixed reference domain and utilizing reduced-basis neural operator architectures with derivative-informed training, the proposed framework jointly learns the solution map and its derivatives with respect to both the random parameters and the design variables. This framework directly addresses the coupled challenges of geometric variability, sampling complexity, fast online evaluation, and gradient accuracy that arise in risk-based shape optimization. 

Beyond the methodological contributions, to the best of our knowledge, this work provides the first end-to-end theoretical analysis of RBNO surrogates for optimization under uncertainty. We proved the universal approximation properties (Theorem \ref{theorem:rbno_ua_semibounded}) of the multi-input RBNO and established bounds on the induced optimization error in the risk neutral setting (\Cref{theorem:rbno_ouu_approx}), providing a rigorous justification for surrogate-based optimization.    

Through a collection of numerical experiments, we have demonstrated that derivative-informed neural operators can efficiently solve a range of shape OUU problems. In all cases, incorporating Jacobian information during training led to substantial improvements in state prediction accuracy, gradient fidelity, and most importantly, reliable optimal designs. These observations are consistent with the a priori error bounds derived in Section \ref{sec:formulation}, which show that optimization errors are directly controlled by errors in both the solution and its shape derivatives. As a result, Shape-DINO consistently produced more robust solutions and more reliable optimization behavior than non-derivative neural operators. Moreover, the reduced-basis architecture enabled fast inference, while surrogate-based SAA allowed large-batch sampling that would be computationally infeasible using PDE solves alone. 

Our experiments further showed that Shape-DINO achieves significantly better cost–accuracy trade-offs than PDE-based approaches. For comparable levels of accuracy, Shape-DINO required several orders of magnitude fewer PDE solves than direct SAA-based PDE optimization, while enabling fast online evaluation once trained. This performance aligns with the approximation results established in Section \ref{sec:shape_dino}, which provides theoretical guarantee that reduced-basis multi-parameter neural operators can approximate both the solution and its derivatives arbitrarily well in appropriate $C^1$ norms. In practice, this allows the offline training cost to be amortized across multiple optimization runs involving different objectives, risk measures, or regularization choices, enabling repeated design queries, interactive design-space exploration, and near-real-time shape optimization under uncertainty.

In addition, we demonstrated that the proposed framework scales to complex 3D problems, where fully PDE-based optimization is computationally prohibitive. These results highlight the potential of Shape-DINO for novel applications involving high-dimensional geometry, uncertainty, and complex physical models.

We also acknowledge the limitations of our framework. When the optimization objective is relatively simple, i.e. accurately approximated with few PDE solves, our framework may be unnecessarily complex. In such settings, classical PDE-based optimization may be more appropriate. Shape-DINO is more advantageous in regimes where geometric variability, uncertainty, and sampling complexity make traditional approaches computationally infeasible. Also, our framework relies on diffeomorphic mesh deformation and elastic extensions, which are effective for moderate geometric variations. However, the chosen shape parametrization and training distribution may not adequately cover large or extreme deformations. If the optimization explores shapes far outside this regime, mesh distortion can occur and can affect both PDE accuracy and surrogate training. 

Finally, the results of this work suggest several directions for future research, including extensions to incorporation of higher-order derivative information, adaptive sampling strategies guided by optimization progress, and applications to more complex geometric parameterizations. Moreover, rather than serve as proxies for the PDE solves in evaluating the cost functional, surrogates can act as control variates to be employed along with small numbers of PDE-based evaluations to accelerate estimation of risk measures within a variance reduction framework \cite{NgWillcox14,ChenVillaGhattas19}. This opens the door to methods that are intermediate between neural operator-based and PDE-based cost functional evaluation, and to studying optimal tradeoffs between PDE solves during Shape-DINO training versus PDE solves during control variate-based optimization. 

More broadly, this work demonstrates that incorporating shape and parametric sensitives into neural operator training is essential for reliable and scalable OUU, positioning Shape-DINO as a powerful tool for complex PDE-constrained design optimization workflows. Beyond risk-based shape optimization, we envision that the proposed shape derivative informed operator learning framework can be employed in a broad range of settings, including PDE-based inverse problems for geometric and shape reconstruction, Bayesian inversion under geometric uncertainty, and near real time parametric deterministic shape optimization. This highlights the potential of Shape-DINO as a general computational framework for design, inference, and decision making for complex engineered systems.

\section{Acknowledgments}\label{sec:acknowledgments}
This research was partially supported by the U.S.\ Department of Energy under award DE-SC0023171,
the U.S.\ Air Force Office of Scientific Research under award FA9550-24-1-0327, by the U.S.\ National Science Foundation under awards DMS-2324643 and OAC-2313033, and by the Royal Society of New Zealand Te Ap\={a}rangi (Marsden Fund
Council) grant MFP-24-UOA-279.
The authors acknowledge the Texas Advanced Computing Center (TACC) at The University of Texas at Austin for providing computational resources that have contributed to the research results reported within this paper.

\addcontentsline{toc}{section}{References}
\bibliographystyle{model1-num-names}
\biboptions{sort,numbers,comma,compress}                 
\bibliography{shapedino.bib}

\appendix
\renewcommand{\thesection}{\Alph{section}}

\section{Spatial and Reference Domain Transformations}

\subsection{Common Coordinate Transformations} \label{appendix:coordinate_transforms}
Recall that under the diffeomorphism $\diffeo$, 
we can convert between spatial and reference domain coordinates by a change of variables 
\[
    x = \diffeo (X), \qquad X = \diffeo^{-1}(x)
\]
for $X \in \Omegaref$ and $x \in \diffeo(\Omegaref)$.
Differential and integral quantities transform accordingly. 
In \Cref{tab:transform}, we present the coordinate transformations for common quantities that are used in our numerical examples. 
Further details can be found in standard continuum mechanics texts such as \cite{GonzalezStuart01}.
\begin{table}[!htpb]
    \centering
    \renewcommand{\arraystretch}{1.2}
    \caption{Common coordinate transformations. Here $\defgrad := \nabla_X \diffeo$ is the deformation gradient tensor, $|\cdot|$ is the Euclidean norm of the vector, and $(\cdot)^T$ corresponds to the matrix transpose.}
    \label{tab:transform}
    \begin{tabular}{l c c}
        \toprule
         & Spatial coordinates $x$ & Reference coordinates $X$ \\
        \midrule
         Coordinate & $x$ & $\diffeo(X)$ \\
         Function & $u(x) $& $u(\diffeo(X)) = (T_{\diffeo}^{-1}u)(X)$  \\
         Gradient (scalar function) & $\nabla_x u(x) $& $\defgrad^{-T}(X) \nabla_X u(\diffeo(X)) $\\
         Gradient (vector function) & $\nabla_x v(x) $& $\nabla_X v(\diffeo(X)) \defgrad^{-1}(X) $\\
         Divergence & $\nabla_x \cdot v(x) $ & $\mathrm{tr}(\nabla_X v(\diffeo(X)) \defgrad^{-1}(X) )$ \\ 
         Volume element & $dx$ &  $\det \defgrad(X) dX$\\
         Vector surface element & $n(x) ds(x)$ & $\det \defgrad(X) \defgrad^{-T}(X) N(X) dS(X)$\\ 
         Scalar surface element & $ds(x)$ & $|\det \defgrad(X) \defgrad^{-T}(X) N(X)|dS(X)$\\ 
         \bottomrule
    \end{tabular}
\end{table}

\subsection{Proof of Proposition \ref{prop:pullback_pushforward_operators}}
\label{appendix:proof_pushforward}
\begin{proof}
    We start by deriving the norm bounds for the pushforward operator.
    First, recall that for $\defgrad(X) = \nabla_X \diffeo(X)$, we have 
    \[|\det \defgrad(X)| \leq \|\defgrad(X)\|_{\Op(\bR^{d},\bR^{d})}^{d} \leq \|\diffeo\|_{W^{1,\infty}(\Omega_0)}^{d}.\]
    Consider now any $u \in L^2(\Omega_0)$ and its pushforward $T_{\diffeo}^{-1}u = u(\diffeo^{-1}(x))$. By a change of variables, the $L^2(\Omega_1)$ norm of $T_{\diffeo^{-1}} u$ can be written as
    \begin{align*}
        \| T_{\diffeo}^{-1} u \|_{L^2(\Omega_1)}^2 
           &= \int_{\Omega_1} |(T_{\diffeo}^{-1}u)(x)|^2 dx \\
           &= \int_{\Omega_1} |u(\diffeo^{-1}(x))|^2 dx \\
           &= \int_{\Omega_0} |u(X)|^2 \det \defgrad(X) dX \\
           &\leq \|\diffeo\|_{W^{1,\infty}(\Omega_0)}^{d} \|u\|_{L^2(\Omega_0)}^2,
    \end{align*}
    which is the desired result for the $L^2$ case.

    For the $H^1$ case, we consider any $u \in H^1(\Omega_0)$.
    The $H^1(\Omega_1)$ seminorm of the pushforward can be written as 
    \begin{align*}
        \int_{\Omega_1} |\nabla_x (T_{\diffeo}^{-1}u)(x)|^2 dx 
           &= \int_{\Omega_1} |\nabla_x u(\diffeo^{-1}(x))|^2 dx \\
           &= \int_{\Omega_0} |\defgrad(X)^{-T} \nabla_X u(X)|^2 \det \defgrad(X) dX \\
           &\leq \|\diffeo\|_{W^{1,\infty}(\Omega_0)}^d \|\diffeo^{-1}\|_{W^{1,\infty}(\Omega_1)}^2 |u|_{H^1(\Omega_0)}^2,
    \end{align*}
    Combining this with the $L^2$ result, we have 
    \begin{equation}\label{eq:change_of_variable_h1_norm}
        \|T_{\diffeo}^{-1} u\|_{H^1(\Omega_1)}^2
        \leq 
            \left( 1 + \|\diffeo^{-1}\|_{W^{1,\infty}(\Omega_1)}^2 \right)
            \|\diffeo\|^d_{W^{1,\infty}(\Omega_0)}
            \|u\|_{H^1(\Omega_0)}^2.
    \end{equation}
    This yields the desired result for the $H^1$ case. Repeating the argument with $\diffeo$ instead of $\diffeo^{-1}$ yields the results for the pullback operator $T_{\diffeo}$.
\end{proof}

\section{Optimization Error}
\label{appendix:opt_error}
\subsection{Derivative Domination Bound}
\label{appendix:gradient_domination}
\begin{proof}(of \Cref{theorem:optimization_error})
    The proof is based on standard variational analysis. 
    In the following, we will use the shorthand notation $\langle \cdot ,\cdot \rangle$ 
    to denote the duality pairing $\langle \cdot, \cdot \rangle_{\cZ' \times \cZ}$.
    We first note that by $\lambda$-strong convexity of $\cJ$, for any $z_1, z_2 \in \cZ_{ad}$, we have
    \[
        \cJ(z_2) \geq \cJ(z_1) + \langle D_z \cJ(z_1), z_2 - z_1 \rangle + \frac{\lambda}{2} \|z_2 - z_1\|_{\cZ}^2,
    \]
    which additionally implies that 
    \[
        \langle D_z \cJ(z_2) - D_z \cJ(z_1), z_2 - z_1 \rangle \geq \lambda \|z_2 - z_1\|_{\cZ}^2.
    \]

    \paragraph{Optimality gap} 
    This is simply a constrained optimization form of the Polyak-\L{}ojasiewicz inequality and can be shown similarly.
    Using the strong convexity of $\cJ$, we have 
    \[
        \cJ(z^{\star}) \geq \cJ(z^{\dagger}) 
            + \langle D_z \cJ(z^{\dagger}), z^{\star} - z^{\dagger} \rangle
            + \frac{\lambda}{2} \|z^{\star} - z^{\dagger}\|_{\cZ}^2.
    \]
    We can rearrange this to obtain
    \begin{equation}\label{eq:optgap_base}
        \cJ(z^{\dagger}) - \cJ(z^{\star}) \leq
            - \langle D_z \cJ(z^{\dagger}), z^{\star} - z^{\dagger} \rangle
            - \frac{\lambda}{2} \|z^{\star} - z^{\dagger}\|_{\cZ}^2.
    \end{equation}
    Since $z^{\dagger}$ satisfies the first-order optimality conditions of $\cJ_{\theta}$, 
    we have additionally that
    \begin{equation}\label{eq:optanalysis_firstopt}
        \langle D_z \cJ_{\theta}(z^{\dagger}), z^{\star} - z^{\dagger} \rangle \geq 0.
    \end{equation}
    Adding \eqref{eq:optanalysis_firstopt} to \eqref{eq:optgap_base}, we have 
    \begin{align}
        \cJ(z^{\dagger}) - \cJ(z^{\star}) & \leq
           \langle D_z \cJ_{\theta}(z^{\dagger}) - D_z \cJ(z^{\dagger}), z^{\star} - z^{\dagger} \rangle
            - \frac{\lambda}{2} \|z^{\star} - z^{\dagger}\|_{\cZ}^2 \nonumber \\
        & \leq
           \| D_z \cJ_{\theta}(z^{\dagger}) - D_z \cJ(z^{\dagger})\|_{\cZ'} \|z^{\star} - z^{\dagger}\|_{\cZ}
            - \frac{\lambda}{2} \|z^{\star} - z^{\dagger}\|_{\cZ}^2. \label{eq:optgap_almost}
    \end{align}
    The right hand side of \eqref{eq:optgap_almost} is of the form a quadratic function 
    $a t - \lambda t^2/2$, 
    where $a = \| D_z \cJ_{\theta}(z^{\dagger}) - D_z \cJ(z^{\dagger})\|_{\cZ'} \geq 0$ and $t$ plays the role of  $\|z^{\star} - z^{\dagger}\|_{\cZ}$.
    This quadratic is maximized at $t^{\star} = a / \lambda$.
    Substituting this maximizer into \eqref{eq:optgap_almost} as an upper bound yields 
    \begin{equation}
        \cJ(z^{\dagger}) - \cJ(z^{\star}) \leq
           \frac{1}{2\lambda}\| D_z \cJ_{\theta}(z^{\dagger}) - D_z \cJ(z^{\dagger})\|_{\cZ'}^2,
    \end{equation}
    which is the desired result.

    \paragraph{Optimization error} 
    We begin with \eqref{eq:optanalysis_firstopt}, to which we add and subtract the derivative terms $D_z \cJ(z^{\star})$ and $D_z \cJ(z^{\dagger})$. This gives 
    \[
        \langle 
        D_z \cJ_{\theta}(z^{\dagger})
        - D_z \cJ (z^{\dagger})
        + D_z \cJ (z^{\dagger})
        - D_z \cJ (z^{\star})
        + D_z \cJ (z^{\star})
        ,
        z^{\star} - z^{\dagger} \rangle \geq 0.
    \]
    Rearranging this yields
    \begin{equation}\label{eq:opterror_base_inequality}
        \langle 
            D_z \cJ (z^{\star}) - D_z \cJ (z^{\dagger}), z^{\star} - z^{\dagger} 
        \rangle 
        \leq 
        \langle 
            D_z \cJ_{\theta}(z^{\dagger}) - D_z \cJ (z^{\dagger}), z^{\star} - z^{\dagger} 
        \rangle
        +
        \langle 
            D_z \cJ (z^{\star}) , z^{\star} - z^{\dagger} 
        \rangle.
    \end{equation}
    On the left hand side of \eqref{eq:opterror_base_inequality}, strong convexity implies that 
    \begin{equation} \label{eq:opterror_lhs}
        \langle 
            D_z \cJ (z^{\star}) - D_z \cJ (z^{\dagger}), z^{\star} - z^{\dagger} 
        \rangle  
        \geq 
        \lambda \|z^{\star} - z^{\dagger}\|_{\cZ}^2.
    \end{equation}
    On the right hand side of \eqref{eq:opterror_base_inequality}, we have that 
    \begin{equation} \label{eq:opterror_rhs1}
        \langle 
            D_z \cJ_{\theta}(z^{\dagger}) - D_z \cJ (z^{\dagger}), z^{\star} - z^{\dagger} 
        \rangle
        \leq 
            \| D_z \cJ_{\theta}(z^{\dagger}) - D_z \cJ (z^{\dagger}) \|_{\cZ'} \| z^{\star} - z^{\dagger} \|_{\cZ},
    \end{equation}
    and 
    \begin{equation} \label{eq:opterror_rhs2}
        \langle 
            D_z \cJ (z^{\star}) , z^{\star} - z^{\dagger} 
        \rangle
        \leq 0,
    \end{equation}
    where this latter inequality follows from the first-order optimality conditions for $\cJ$.
    Substituting \eqref{eq:opterror_lhs}--\eqref{eq:opterror_rhs2} into \eqref{eq:opterror_base_inequality}, we have
    \[
        \lambda \|z^{\star} - z^{\dagger}\|_{\cZ}^2 \leq \| D_z \cJ_{\theta}(z^{\dagger}) - D_z \cJ (z^{\dagger}) \|_{\cZ'} \| z^{\star} - z^{\dagger} \|_{\cZ},
    \]
    from which the desired result follows.
    
\end{proof}

\subsection{Derivative Error of the Expectation}
\label{appendix:grad_error_summary}
In this section, we present the proof of \Cref{prop:surrogate_gradient_error}.
To this end, we first show that the expectation leads to a continuously differentiable functional 
$z \mapsto \bE_{\mu_m}[Q(u(\cdot, z), \cdot, z)]$ when $\mu_m, u, Q$ satisfy
\Cref{assumption:pde_properties,assumption:qoi_properties}. 
This allows us to compare the difference in the derivatives when $u$ is replaced by a surrogate $u_{\theta}$ as in \Cref{prop:surrogate_gradient_error}.

To begin, we note that \Cref{assumption:pde_properties} implies that the mappings $z \mapsto u(\cdot,z)$ and $z \mapsto D_z u(\cdot, z)$ are Lipschitz continuous. 
\begin{proposition}
    \label{proposition:lipschitz_u} 
    Suppose that $u$, satisfies \Cref{assumption:pde_properties} 
    with the Borel probability measure $\mu_m$, compact admissible set $\cZ_{ad}$, and open set $\openset$ 
    for some $p \geq 2$. 
    Then, the mappings 
    $\openset \ni z \mapsto u(\cdot, z) \in L^p_{\mu_m}(\cM; \cU)$
    and
    $\openset \ni z \mapsto D_z u(\cdot, z) \in L^p_{\mu_m}(\cM; \HS(\cZ,\cU))$ are Lipschitz continuous. 
    That is, for every $z_1, z_2 \in \openset$
    \begin{align}
        \| u(\cdot, z_2) - u(\cdot, z_1) \|_{L^p_{\mu_m}(\cM; \cU)} &\leq \lip_{u} \|z_2 - z_1\|_{\cZ}, \\
        \| D_z u(\cdot, z_2) - D_z u(\cdot, z_1) \|_{L^p_{\mu_m}(\cM; \HS(\cZ,\cU))} &\leq \lip_{u_z} \|z_2 - z_1 \|_{\cZ}.
    \end{align}
    Moreover,  
    there exist constant $\ubu$ and $\ubDzu$
    such that 
    \begin{align}
        \| u(\cdot, z) \|_{L^p_{\mu_m}(\cM; \cU)} &\leq \ubu, \label{eq:ub_u}\\
        \| D_z u(\cdot, z) \|_{L^p_{\mu_m}(\cM; \HS(\cZ,\cU))} &\leq \ubDzu,  \label{eq:ub_Dzu}
    \end{align}
    for every $z \in \cZ_{ad}$.
\end{proposition}
\begin{proof}
    First, for any $z \in \openset$, \eqref{eq:u_lip} of \Cref{assumption:pde_properties} yields
    \[
        \| u(m,z) \|_{\cU}^p \leq (\|u(0,z)\|_{\cU} + \lip_{u} \|m\|_{\cM})^p, 
    \]
    such that $\bE_{m \sim \mu}[ \|u(m,z)\|_{\cU}^p < \infty$ since $\mu_m$ has finite $p^{th}$ moments.
    Moreover, for any $z_1, z_2 \in \openset$, 
    \[
        \|u(\cdot, z_2) - u(\cdot, z_1) \|_{L^p_{\mu_m}}^p
        = \bE_{m \sim \mu_m}[\|u(m, z_2) - u(m,z_1)\|_{\cU}^p]
        \leq \lip_{u}^p \|z_2 - z_1\|_{\cZ}^p,
    \]
    which shows that $z \mapsto u(\cdot, z)$ is Lipschitz continuous with constant $\lip_{u}$.
    An analogous application of \eqref{eq:Dzu_lip} yields Lipschitz continuity of $z \mapsto D_z u(\cdot, z)$ with constant $\lip_{u_z}$.
    This continuity combined with compactness of the admissible set $\cZ_{ad} \subset \openset$ then implies the existence of the upper bounds in \eqref{eq:ub_u} and \eqref{eq:ub_Dzu}.
\end{proof}
    
On the other hand, the Lipschitz continuity of $D_u Q$ and $D_{uz} Q$ from \Cref{assumption:qoi_properties} implies the following bounds.
\begin{lemma}[Quadratic bounds] \label{lemma:quadratic_bounds}
Suppose $Q$ satisfies \Cref{assumption:qoi_properties} with the compact admissible set $\cZ_{ad}$ and open set $\openset$. 
Then, for every $u_1, u_2 \in \cU$, $m \in \cM$, and $z \in \openset$, we have 
\begin{equation}\label{eq:quadratic_bound}
    |Q(u_2, m, z) - Q(u_1, m, z)| \leq \| D_u Q(u_1, m, z) \|_{\cU'} \|u_2 - u_1 \|_{\cU} + \frac{\lip_{Q_u}}{2} \| u_2 - u_1 \|_{\cU}^2
\end{equation}
and 
\begin{equation}\label{eq:quadratic_derivative_bound}
    \|D_z Q(u_2, m, z) - D_z Q(u_1, m, z) \|_{\cZ'} 
    \leq 
    \| D_{uz} Q(u_1, m, z) \|_{\Op(\cU,\cZ')} \|u_2 - u_1 \|_{\cU} + \frac{\lip_{Q_{uz}}}{2} \| u_2 - u_1 \|_{\cU}^2.
\end{equation}
\end{lemma}

\begin{proof}
    Consider first \eqref{eq:quadratic_bound}. 
    Taking a Taylor expansion for $Q$ yields
    \[
        Q(u_2, m, z) - Q(u_1,m,z)   = \int_0^{1} D_u Q(u_1 + th, m, z) h dt,
    \]
    where $h = u_2 - u_1$. Taking absolute values and making use of the Lipschitz continuity of $D_u Q$, we have
    \begin{align*}
        | Q(u_2, m, z)  - Q(u_1, m, z) | 
        & \leq \int_{0}^{1} 
        \|D_u Q(u_1 + th, m, z) \|_{\cU'} \|h\|_{\cU} dt \\
        & \leq \int_{0}^{1} 
        \left( \| D_u Q(u_1, m, z) \|_{\cU'} + \lip_{Q_u} t \|h\|_{\cU} \right)
        \|h\|_{\cU} dt  \\
        & \leq
        \|D_u Q(u_1, m, z)\|_{\cU'} \|h\|_{\cU}
        + \frac{\lip_{Q_u}}{2} \|h\|_{\cU}^2.
    \end{align*}
    For $D_z Q$, we can use a similar Taylor expansion
    \[
        D_z Q(u_2, m, z) - D_z Q(u_1, m, z)  = \int_{0}^{1} D_u D_z Q (u_1 + th, m, z) h dt = 
        \int_{0}^{1} D_{uz} Q (u_1 + th, m, z) h dt,
    \]
    which follows from the fact that $Q$ is twice continuously differentiable. 
    Taking $\|\cdot\|_{\cZ'}$ and making use of the Lipschitz continuity of $D_{uz} Q$ yields \eqref{eq:quadratic_derivative_bound}.
\end{proof}

Moreover, the bounds in \Cref{lemma:quadratic_bounds} can be written as follows.
\begin{corollary}\label{corollary:other_quadratic_bound}
    Consider the setting of \Cref{lemma:quadratic_bounds}. We have 
    \begin{equation}\label{eq:cor_quadratic_bound}
        |Q(u_2, m, z) - Q(u_1, m, z)| \leq (\ubDuQ + \lip_{Q_u} \|u_1\|_{\cU}) \|u_2 - u_1\|_{\cU} + \frac{\lip_{Q_u}}{2}\|u_2 - u_1\|_{\cU}^2
    \end{equation}
    and 
    \begin{equation}\label{eq:cor_quadratic_bound_derivative}
        \|D_z Q(u_2, m, z) - D_z Q(u_1, m, z) \|_{\cZ'} 
        \leq (\ubDuzQ + \lip_{Q_{uz}} \|u_1\|_{\cU}) \|u_2 - u_1\|_{\cU} + \frac{\lip_{Q_{u_z}}}{2}\|u_2 - u_1\|_{\cU}^2.
    \end{equation}
\end{corollary}
\begin{proof}
    Using both the Lipschitz continuity of $D_u Q$ and its boundedness at $u = 0$, we can write
    \[
        \|D_u Q(u_1,m,z)\|_{\cU'} \leq \|D_u Q(0,m,z)\|_{\cU'} + \lip_{Q_u} \|u_1\|_{\cU} \leq \ubDuQ + \lip_{Q_{u}} \|u_1\|_{\cU}.
    \]
    Similarly, 
    \[
        \|D_{uz} Q(u_1,m,z)\|_{\Op(\cU,\cZ')} \leq \|D_{uz} Q(0,m,z)\|_{\cL(\cU,\cZ')} + \lip_{Q_{uz}} \|u_1\|_{\cU} \leq \ubDuzQ + \lip_{Q_{uz}} \|u_1\|_{\cU}.
    \]
    Substituting the above expressions into \eqref{eq:quadratic_bound} and \eqref{eq:quadratic_derivative_bound} yield the alternative bound given in this corollary.
\end{proof}

Alternatively, we can directly derive upper bounds on $Q$ and $D_z Q$ by applying the Taylor expansion at the point $u_1 = 0$.
\begin{corollary}\label{corollary:direct_q_bounds}
    Consider the setting of \Cref{lemma:quadratic_bounds}. For every $u \in \cU$, $m \in \cM$, and $z \in \openset$, we have 
    \begin{equation}\label{eq:cor_direct_Q_bound}
        |Q(u, m, z)| \leq \ubQ + \ubDuQ \|u\|_{\cU} + \frac{\lip_{Q_u}}{2} \|u\|_{\cU}^2
    \end{equation}
    and 
    \begin{equation}\label{eq:cor_direct_DzQ_bound}
        \|D_z Q(u, m, z)\|_{\cZ'} \leq \ubDzQ + \ubDuzQ \|u\|_{\cU} + \frac{\lip_{Q_{u_z}}}{2} \|u\|_{\cU}^2.
    \end{equation}
\end{corollary}
\begin{proof}
    The proof is almost identical to that of \Cref{lemma:quadratic_bounds}, where instead we take $u_2 = u$ and $u_1 = 0$ such that $h = u_2 - u_1 = u$. 
    We therefore have 
    \begin{equation}
        | Q(u,m,z) | \leq |Q(0,m,z)| + \|D_u Q(0,m,z)\|_{\cU'} \|u\|_{\cU} + \frac{\lip_{Q_u}}{2} \|u\|_{\cU}^2
    \end{equation}
    and  
    \begin{equation}
        \| D_z Q(u,m,z) \|_{\cZ'} \leq \|D_z Q(0,m,z)\|_{\cZ'} + \|D_{uz} Q(0,m,z)\|_{\cL(\cU,\cZ')} \|u\|_{\cU} + \frac{\lip_{Q_{uz}}}{2} \|u\|_{\cU}^2.
    \end{equation}
    Substituting the upper bounds from \Cref{assumption:qoi_properties} then yields the desired result.
\end{proof}

\subsubsection{Differentiability of the Quantity of Interest}
\label{appendix:differentiability_of_qoi}
As discussed, for any fixed $z$, the quantity $Q(u(m,z),m,z)$ is a random variable depending on the realization of $m \in \cM$. 
Thus, we can treat the mapping $z \mapsto Q(u(\cdot, z), \cdot, z)$ as a mapping from the variable $z$ to the space of random variables. 
Under \Cref{assumption:pde_properties,assumption:qoi_properties}, we can show that this is a continuously differentiable mapping from 
$\openset$ to $L^1_{\mu_m}$.
\begin{proposition}\label{prop:L1_differentiability}
    Suppose $u$ satisfies \Cref{assumption:pde_properties} and $Q$ satisfies \Cref{assumption:qoi_properties} with the Borel probability measure $\mu_m$, admissible set $\cZ_{ad}$, and open set $\openset$ for some $p \geq 2$.
    Then, the map $f_Q : z \mapsto Q(u(\cdot, z), \cdot, z)$ is a continuously differentiable function from $\openset$ to $L^1_{\mu_m}(\cM;\bR)$.
\end{proposition}
\begin{proof}
    \textit{Well-defined:}
    We begin by verifying $f_Q(z)$ is in $L^1_{\mu_m}$ for any $z \in \openset$. 
    By \eqref{eq:cor_direct_Q_bound} of \Cref{corollary:direct_q_bounds}, we have
    \[
        \|f_Q(z)\|_{L^1_{\mu_m}} = 
        \bE_{m \sim \mu}[|Q(u(m,z),m,z)|] 
        \leq
        \bE_{m \sim \mu_m} \left[ \ubQ + \ubDuQ\|u(m,z)\|_{\cU} + \frac{\lip_{Q_u}}{2} \|u(m,z)\|_{\cU}^2 \right].
    \]
    Since $u(\cdot, z) \in L^p_{\mu_m}(\cM, \cU)$ with $p \geq 2$, the resulting expectation is finite such that $f_Q(z) \in L^1_{\mu_m}$
    
    \textit{Differentiability:}
    Since $Q$ and $u$ are both continuously differentiable, the derivative is formally given by 
    \[
        D_z f_Q(z) = D_u Q(u(\cdot, z), \cdot, z) D_z u(\cdot,z) + D_z Q(u(\cdot,z), \cdot, z).
    \]
    We verify this is indeed the derivative. In particular, we consider the residual
    \[
        r(z, h) := \| Q(u(\cdot, z+h), \cdot, z + h ) - Q(u(\cdot, z), \cdot ,z) - D_z f_Q(z) h \|_{L^1_{\mu_m}}
    \]
    is $o(\|h\|_{\cZ})$.
    We begin by letting $A(m,z) := D_u Q(u(m,z),m,z) D_z u(m,z) + D_zQ(u(m,z),m,z)$, such that 
    \begin{align*}
        \frac{r(z,h)}{\|h\|_{\cZ}} 
        = \bE_{m \sim \mu_m} \left[ \frac{
            | Q(u(m,z+h), m, z+h) - Q(u(m,z), m, z) - A(m,z) h |}
            {\|h\|_{\cZ}}\right].
    \end{align*}
    In particular, by the differentiability of $Q$ and $u$, the integrand converges to zero pointwise in $m$ as $h \rightarrow 0$.
    We now seek a dominating function on the integrand.
    To this end, define
    \[
        \anoperator_1 := D_u Q(u(m,z),m,z) D_z u(m,z),  \qquad \anoperator_2 := D_z Q(u(m,z),m,z)
    \]
    such that $A = \anoperator_1 + \anoperator_2$, and 
    \[
        \differ_1 := Q(u(m,z + h), m, z+h) - Q(u(m,z), m, z+h), \quad \differ_2 := Q(u(m,z), m, z+h) - Q(u(m,z),m,z),
    \]
    such that $Q(u(m,z + h), m, z+h) - Q(u(m,z),m,z) = \differ_1 + \differ_2$.
    We can bound each term separately. 
    \begin{enumerate}
        \item For $\anoperator_1$, we make use of the Lipschitz continuity of $D_u Q$, such that 
        \begin{equation}\label{eq:A1_bound}
            \frac{|\anoperator_1(m,z) h|}{\|h\|_{\cZ}} \leq  
            \|\anoperator_1(m,z)\|_{\cZ'}  
            \leq  \ubDuQ \|D_z u(m,z) \|_{\Op(\cZ,\cU)} + \lip_{Q_u} \|u(m,z)\|_{\cU}  \|D_z u(m,z) \|_{\Op(\cZ,\cU)}.
        \end{equation}
        
        \item For $\anoperator_2$, we make use of \eqref{eq:cor_direct_DzQ_bound} from \Cref{corollary:direct_q_bounds}, such that 
        \begin{equation}\label{eq:A2_bound}
            \frac{|\anoperator_2(m,z) h|}{\|h\|_{\cZ}}
            \leq \|\anoperator_2(m,z)\|_{\cZ'}
            \leq \ubDzQ + \ubDuzQ \|u(m,z)\|_{\cU} + \frac{\lip_{Q_{uz}}}{2} \|u(m,z)\|_{\cU}^2.
        \end{equation}
    
        \item For $\differ_1$, we make use of \eqref{eq:cor_quadratic_bound} from \Cref{corollary:other_quadratic_bound}, such that 
        \begin{align*}
            |\differ_1(m,z)| & \leq (\ubDuQ + \|u(m,z)\|_{\cU}) \|u(m,z+h) - u(m,z) \|_{\cU} + \frac{\lip_{Q_u}}{2} \|u(m,z+h) - u(m,z)\|_{\cZ}^2 \\
            & \leq (\ubDuQ + \|u(m,z)\|_{\cU}) \lip_{u} \|h\|_{\cZ} + \frac{\lip_{Q_u} \lip_{u}^2 }{2} \|h\|_{\cZ}^2,
        \end{align*}
        where this latter inequality follows from the Lipschitz continuity of $u$ from \Cref{assumption:pde_properties}. 
        Therefore,
        \begin{equation}
            \frac{|\differ_1(m,z)|}{\|h\|_{\cZ}} \leq (\ubDuQ + \|u(m,z)\|_{\cU}) \lip_u + \frac{\lip_{Q_u} \lip_u^2}{2} \|h\|_{\cZ}.
        \end{equation}
        \item For $\differ_2$, we first make use of the mean value theorem, such that 
        \[
            |\differ_2(m,z)| \leq \| D_z Q(u(m,z),m, \zeta) \|_{\cZ'} \|h\|_{\cZ},
        \]
        where $\zeta = z + th$ for some $t \in [0,1]$. 
        For sufficiently small $h$, we have that $\zeta \in \openset$. 
        This allows us to apply \eqref{eq:cor_direct_DzQ_bound} from \Cref{corollary:direct_q_bounds} to obtain
        \begin{equation}
            \frac{|\differ_2(m,z)|}{\|h\|_{\cZ}} 
            \leq
            \|D_z Q(u(m,z),m,\zeta) \|_{\cZ'} 
            \leq 
            \ubDzQ + \ubDuzQ \|u(m,z)\|_{\cU} + \frac{\lip_{Q_{uz}}}{2} \|u(m,z)\|_{\cU}^2.
        \end{equation}
    \end{enumerate}
    In each case, the upper bounds are quadratic forms of $\|u(m,z)\|_{\cU}$ and $\|D_z u(m,z)\|_{\Op(\cZ,\cU)}$ (or equivalently $\|D_z u(m,z)\|_{\HS(\cZ,\cU)}$), 
    both of which are $L^2_{\mu_m}$ functions. 
    Thus, the integrand is dominated by an $L^1_{\mu_m}$ function. 
    The dominated convergence theorem then allows passing to the limit such that
    \[
        \frac{r(z,h)}{\|h\|_{\cZ}} \rightarrow 0 \qquad \text{as } h \rightarrow 0.
    \]

    \textit{Continuity of the derivative:}
    To see continuity, consider any $z_1, z_2 \in \openset$. 
    We have that
    \[
        \| D_z f_Q(z_2) - D_z f_Q(z_1)  \|_{\cZ'}
        \leq 
        \bE_{m \sim \mu_m}[
            \| \anoperator_1(m, z_2) - \anoperator_1(m,z_1) \|_{\cZ'}
            + \| \anoperator_2(m, z_2) - \anoperator_2(m,z_1) \|_{\cZ'}
        ].
    \]
    Continuity of $D_u Q$, $D_z Q$, $u$, and $D_z u$ implies that the integrand converges to zero for every $m$.
    Again, it remains to find a dominating function that allows the limit $z_2 \rightarrow z_1$ 
    to be taken outside of the expectation.
    In particular, \eqref{eq:A1_bound} yields 
    \[
        \|\anoperator_1(m,z_1)\|_{\cZ'}
        \leq  \ubDuQ \|D_z u(m,z_1) \|_{\Op(\cZ,\cU)} + \lip_{Q_u} \|u(m,z_1)\|_{\cU}  \|D_z u(m,z_1) \|_{\Op(\cZ,\cU)}.
    \]
    Making use of the Lipschitz continuity of $u$ and $D_z u$, we also have 
    for $h = z_2 - z_1$
    \begin{align*}
        \|\anoperator_1(m,z_2)\|_{\cZ'}
        & \leq  \ubDuQ \|D_z u(m,z_2) \|_{\Op(\cZ,\cU)} + \lip_{Q_u} \|u(m,z_2)\|_{\cU}  \|D_z u(m,z_2) \|_{\Op(\cZ,\cU)} \\
        & \leq  \ubDuQ \left( \|D_z u(m,z_1) \|_{\Op(\cZ,\cU)} + \lip_{u_z} \|h\|_{\cZ} \right) \\
        & \qquad + \lip_{Q_u} (\|u(m,z_1)\|_{\cU} + \lip_{u} \|h\|_{\cZ})  
            ( \|D_z u(m,z_1) \|_{\Op(\cZ,\cU)} + \lip_{u_z} \|h\|_{\cZ}),
    \end{align*}
    For sufficiently small $h$, both 
    $\|\anoperator_1(m,z_1)\|_{\cZ'}$
    and
    $\|\anoperator_1(m,z_2)\|_{\cZ'}$
    are dominated by a common $L^1_{\mu_m}$ function.
    Similarly, for $\anoperator_2$, \eqref{eq:A2_bound} yields 
    \[
        \|\anoperator_2(m,z_1)\|_{\cZ'}
        \leq 
        \ubDzQ + \ubDuzQ \|u(m,z_1)\|_{\cU} + \frac{\lip_{Q_{uz}}}{2} \|u(m,z_1)\|_{\cU}^2,
    \]
    while Lipschitz continuity of $u$ and $D_z u$ imply
    \[
        \|\anoperator_2(m,z_2)\|_{\cZ'} 
        \leq 
        \ubDzQ + \ubDuzQ (\|u(m,z_1)\|_{\cU} + \lip_{u} \|h\|_{\cZ}) + \frac{\lip_{Q_{uz}}}{2} (\|u(m,z_1)\|_{\cU} + \lip_{u} \|h\|_{\cZ})^2.
    \]
    such that 
    $\|\anoperator_2(m,z_1)\|_{\cZ'}$
    and
    $\|\anoperator_2(m,z_2)\|_{\cZ'}$
    are also dominated by a $L^1_{\mu_m}$ function.
    The dominated convergence theorem then allows passing to the limit, such that 
    \[
        \| D_z f_Q(z_2) - D_z f_Q(z_1) \|_{\cZ'} \rightarrow 0 \qquad \text{as } z_2 \rightarrow z_1,
    \]
    which shows the continuity of $D_z f_Q$.
\end{proof}
    
A corollary of this result is that the cost function defined by the expectation of $Q$ is continuously differentiable.
\begin{corollary}\label{corollary:expectation_differentiability}
    Consider the setting of \Cref{prop:L1_differentiability}. 
    Then, the function $\cJ : \openset \rightarrow \bR$ given by 
    \begin{equation}
        \cJ(z) := \bE_{m \sim \mu_m}[Q(u(m,z),m,z)]
    \end{equation}
    is continuously differentiable, and its derivative is given by 
    \begin{equation}
        D_z \cJ(z) = \bE_{m \sim \mu_m}[D_u Q(u(m,z),m,z) D_z u(m,z) + D_z Q(u(m,z),m,z)].
    \end{equation}
\end{corollary}
\begin{proof}
    This follows from the chain rule, and the fact that the expectation $\bE_{\mu_m}[ \cdot ]$ is a continuous linear functional on $L^1_{\mu_m}$, and therefore is continuously differentiable.
\end{proof}

\subsubsection{Proof of Proposition \ref{prop:surrogate_gradient_error}}
\label{appendix:proof_grad_error}
We now present the proof for \Cref{prop:surrogate_gradient_error}.
\begin{proof}
    Recall that by \Cref{corollary:expectation_differentiability}, we have the following expressions for the derivatives of 
    $\cJ^{\mean}(z) = \bE_{m \sim \mu_m}[Q(u(m,z),m,z)]$ and 
    $\cJ_{\theta}^{\mean}(z) = \bE_{m \sim \mu_m}[Q(u_{\theta}(m,z),m,z)]$,
    \begin{align*}
        D_z \cJ^{\mean}(z) &= \bE_{m \sim \mu_m}[D_u Q(u(m,z),m,z) D_z u(m,z) + D_z Q(u(m,z),m,z)], \\
        D_z \cJ_{\theta}^{\mean}(z) &= \bE_{m \sim \mu_m}[D_u Q(u_{\theta}(m,z),m,z) D_z u_{\theta}(m,z) + D_z Q(u_{\theta}(m,z),m,z)]. 
    \end{align*}
    We can write their difference using the following three terms
    \begin{align*}
        \differ_1(m,z) &:= \left( D_u Q(u(m,z),m,z) - D_u Q(u_{\theta}(m,z), m, z)  \right) D_z u(m,z), \\
        \differ_2(m,z) &:= D_u Q(u_{\theta}(m,z), m, z)  \left( D_z u(m,z) - D_z u_{\theta}(m,z) \right), \\
        \differ_3(m,z) &:=  D_z Q(u(m,z),m,z) - D_z Q(u_{\theta}(m,z), m, z), 
    \end{align*}
    such that 
    \[
        D_z \cJ^{\mean}(z) - D_z \cJ_{\theta}^{\mean} (z) = 
        \bE_{m \sim \mu_m}[ \differ_1(m,z) + \differ_2(m,z) + \differ_3(m,z)].
    \]
    Consider now any $z \in \cZ_{ad}$. We can bound the norms of each $\differ_i$ separately. 
    To simplify the notation, we will use 
    $\|u(\cdot, z)\|_{L^2_{\mu_m}}$ to denote the state Bochner norm $\|u(\cdot, z)\|_{L^2_{\mu_m}(\cM; \cU)}$
    and $\|D_z u(\cdot, z)\|_{L^2_{\mu_m}}$ to denote the derivative Bochner norm $\|D_z u(\cdot, z)\|_{L^2_{\mu_m}(\cM; \HS(\cZ,\cU))}$
    (likewise for the errors).
    Moreover, we will make use of the fact that $\|\cdot\|_{\Op(\cZ,\cU)} \leq \|\cdot \|_{\HS(\cZ,\cU)}$ throughout, and will not explicitly note this in the derivations below.
    We then have the following:
    \begin{enumerate}
        \item For $\differ_1$, we have  
        \[
            \|\differ_1(m,z)\|_{\cZ'}
            \leq 
            \| D_u Q(u(m,z), m, z) - D_u Q (u_{\theta}(m,z), m, z) \|_{\cU'}
            \| D_z u(m,z) \|_{\Op(\cZ, \cU)}.
        \]
        The Lipschitz continuity of $D_u Q$ then yields 
        \[
            \|\differ_1(m,z)\|_{\cZ'}
            \leq 
            \lip_{Q_u} \| u(m,z) - u_{\theta}(m,z) \|_{\cU}
            \| D_z u(m,z) \|_{\Op(\cZ, \cU)}
        \]
        Then, taking expectations and using the Cauchy--Schwarz inequality, we have 
        \begin{align*}
            \bE_{m\sim\mu_m}[\|\differ_1(m,z)\|_{\cZ'}]
            \leq 
            \lip_{Q_u} 
            \|u(\cdot, z) - u_{\theta}(\cdot, z)\|_{L^2_{\mu_m}}
            \|D_z u(\cdot, z)\|_{L^2_{\mu_m}}
            \leq \lip_{Q_u} \ubDzu \|u(\cdot, z) - u_\theta(\cdot, z)\|_{L^2_{\mu_m}},
        \end{align*}
        where we have made use of the upper bound from \Cref{proposition:lipschitz_u}.

        \item For $\differ_2$, we have 
        \[
            \|\differ_2(m,z)\|_{\cZ'} \leq \|D_u Q(u_{\theta}(m,z), m, z)\|_{\cU'} \|D_z u(m,z) - D_z u_{\theta}(m,z)\|_{\Op(\cZ, \cU)}.
        \]
        We can further make use of the Lipschitz continuity of $D_u Q$ to write 
        \begin{align*}
            \|D_u Q(u_{\theta}(m,z), m, z)\|_{\cU'}
            &\leq
            \|D_u Q(0, m, z) \|_{\cU'}
            +
            \lip_{Q_u} \| u(m,z) - u_{\theta}(m,z) \|_{\cU}
            +
            \lip_{Q_u} \| u(m,z) \|_{\cU} \\
            & \leq 
            \ubDuQ
            +
            \lip_{Q_u} \| u(m,z) - u_{\theta}(m,z) \|_{\cU}
            +
            \lip_{Q_u} \| u(m,z) \|_{\cU}.
        \end{align*}
        Taking expectations, and using Cauchy--Schwarz inequality along with the upper bound $\|u(\cdot,z)\|_{L^2_{\mu_m}} \leq \ubu$ 
        from \Cref{proposition:lipschitz_u}, 
        we also have 
        \begin{align*}
            \bE_{m\sim\mu_m}[\|\differ_2(m,z)\|_{\cZ'}]
            & \leq 
            (\ubDuQ + \lip_{Q_u} \ubu) \| D_z u(\cdot, z) - D_z u_{\theta}(\cdot, z)\|_{L^2_{\mu_m}} \\
            & \qquad + 
            \lip_{Q_u} \|u(\cdot, z) - u_{\theta}(\cdot, z) \|_{L^2_{\mu_m}} \| D_z u(\cdot, z) - D_z u_{\theta}(\cdot, z)\|_{L^2_{\mu_m}}.
        \end{align*}
        Note that by Young's inequality, $ab \leq \tfrac{a^2}{2} + \tfrac{b^2}{2}$, we can write this without cross terms as 
        \begin{align*}
            \bE_{m\sim\mu_m}[\|\differ_2(m,z)\|_{\cZ'}]
            & \leq 
            (\ubDuQ + \lip_{Q_u} \ubu) \| D_z u(\cdot, z) - D_z u_{\theta}(\cdot, z)\|_{L^2_{\mu_m}} \\
            & \qquad + 
            \frac{\lip_{Q_u}}{2} \|u(\cdot, z) - u_{\theta}(\cdot, z) \|_{L^2_{\mu_m}}^2
            + \frac{\lip_{Q_u}}{2}
            \| D_z u(\cdot, z) - D_z u_{\theta}(\cdot, z)\|_{L^2_{\mu_m}}^2.
        \end{align*}

        \item Finally, for $\differ_3$, we directly use \eqref{eq:cor_quadratic_bound_derivative} from \Cref{corollary:other_quadratic_bound},
        \[
            \|\differ_3(m,z)\|_{\cZ'} \leq 
            (\ubDuzQ + \lip_{Q_{uz}} \|u(m,z)\|_{\cU})  \|u(m,z) - u_{\theta}(m,z)\|_{\cU}
            + \frac{\lip_{Q_{uz}}}{2} \|u(m,z) - u_{\theta}(m,z)\|_{\cU}^2.
        \]
        Again, taking expectations and using the Cauchy--Schwarz along with the upper bound from \Cref{proposition:lipschitz_u}, we have
        \[
            \bE_{m \sim \mu_m}[\|\differ_3(m,z)\|_{\cZ'}]
            \leq 
            \left( \ubDuzQ  + \lip_{Q_{uz}} \ubu \right)
            \|u(\cdot, z) - u_{\theta}(\cdot, z)\|_{L^2_{\mu_m}}
            + \frac{\lip_{Q_{uz}}}{2} 
            \|u(\cdot, z) - u_{\theta}(\cdot, z)\|_{L^2_{\mu_m}}^2.
        \]
    \end{enumerate}

    Combining the above, we have 
    \begin{align*}
        \| D_z \cJ^{\mean}(z) - D_z \cJ_{\theta}^{\mean} (z) \|_{\cZ'} 
        & \leq 
        C_1 \left( 
            \|u(\cdot, z) - u_{\theta}(\cdot, z)\|_{L^2_{\mu_m}}
            + \|D_z u(\cdot, z) - D_z u_{\theta}(\cdot, z)\|_{L^2_{\mu_m}}
        \right) \\ 
        & \qquad + C_2 \left ( 
            \|u(\cdot, z) - u_{\theta}(\cdot, z)\|_{L^2_{\mu_m}}^2 
            + \|D_z u(\cdot, z) - D_z u_{\theta}(\cdot, z)\|_{L^2_{\mu_m}}^2
           \right),
    \end{align*}
    where $C_1, C_2 > 0$ are constants depending on $\ubDuQ, \ubDuzQ, \lip_{Q_u}, \lip_{Q_{uz}}, \ubu$ but not $u_{\theta}$.
\end{proof}
    
\section{Universal Approximation Results}
\label{appendix:ua_all}
In this section, we develop the proofs for the universal approximation property of RBNOs in terms of both the outputs and the derivatives (\Cref{theorem:rbno_ua_semibounded}), and its subsequent application to risk neutral optimization problems (\Cref{theorem:rbno_ouu_approx}). In order to prove \Cref{theorem:rbno_ua_semibounded}, we decompose the approximation errors of an RBNO into the dimension reduction errors associated with projecting the inputs and outputs and the neural network approximation error in the resulting latent space. 
We bound each source of the approximation errors separately, beginning with the latent space neural network in \Cref{appendix:ua_nn} followed by the dimension reduction in \Cref{appendix:dimension_red_errors}, before combining them into a proof of \Cref{theorem:rbno_ua_semibounded} in \Cref{appendix:proof_rbno_ua_semibounded}.
The proof of \Cref{theorem:rbno_ouu_approx} then follows from \Cref{theorem:rbno_ua_semibounded} along with \Cref{theorem:optimization_error} and \Cref{prop:surrogate_gradient_error}.

Within this section, we will deviate from the notational conventions introduced in \Cref{remark:geometry_notation} relating to the change of coordinates between spatial and reference domains.
Namely, $x$ will be used to denote generic vectors in $\bR^{d_x}$, and $\widetilde{\cdot}$ will not carry any implications about whether it is a spatial or reference domain quantities.

\subsection{Finite-Dimensional Neural Network Approximation}
\label{appendix:ua_nn}
We first prove a variant of the universal approximation theorem for neural networks in finite dimensions, which we state as follows.

\begin{theorem}[Universal approximation on semibounded sets]\label{theorem:ua_semibounded}
    Let $g \in C^1(\mathbb{R}^{d_x} \times \openset; \bR^{d_g})$, where $\openset \subset \bR^{d_y}$ is open 
    and let $\mu_x$ be a Borel probability measure on $\bR^{d_x}$.
    Suppose that for some $p \geq 1$,
    \begin{enumerate}
        \item the mapping $\openset \ni y \mapsto g(\cdot, y) \in L^p_{\mu_x}(\bR^{d_x}; \bR^{d_g})$ is continuous, and
        \item the mapping $\openset \ni y \mapsto D_y g(\cdot, y) \in L^p_{\mu_x}(\bR^{d_x}; \bR^{d_g \times d_y})$ is continuous, where $\bR^{d_g \times d_y}$ is equipped with the Frobenius norm $\|\cdot\|_{F}$.
    \end{enumerate}
    Then for any $\epsilon > 0$ and any compact set $K_y \subset \openset$, 
    there exists a GELU neural network $g_{\theta}$ such that 
    \[
        \sup_{y \in K_y} \mathbb{E}_{x \sim \mu_x} \left[ \| g(x,y)- g_{\theta}(x,y) \|_2^p \right] + \mathbb{E}_{x \sim \mu_x} \left[ \| D_y g(x,y)- D_y g_{\theta}(x,y) \|_F^p \right] < \epsilon^p.
    \]
\end{theorem}

The result differs from conventional universal approximation results in two ways. 
First, the approximation accuracy considers both the output and the derivative error. 
Second, instead of considering approximation errors as either uniform over compact sets or integral quantities over distributions (with potentially unbounded support), 
we decompose the inputs into two sets of variables, here denoted $x$ and $y$, where integration is performed over the former while the supremum is taken over the latter.
This anticipates the differing roles of the two inputs for the solution operator $(m,z) \mapsto u(m,z)$, where integration is with respect to $m$ while optimization is with respect to $z$.

In this work, we prove \Cref{theorem:ua_semibounded} following a cutoff argument used in \cite{yao2025derivative}. 
That is, we use conventional universal approximation results of feedforward neural networks over compact sets to construct a neural network $\tilde{g}$ with high accuracy over a compact subset $K_x \times K_y \subset \bR^{d_x} \times \openset$.
We then construct a cutoff function in the form of a second neural network such that when applied to $\tilde{g}$, controls the approximation error outside of the compact set $K_x \times K_y$.
Finally, we show that this construction can be taken to the limit as $K_x$ approaches the entire $\bR^{d_x}$, allowing the $L^p_{\mu_x}$ error over $\bR^{d_x}$ to be made arbitrarily small uniformly over $y \in K_y$.

We present proof of \Cref{theorem:ua_semibounded} follows. 
First, we show how an appropriate cutoff function can be constructed as a neural network (\Cref{lemma:ell1,lemma:cutoff,lemma:cutoff_ell2}).
We then show how the cutoff function can be used in combination with standard universal approximation results to construct approximations of $g$ that are accurate within a compact set while having controlled growth rates outside of the compact set (\Cref{lemma:bounded_compact,lemma:bounded_approx_ua}).
Finally, in Section \ref{sec:completing_the_proof}, we complete the proof for \Cref{theorem:ua_semibounded} by combining the aforementioned results through the limiting argument taking $K_x$ to $\bR^{d_x}$.

\subsubsection{Neural Network Constructions}
\label{appendix:nn_constructions}
In order to construct the cutoff function, we first require an approximation of the norm function $\|\cdot\|$ as a neural network.
As shown in \cite[Lemma D.4]{yao2025derivative}, we can use GELU activation functions to construct an approximation of the absolute value function $|\cdot|$.
We summarize this result as follows.
\begin{lemma}[Absolute value function]\label{lemma:absval}
    Given $\theta > 0$, define for $x \in \mathbb{R}$
    \begin{equation}
        \fabs (x; \theta) := \frac{ \GELU(\theta x) + \GELU(-\theta x) }{\theta},
    \end{equation}
    Then, for any $\epsilon > 0$, there exists $\theta > 0$ such that 
    \begin{equation}
        \left| \fabs(x;\theta) - |x| \right| \leq \epsilon
    \end{equation}
    and $\fabs(x;\theta) \geq 0$ for all $x \in \mathbb{R}$.    
\end{lemma}

We can then use this to construct an approximation to the $\ell_1$ norm as follows.
\begin{lemma}[$\ell_1$ norm]\label{lemma:ell1}
        Given $\theta > 0$, define for $x \in \mathbb{R}^{d_x}$
    \begin{equation}
        \fnorm (x; \theta) := \ones^T \frac{ \GELU(\theta x) + \GELU(-\theta x) }{\theta},
    \end{equation}
    where $\mathbf{1} := (1, \dots, 1) \in \mathbb{R}^{d_x}$.
    Then, for any $\epsilon > 0$, there exists $\theta > 0$ such that 
    \begin{equation}
        \left| \fnorm(x;\theta) - \|x\|_1 \right| \leq \epsilon
    \end{equation}
    and $\fnorm(x;\theta) \geq 0$ for all $x \in \mathbb{R}^{d_x}$.
\end{lemma}

\begin{proof}
    As shown in \Cref{lemma:absval}, the GELU activation function can be used to construct an approximation of the absolute value function on $\mathbb{R}$. 
    Then, for any $\epsilon$, there exists $\theta > 0$ such that 
    \begin{equation}
        \fabs(t; \theta) := \frac{\GELU(\theta t) + \GELU(-\theta t)}{\theta},
    \end{equation}
    satisfies
    $\fabs(t;\theta) \geq 0$ and $|\fabs(t; \theta) - |t|| \leq \epsilon/d_x$ for all $t \in \mathbb{R}$.
    This implies that for $x \in \mathbb{R}^{d_x}$, we have $\fnorm(x;\theta) \geq 0$ and 
    \[
        | \fnorm(x;\theta) - \|x\|_1 | \leq \sum_{j=1}^{d_x} \left| \fabs(x_j; \theta) - |x_j| \right| \leq \epsilon.
    \]
\end{proof}

The GELU activation function can also be used to construct a cutoff function $\tilde{f}_{\cutoff}$ on $\bR$, i.e., $\tilde{f}_{\cutoff}(t) \approx 1$ for $t \leq R$ while $\tilde{f}_{\cutoff}(t) \approx 0$ for $t \gg R$.  
The following is a simplified statement of \cite[Lemma D.5]{yao2025derivative}.
\begin{lemma}[Cutoff function]\label{lemma:cutoff}
Given $\theta, b > 0$, define for $x \in \mathbb{R}$
\begin{equation}\label{eq:cutoff}
    \fcutoff(x; \theta, b) :=  \frac{\GELU(\theta(x-b-1)) - \GELU(\theta(x-b))}{\theta}+1.
\end{equation}
Then, for any $\epsilon > 0$ and $R > 1$, 
there exist $\theta, b > 0$ such that the following hold:
\begin{equation}
    \begin{cases}
        |\fcutoff(x; \theta, b) - 1 | \leq \epsilon & 0 \leq x \leq R, \\
        |\fcutoff(x; \theta, b) | \leq \epsilon & x \geq 4R, \\
    \end{cases}
\end{equation}
    Additionally, there exists $\ub > 0$ independent of $\theta, b$ such that 
    $| \fcutoff(x; \theta, b)| \leq \ub$ for all $x \in \bR$.
\end{lemma}

\Cref{lemma:ell1,lemma:cutoff} can then be used to construct a cutoff function based on the $\ell_2$ norm of the input. 
\begin{lemma}[Cutoff function using $\ell_2$ norm]\label{lemma:cutoff_ell2}
    There exists $\ub > 0$ such that 
    for any $\epsilon > 0$ and $R > 1$, 
    there exists a two-layer GELU neural network $\funcut : \mathbb{R}^{d_x} \rightarrow \mathbb{R}$ satisfying 
    \begin{equation}
        \begin{cases}
            |\funcut(x) - 1 | \leq \epsilon & 0 \leq \|x\|_{2} \leq R, \\
            |\funcut(x)| \leq \epsilon & \|x\|_2 \geq \scalerad (\sqrt{d_x} + 1) R, \\
            |\funcut(x) | \leq \ub & x \in \mathbb{R}^{d_x}.
        \end{cases} 
    \end{equation}
     
\end{lemma}
\begin{proof}
    We define $\funcut(x) = \funcut(x; \theta_{\ell_1}, \theta_{\cutoff}, b_{\cutoff})$ as the composition of $\fnorm$ and $\fcutoff$, i.e.,
    \begin{equation}
        \funcut(x; \theta_{\ell_1}, \theta_{\cutoff}, b_{\cutoff}) = \fcutoff(\fnorm(x; \theta_{\ell_1}); \theta_{\cutoff}, b_{\cutoff}).
    \end{equation}
    We will choose the cutoff radius for $\fcutoff$ to account for the conversion for $\ell_1$ to $\ell_2$, i.e., $\|x\|_2 \leq \|x\|_{1} \leq \sqrt{d_x} \|x\|_2$.
    
    To this end, let $\epsilon, R$ be given. We construct $\fnorm$ to satisfy 
    \[
        | \fnorm(x) - \|x\|_1 |  \leq R, \qquad \forall x \in \bR^{d_x},
    \]
    and $\fcutoff$ to satisfy
    \[
    \begin{cases}
        |\fcutoff(x; \theta, b) - 1 | \leq \epsilon & 0 \leq x \leq \tilde{R}, \\
        |\fcutoff(x; \theta, b) | \leq \epsilon & x \geq 4\tilde{R}, \\
        |\fcutoff(x; \theta, b) | \leq \ub & x \in \bR^{d_x}, \\
    \end{cases}
    \]
    where $\tilde{R} := (\sqrt{d_x} + 1) R$.

    We then consider the following cases.
    \begin{enumerate}
        \item $\|x\|_2 \leq R$. We have 
        \[ 
            \fnorm(x) \leq \|x\|_1 + R \leq (\sqrt{d_x} + 1)R = \tilde{R},
        \]
        implying that $| \fcutoff(\fnorm(x)) - 1 | \leq \epsilon$.
        \item $\|x\|_2 \geq \scalerad(\sqrt{d_x} + 1)R$. We have 
        \[
            \fnorm(x) \geq \|x\|_1 - R \geq \scalerad(\sqrt{d_x} + 1) R - R \geq 4 \tilde{R},
        \]
        implying that $| \fcutoff(\fnorm(x)) | \leq \epsilon$.
    \end{enumerate}
    Finally, $|\funcut(x)| \leq \ub$ holds for any $x \in \mathbb{R}^{d_x}$ with the same $\ub$ from \Cref{lemma:cutoff}.
    Since, both $\fcutoff$ and $\fnorm$ can be interpreted as single-layered neural networks, their resulting composition $\funcut(x)$ is a two-layered neural network.
    
\end{proof}

\subsubsection{Neural Network Approximations with Controlled Growth}

Universal approximation theorems, such as \cite[Theorem 4.1]{Pinkus99}
show that single-layer neural networks are able to approximate continuously differentiable functions arbitrarily-well in terms of both their output value and their derivatives over compact input domains.
As a corollary, we can show that for GELU activation functions, the resulting approximant can be constructed to have globally bounded outputs and derivatives despite the fact that GELU itself is unbounded.

\begin{lemma}[Bounded approximation on compact sets]\label{lemma:bounded_compact}
Suppose $g \in C^1(\mathbb{R}^{d_x} \times \openset; \bR^{d_g})$
where $\openset \subset \bR^{d_y}$ is open. 
Then, for any $\epsilon > 0$ and compact sets $K_x \subset \mathbb{R}^{d_x}$ and $K_y \subset \openset$,
there exists a single-layer GELU neural network $g_{\theta}$ such that 
simultaneously
\begin{equation}
    \sup_{(x,y) \in K_x \times K_y} \| g(x,y) - g_{\theta}(x,y) \|_2 \leq \epsilon,
\end{equation}
and 
\begin{equation}
    \sup_{(x,y) \in K_x \times K_y} \| D_y g(x,y) - D_y g_{\theta}(x,y) \|_F \leq \epsilon.
\end{equation}
Moreover, $g_{\theta}$ and its derivatives are bounded, i.e., there exists $\ub = M(\epsilon) > 0$ such that simultaneously
\begin{equation}
    \sup_{(x,y) \in \mathbb{R}^{d_x} \times \mathbb{R}^{d_y}} \|  g_{\theta}(x,y) \|_{2} \leq \ub
\end{equation}
and
\begin{equation}
    \sup_{(x,y) \in \mathbb{R}^{d_x} \times \mathbb{R}^{d_y}} \| D_y g_{\theta}(x,y) \|_{F} \leq \ub.
\end{equation}
\end{lemma}

\begin{proof}
    First, consider the function 
    $\sigma(x) = \GELU(x + 1) - \GELU(x-1)$.
    It can be shown that $\sigma \in C^{\infty}(\mathbb{R})$ and 
    $\sup_{x \in \mathbb{R}} |\sigma^{(k)}(x)| < \infty$ for all $k \in \{0\} \cup \bN$. That is, $\sigma$ is a smooth, non-constant function with bounded output and derivatives. 

    Now, let $\xi := (x,y)$ and $K := K_x \times K_y$. 
    By standard universal approximation theorems, such as \cite[Theorem 4.1]{Pinkus99}, 
    there exists a single-layered neural network of the form 
    \begin{equation}
        g_{\theta}(\xi) = {W}_2 \sigma({W}_1 \xi + {b}_1) + {b}_2
    \end{equation}
    such that simultaneously
    \begin{equation}
        \sup_{\xi \in K} \| g(\xi) - g_{\theta}(\xi) \|_2 \leq \epsilon
    \end{equation}
    and 
    \begin{equation}
        \sup_{\xi \in K} \| D_\xi g(\xi) - D_\xi g_{\theta}(\xi) \|_F \leq \epsilon.
    \end{equation}
    We recognize that $g_{\theta}(\xi)$ can instead be written as 
    \begin{equation}
        g_{\theta} = {W}_2 ( \GELU({W}_1 \xi + {b}_1 + 1) - \GELU({W}_1 \xi + {b}_1 - 1) ) + {b}_2,
    \end{equation}
    which can be thought of as a single-layer GELU neural network. Moreover, since $\sigma$ has bounded values and derivatives, we have that $D_\xi g_{\theta}(\xi)$ is bounded for all $\xi \in \bR^{d_x} \times \bR^{d_y}.$
    Finally, it suffices to recognize that 
    \[
        \| D_y g(x,y) - D_{y} g_{\theta}(x,y) \|_{F} \leq 
        \| D_\xi g(\xi) - D_\xi g_{\theta}(\xi) \|_{F}.
    \]
    
\end{proof}
\begin{remark}
    Within the proof above, we have referenced universal approximation theorems that consider the approximation of functions defined on the entire input space,
    e.g., $g \in C^1(\bR^{d_x} \times \bR^{d_y}; \bR^{d_g})$, while we have only assumed that $g$ is defined on the open subset $\bR^{d_x} \times \openset$.
    This is largely inconsequential for the following reason. 
    Since $K \Subset \bR^{d_x} \times \openset$, where $\Subset$ denotes a compact embedding,
    there exist open sets $\openset_1, \openset_2$ such that $K \Subset \openset_1 \Subset \openset_2 \Subset \openset$. 
    Moreover, for the closure $\mathrm{cl}(\openset_1)$, we consider a smooth external approximation to the indicator function, 
    $\mathbbm{1}_{\mathrm{cl}(\openset_1)} \in C^{\infty}(\bR^{d_x} \times \bR^{d_y};\bR)$, 
    such that $\mathbbm{1}_{\mathrm{cl}(\openset_1)}|_{\openset_1} = 1$ and $\mathbbm{1}_{\mathrm{cl}(\openset_1)}|_{\openset_2^c} = 0$.
    This can be used to define $\tilde{g}(x,y) = \mathbbm{1}_{\mathrm{cl}(\openset_1)}(x,y) g(x,y)$, which we extend by $0$ to $y \notin \openset$.
    The extension $\tilde{g}$ is therefore in $C^1(\bR^{d_x} \times \bR^{d_y}; \bR^{d_g})$ and its outputs and derivatives coincide with $g$ over the set $K$.
    It suffices then to consider a neural network approximation $g_{\theta}$ of $\tilde{g}$ instead of $g$ when using the conventional universal approximation theorems.
\end{remark}

\begin{lemma}[Bounded approximation on semi-bounded sets]\label{lemma:bounded_approx_ua}
    Suppose $g \in C^1(\mathbb{R}^{d_x} \times \openset; \bR^{d_g})$ 
    where $\openset \subset \bR^{d_y}$ is open. Then, for any compact subset $K_y \subset \openset$, $\epsilon > 0, R > 1$, and $p \geq 1$, 
    there exists a three-layer GELU neural network $g_{\theta}$ such that simultaneously
    \begin{equation}\label{eq:function_error_inside}
        \sup_{\|x\|_2 \leq R, y \in K_y} \|g(x,y) - g_{\theta}(x,y) \|_2 \leq \epsilon,
    \end{equation}
    and 
    \begin{equation}\label{eq:derivative_error_inside}
        \sup_{\|x\|_2 \leq R, y \in K_y} \|D_y g(x,y) - D_y g_{\theta}(x,y) \|_F \leq \epsilon.
    \end{equation}
    Moreover, $g_{\theta}$ can be chosen such that 
    \begin{equation}\label{eq:function_growth_rate}
        \|g(x,y) - g_{\theta}(x,y) \|_2^p \leq \const_p (\|g(x,y)\|_2^p + 1)
    \end{equation}
    and 
    \begin{equation}\label{eq:derivative_growth_rate}
        \|D_y g(x,y) - D_y g_{\theta}(x,y) \|_F^p \leq \const_p (\|D_y g(x,y)\|_F^p + 1)
    \end{equation}
    for all $x \in \mathbb{R}^{d_x}$ and $y \in K_y$,
    where $\const_p > 0$ is a fixed constant that does not depend on $\epsilon$ or $R$.

\end{lemma}

\begin{proof}
    Our strategy is as follows. We first use \Cref{lemma:bounded_compact} to construct a bounded approximation $\tilde{g}_{\theta}$ to $g$ on a compact set. We then control the growth of the approximation $\tilde{g}_{\theta}$ in the $x$ direction by (approximately) multiplying it by a cutoff function $\fcutoff$.

    To this end, let $g, \epsilon, R$ be given. 
    For any $\epsilon_{g} > 0$,
    \Cref{lemma:bounded_compact} allows us to construct $\tilde{g}_{\theta}$, a bounded approximation of $g$ over the compact sets $K_x := \{ \|x\|_2 \leq 4(\sqrt{d_x} + 1)R \}$ and $K_y$, such that 
    \[
         \sup_{(x,y) \in K_x \times K_y} \|g(x,y) - \tilde{g}_{\theta}(x,y) \|_{2} \leq \epsilon_{g}, 
         \qquad
         \sup_{(x,y) \in K_x \times K_y} \|D_y g(x,y) - D_y \tilde{g}_{\theta}(x,y) \|_{F} \leq \epsilon_{g},
    \]
    and 
    \[
         \sup_{(x,y) \in \mathbb{R}^{d_x} \times \mathbb{R}^{d_y}} \|\tilde{g}_{\theta}(x,y)\|_{2} \leq \ub_{g}, \qquad 
         \sup_{(x,y) \in \mathbb{R}^{d_x} \times \mathbb{R}^{d_y}} \|D_y \tilde{g}_{\theta}(x,y)\|_{F} \leq \ub_{g},
    \]
    noting that $\ub_g = \ub_g(\epsilon_g)$.
    For an arbitrary $\epsilon_{\cutoff} > 0$,
    we also consider the cutoff function $\funcut(x)$ from \Cref{lemma:cutoff_ell2} satisfying 
    \[
        \begin{cases}
            |\funcut(x) - 1| \leq \epsilon_{\cutoff} & \|x\|_2 \leq R  \\
            |\funcut(x)| \leq \epsilon_{\cutoff} & \|x\|_2 \geq \scalerad(\sqrt{d_x} + 1) R  \\
            |\funcut(x)| \leq \ub_{\cutoff} & x \in \mathbb{R}^{d_x},
        \end{cases}
    \]
    where $\ub_{\cutoff}$ does not depend on $\epsilon_{\cutoff}$ or $R$.
    Finally, 
    we consider a single-layered neural network approximation $\ftimes$ of the multiplication function $\times(t_1, t_2) = t_1 t_2$, such that
    \[
        \| \ftimes - \times(\cdot, \cdot) \|_{C^{1}(B)} \leq \epsilon_{\times},
    \]
    for an arbitrary $\epsilon_{\times} > 0$, where 
    $B  = [-\ub_{\cutoff}, \ub_{\cutoff}] \times [-\ub_{g}, \ub_{g}]$.
    This is again possible by standard universal approximation results \cite{Pinkus99}.
    We then take 
    \[
        g_{\theta}(x,y) := \ftimes (\funcut(x), \tilde{g}_{\theta}(x,y)) \approx \funcut(x) \tilde{g}_{\theta}(x,y),
    \]
    which is a three-layer neural network since $\fcutoff$ is a two-layer neural network.
    This has the derivative 
    \[
        D_y g_{\theta}(x,y) = D_2 \ftimes (\funcut(x), \tilde{g}_{\theta}(x,y)) D_y \tilde{g}_{\theta}(x,y),
    \]
    where $D_2$ refers to partial differentiation with respect to the second variable.
    Note that this construction where $\ftimes$ is chosen based on the output range of $\funcut(x)$ and $\tilde{g}$ ensures that 
    \[
        \|\funcut(x) \tilde{g}(x,y) - \ftimes(\funcut(x), \tilde{g}(x,y)) \|_2 \leq \epsilon_{\times}
    \]
    and 
    \[
        \| \funcut(x) I - D_2 \ftimes(\funcut(x), \tilde{g}(x,y)) \|_{F} \leq \epsilon_{\times}
    \]
    for all $x \in \bR^{d_x}$ and $y \in \bR^{d_y}$.

    We now consider the following cases for $x \in \bR^{d_x}$ given that $y \in K_y$
    \begin{enumerate}
        \item $\|x\|_2 \leq R$. By the triangle inequality, we have 
        \begin{align*}
            & \|g(x,y) - g_{\theta}(x,y) \|_2   \\
            & \qquad \leq 
            \|g(x,y) - \tilde{g}_{\theta}(x,y) \|_2  
            +  \|(1 - \funcut(x)) \tilde{g}_{\theta}(x,y) \|_2  
            +  \|\funcut(x) \tilde{g}_{\theta}(x,y) - \ftimes(\funcut(x), \tilde{g}_{\theta}(x,y)) \|_2.
        \end{align*}
        Since $\|x\|_2 \leq R$, we have $\|g(x,y) - g_{\theta}(x,y)\|_2 \leq \epsilon_g$
        and $|1 - \funcut(x)| \leq \epsilon_{\cutoff}$.
        Thus, 
        \[
            \|g(x,y) - g_{\theta}(x,y) \|_2 
            \leq \epsilon_g + \epsilon_{\cutoff} \ub_g + \epsilon_{\times}.
        \]
        By analogous reasoning, we have 
        \begin{align*}
            \|D_y g(x,y) - D_y g_{\theta}(x,y) \|_F   
            & \leq  
            \|D_y g(x,y) - D_y \tilde{g}_{\theta}(x,y) \|_F  \\
            &  \qquad +  \|(1 - \funcut(x)) D_y \tilde{g}_{\theta}(x,y) \|_F   \\
            &  \qquad +  \|\funcut(x) D_y \tilde{g}_{\theta}(x,y) - D_2 \ftimes(\funcut(x), \tilde{g}_{\theta}(x,y)) D_y \tilde{g}_{\theta}(x,y) \|_F \\
            & \leq \epsilon_{g} + \epsilon_{\cutoff} \ub_{g} + \epsilon_{\times} \ub_g.
        \end{align*}

        \item $\|x\|_2 \geq \scalerad (\sqrt{d_x} + 1) R$. We have 
        \begin{align*}
            \| g(x,y) - g_{\theta}(x,y) \|_{2} &\leq 
            \| g(x,y) \|_2  + \| g_{\theta}(x,y) \|_{2} \\
            & \leq \|g(x,y)\|_{2} 
                + |\funcut(x)| \| \tilde{g}_{\theta}(x,y) \|_2
                + \|\ftimes(\funcut(x), \tilde{g}_{\theta}(x,y)) - \funcut(x) \tilde{g}_{\theta}(x,y)\|_2 \\
            & \leq \|g(x,y)\|_2 + \epsilon_{\cutoff} \ub_g + \epsilon_{\times}
        \end{align*}
        and 
        \begin{align*}
            & \|D_y g(x,y) - D_y g_{\theta}(x,y) \|_F   
            \leq   \|D_y g(x,y) \|_F + \|D_y g_{\theta}(x,y) \|_F \\ 
            & \leq   \|D_y g(x,y) \|_F + \|\funcut(x) D_y \tilde{g}_{\theta}(x,y)\|_{F} 
            + \|\funcut(x) D_y \tilde{g}_{\theta}(x,y) - D_2 \ftimes(\funcut(x), \tilde{g}_{\theta}(x,y)) D_y \tilde{g}_{\theta}(x,y) \|_F \\
            & \leq \| D_y g(x,y) \|_F + \epsilon_{\cutoff} \ub_{g} + \epsilon_{\times} \ub_g.
        \end{align*}

        \item $\|x\|_2 < \scalerad(\sqrt{d_x} + 1) R$.  
        We have 
        \begin{align*}
            \| g(x,y) - g_{\theta}(x,y) \|_{2} &\leq 
            \| g(x,y) \|_2  + \| g_{\theta}(x,y) \|_{2} \\
            & \leq \|g(x,y)\|_{2} 
                + |\funcut(x)| \| \tilde{g}_{\theta}(x,y) \|_2
                + \|\ftimes(\funcut(x), \tilde{g}_{\theta}(x,y)) - \funcut(x) \tilde{g}_{\theta}(x,y)\|_2\\
            & \leq \|g(x,y)\|_2 + \ub_{\cutoff} (\|g(x,y)\|_{2} + \epsilon_{g}) + \epsilon_{\times},
        \end{align*}
        and 
        \begin{align*}
            & \|D_y g(x,y) - D_y g_{\theta}(x,y) \|_F   
            \leq   \|D_y g(x,y) \|_F + \|D_y g_{\theta}(x,y) \|_F \\ 
            & \leq   \|D_y g(x,y) \|_F + \|\funcut(x) D_y \tilde{g}_{\theta}(x,y)\|_{F} 
            + \|\funcut(x) D_y \tilde{g}_{\theta}(x,y) - D_2 \ftimes(\funcut(x), \tilde{g}_{\theta}(x,y)) D_y \tilde{g}_{\theta}(x,y) \|_F \\
            & \leq \| D_y g(x,y) \|_F + \ub_{\cutoff}(\|D_{y}g (x,y)\|_{F} + \epsilon_{g}) + \epsilon_{\times} \ub_g.
        \end{align*}
    \end{enumerate}
    Thus, it suffices to choose the approximations with the error tolerances 
    \[
        \epsilon_{g} \leq \min \left\{ \frac{\epsilon}{3}, \frac{1}{2 \ub_{\cutoff}} \right\}, 
        \qquad 
        \epsilon_{\cutoff} \leq \min \left\{ \frac{\epsilon}{3 \ub_g}, \frac{1}{2{\ub_g}} \right\}, 
        \qquad
        \epsilon_{\times} \leq \min \left\{ \frac{\epsilon}{3}, \frac{\epsilon}{3\ub_g}, \frac{1}{2}, \frac{1}{2\ub_g} \right\}.
    \]
    Under these choices, the errors satisfy \eqref{eq:function_error_inside}--\eqref{eq:derivative_growth_rate} with $p=1$ using the constant $\const_1 := 1 + \ub_{\cutoff}$.
    Now, for arbitrary $p \geq 1$, we have 
    \[
        \|g(x,y) - g_{\theta}(x,y)\|_{2}^p \leq \const_1^p \left(\|g(x,y)\|_2 + 1 \right)^p \leq 2^{p-1}\const_1^p \left(\|g(x,y)\|_2^p + 1 \right),
    \]
    where we have made use of the inequality $(a+b)^p \leq 2^{p-1} (a^p + b^p)$.
    Thus, for any $p \geq 1$, \eqref{eq:function_growth_rate} and \eqref{eq:derivative_growth_rate} hold with $\const_p = 2^{p-1}\const_1^p$.
\end{proof}

\subsubsection{Proof of Theorem \ref{theorem:ua_semibounded}}\label{sec:completing_the_proof}
We can now complete the proof of \Cref{theorem:ua_semibounded}.
\begin{proof}
    We consider neural network approximations $g_{\theta}$ that belong to the family of approximations satisfying \Cref{lemma:bounded_approx_ua}.
    Given such a $g_{\theta}$, consider the error for a particular $y \in K_y$, i.e.,
    \begin{align*}
        \bE_{x \sim \mu_x}[\|g(x,y) - g_{\theta}(x,y)\|_2^p]  
            &= \int \|g(x,y) - g_{\theta}(x,y)\|_2^p d\mu_x(x), \\
        \bE_{x \sim \mu_x}[\|D_y g(x,y) - D_y g_{\theta}(x,y)\|_F^p]  
            &= \int \|D_y g(x,y) - D_y  g_{\theta}(x,y)\|_F^p d\mu_x(x).
    \end{align*}
    For any $R > 1$, let $B_R := \{ \|x\|_2 \leq R \}$ and $B_R^{c} := \{ \|x\|_2 > R\}$.
    We can decompose the integrals into 
    \begin{align*}
        \int \|g(x,y) - g_{\theta}(x,y)\|_2^p d\mu_x(x) 
        = \int_{B_R} \|g(x,y) - g_{\theta}(x,y)\|_2^p d\mu_x(x)
            + \int_{B_R^{c}} \|g(x,y) - g_{\theta}(x,y)\|_2^p d\mu_x(x)
    \end{align*}
    and 
    \begin{align*}
        \int \|D_y g(x,y) - D_y g_{\theta}(x,y)\|_F^p d\mu_x(x),
        &= \int_{B_R} \|D_y g(x,y) - D_y g_{\theta}(x,y)\|_F^p d\mu_x(x)  \\
            &\quad + \int_{B_R^c} \|D_y g(x,y) - D_y g_{\theta}(x,y)\|_F^p d\mu_x(x) .
    \end{align*}
    In particular, for $g_{\theta}$ coming from \Cref{lemma:bounded_approx_ua}, we can use the global bound of the error to write this as
    \begin{align}\label{eq:decomposed_l2_error}
        \int \|g(x,y) - g_{\theta}(x,y)\|_2^p d\mu_x(x) 
        \leq \int_{B_R} \|g(x,y) - g_{\theta}(x,y)\|_2^p d\mu_x(x)
            + \const_p \int_{B_R^c} \left( \|g(x,y)\|_{2}^p + 1\right) d\mu_x(x)
    \end{align}
    and 
    \begin{align}\label{eq:decomposed_h1_error}
        \int \|D_y g(x,y) - D_y g_{\theta}(x,y)\|_F^p d\mu_x(x)
        & \leq  \int_{B_R} \|D_y g(x,y) - D_y g_{\theta}(x,y)\|_F^p d\mu_x(x) \nonumber \\
                & \qquad + \const_p \int_{B_R^{c}} \left( \|D_y g(x,y)\|_{F}^p + 1\right) d\mu_x(x),
    \end{align}
    where $\const_p$ is independent of the particular choice of $g_{\theta}$. 

    The approximation strategy is as follows. We first choose $R > 1$ to be sufficiently large such that the errors
    \begin{align*}
        e_R^{0}(y) &:= \const_p \int_{B_R^c} \left( \|g(x,y)\|_{2}^p + 1 \right) d\mu_x(x) \\
        e_R^{1}(y) &:= \const_p \int_{B_R^c} \left( \| D_y g(x,y)\|_{F}^p + 1 \right) d\mu_x(x)
    \end{align*}
    are uniformly small for $y \in K_y$.
    Subsequently, for the selected $R$, we choose $g_{\theta}$ to accurately approximate $g$ on the compact set $B_R \times K_y$ such that the overall error is bounded by $\epsilon^p$.

    We begin by showing that $e_R^{0}(y)$ and $e_R^{1}(y)$ converge uniformly to zero over $K_y$ as $R \rightarrow \infty$.
    Consider first the output error, $e_R^{0}(y)$.
    By the dominated convergence theorem, we have that $e_R^{0}(y) \downarrow 0$ as $R \rightarrow \infty$ for all $y \in \openset$,
    where $\downarrow$ indicates that the convergence is monotonic, i.e., $e_R^{0} \rightarrow 0$ and $e_{R_1}^{0}(y) \geq e_{R_2}^{(0)}(y)$ whenever $R_1 \leq R_2$.
    Moreover, $e_R^{0}(y)$ is continuous. 
    To see this, we begin by noting that $e_R^{0}(y)$ can be decomposed into two integrals,
    \[
        e_R^{0}(y) = C_p \int_{B_R^c} \|g(x,y)\|_2^p d\mu_x(x)
                    + C_p \int_{B_R^c} 1 d\mu_x(x).
    \]
    The term $C_p \int_{B_R^c} 1 d\mu_x(x) = \const_p\mu_x(B_R^c)$ is clearly continuous with respect to $y$. 
    We therefore focus on the continuity of
    \begin{equation}\label{eq:l2_norm_residual_mapping}
        y \mapsto \const_p \int_{B_R^c} \|g(x,y)\|_2^p d\mu_x(x) = \const_p \| \mathbbm{1}_{B_R^c} g(\cdot, y)\|_{L^p_{\mu_x}}^p,
    \end{equation}
    where $\mathbbm{1}_{B_R^c}$ denotes the indicator function for the set $B_R^c$.
    To this end, we consider $y_1, y_2 \in \openset$.
    The reverse triangle inequality yields
    \[
        \left| 
            \|\mathbbm{1}_{B_R^c}(\cdot) g(\cdot, y_2)\|_{L^p_{\mu_x}} - 
            \|\mathbbm{1}_{B_R^c}(\cdot) g(\cdot, y_1)\|_{L^p_{\mu_x}} 
        \right| 
        \leq 
        \| \mathbbm{1}_{B_R^c}(\cdot) \left( g(\cdot, y_2) - g(\cdot, y_1)\right) \|_{L^p_{\mu_x}}.
        \leq 
        \| g(\cdot, y_2) - g(\cdot, y_1) \|_{L^p_{\mu_x}}.
    \]
    From the assumed continuity of the mapping $\openset \ni y \mapsto g(\cdot, y_2) \in L^p_{\mu_x}$, we have that 
    \[
            \|g(\cdot, y_2) - g(\cdot, y_1)\|_{L^p_{\mu_x}} 
        \rightarrow 0  \qquad \text{as } y_2 \rightarrow y_1.
    \]
    Hence, the mapping $\openset \ni y \mapsto \mathbbm{1}_{B_R^c} g(\cdot, y) \in L^p_{\mu_x}$ is continuous, and by extension, so is the mapping \eqref{eq:l2_norm_residual_mapping}.
    Dini's theorem then implies that the convergence of $e_R^0(y) \downarrow 0$ is uniform over $y \in K_y \Subset \openset$.
    Using analogous reasoning, we can also show that the same is true of $e_R^{1}(y)$.

    Now, let $\epsilon > 0$ be arbitrary. 
    By the above results, there exists an $R > 1$ such that 
    \begin{equation}\label{eq:outside_ball_error}
        \sup_{y \in K_y} e_R^{0}(y) + e_R^{1}(y) \leq \frac{\epsilon^p}{2}.
    \end{equation}
    For this choice of $R > 1$, \Cref{lemma:bounded_approx_ua} yields the existence of a $g_{\theta}$ such that 
    \begin{equation}\label{eq:inside_ball_error}
        \sup_{x \in B_R, y \in K_y} \|g(x,y) - g_{\theta}(x,y)\|_2^p
            + \|D_y g(x,y) - D_y g_{\theta}(x,y)\|_F^p 
            \leq \frac{\epsilon^p}{2}.
    \end{equation}
    Substituting \eqref{eq:inside_ball_error} and \eqref{eq:outside_ball_error} into \eqref{eq:decomposed_l2_error} and \eqref{eq:decomposed_h1_error} yields
    \[
        \int \|g(x,y) - g_{\theta}(x,y)\|_2^p + \|D_y g(x,y) - D_y g_{\theta}(x,y)\|_{F}^p d\mu_x(x) 
        \leq \epsilon^p
    \]
    for every $y \in K_y$, which is the desired result.
    
\end{proof}
\subsection{Dimension Reduction Errors for Lipschitz Smooth Operators}
\label{appendix:dimension_red_errors}
Recall that we are interested in the approximation capabilities of the reduced basis architecture
\[
    u_{\theta}(m,z) = \mathbf{\Phi}_{\rru} g_{\theta}( \mathbf{\Psi}_{\rrm}^* (m- \shiftm), z) + \shiftu,
\]
where $\mathbf{\Phi}_{\rru}$ and $\mathbf{\Psi}_{\rrm}$
refer to the decoding onto the
reduced orthonormal bases $\{\phi_i\}_{i=1}^{\rru}$ and $\{\psi_i \}_{i=1}^{\rrm}$ of $\cU$ and $\cM$, respectively.
We first show that for any orthonormal bases $\{ \phi_i \}_{i=1}^{\infty}$ and $\{ \psi_i \}_{i=1}^{\infty}$ 
dimension reduction errors can be made arbitrarily small by choosing sufficiently large $\rru$ and $\rrm$.
We begin with the output case.
\begin{lemma}[Output dimensionality reduction error]\label{lemma:output_truncation_lipschitz}
    Suppose that $u$ satisfies \Cref{assumption:pde_properties} with the Borel probability measure $\mu_m$, compact admissible set $\cZ_{ad}$, and open set $\openset \supset \cZ_{ad}$ for some $p \geq 2$. 
    Let $\{ \phi_i \}_{i=1}^{\infty}$ be an orthonormal basis of $\cU$ and $\projru{r} = \sum_{i=1}^{r} \phi_i \phi_i^T$ be the orthogonal projection onto the reduced basis $\{\phi_i\}_{i=1}^{r}$.
    Then, for any $\epsilon > 0$, 
    there exists $r_{\epsilon} \in \bN$ such that simultaneously 
    \[
        \sup_{z \in \cZ_{ad}} \bE_{m \sim \mu_m} [ \|(I - \projru{r}) u(m,z)\|_{\cU}^p] \leq \epsilon^{p}
    \]
    and
    \[
        \sup_{z \in \cZ_{ad}} \bE_{m \sim \mu_m} [ \|(I - \projru{r}) D_z u(m,z)\|_{\HS(\cZ, \cU)}^p] \leq \epsilon^{p},
    \]
    whenever $r \geq r_{\epsilon}$.
\end{lemma}
\begin{proof}
    Consider first the output projection error. We define
    \[
        e_{r}^{0}(z) := \bE_{m \sim \mu_m}[ \|(I - \projru{r}) u(m,z) \|_{\cU}^p].
    \]
    We claim that $e_r^{0}(z)$ is continuous over the compact set $\cZ_{ad}$ for every $r \in \bN$ 
    and converges monotonically to zero for every $z \in \cZ_{ad}$, 
    By Dini's theorem, this then implies that the convergence to zero is uniform over $\cZ_{ad}$.

    To see continuity, we consider the mapping $\openset \ni z \mapsto (I-\projru{r}) u(\cdot,z) \in L^p_{\mu_m}(\cM;\cU)$.
    Since $(I - \projru{r})$ is a projection, for every $z_1, z_2 \in \openset$ we have that 
    \begin{equation}\label{eq:pnorm_bound}
        \| (I - \projru{r}) u(\cdot,z_2) - (I - \projru{r}) u(\cdot, z_1)\|_{L^p_{\mu_m}}^p
        \leq \| u(\cdot, z_2) - u(\cdot, z_1) \|_{L^p_{\mu_m}}^p.
    \end{equation}
    By \Cref{proposition:lipschitz_u}, $\openset \ni z \mapsto u(\cdot, z) \in L^p_{\mu_m}(\cM;\cU)$ is continuous,
    and thus when combined with \eqref{eq:pnorm_bound} implies that $\openset \ni z \mapsto (I-\projru{r}) u(\cdot,z) \in L^p_{\mu_m}(\cM;\cU)$ is also continuous.
    
    To see pointwise monotonic convergence, 
    we note that for every $m \in \cM$ and $z \in \cZ_{ad}$, 
    we have that $\|(I - \projru{r}) u(m,z) \|_{\cU}^p \downarrow 0$ as $r \rightarrow \infty$, where $\downarrow$ again denotes monotonic convergence.
    This implies that the sequence of errors are monotonically decreasing, since
    \[
        e_{r_1}^{0}(z) 
        = \bE_{m \sim \mu_m}[\|(I - \projru{r_1})u(m,z)\|_{\cU}^p] 
        \leq 
        \bE_{m \sim \mu_m}[\|(I - \projru{r_2})u(m,z)\|_{\cU}^p] 
        = e_{r_2}^{0}(z)
    \]
    whenever $r_1 \geq r_2$.
    Moreover, since the integrand is bounded by 
    \[ \|(I - \projru{r}) u(m,z)\|_{\cU}^p \leq \|u(m,z)\|_{\cU}^p,\]
    we have the dominating function $\|u(\cdot, z)\|_{\cU}^p$, which is integrable given that $u(\cdot,z) \in L^p_{\mu_m}(\cM;\cU)$.
    Thus, the dominated convergence theorem implies that $e_{r}^{0}(z) \downarrow 0 $ for every $z \in \cZ_{ad}$,

    A similar procedure can be performed for the derivative error. That is, we can define
    \[
        e_{r}^{1}(z) := \bE_{m \sim \mu_m}[ \|(I - \projru{r}) D_z u(m,z) \|_{\HS(\cZ, \cU)}^p],
    \]
    which is again continuous due to \Cref{proposition:lipschitz_u}.
    Moreover, we can rewrite the Hilbert--Schmidt norm as 
    \[
        \| (I - \projru{r}) D_z u(m,z) \|_{\HS(\cZ,\cU)}^p 
        = \| D_z u(m,z)^* (I - \projru{r}) \|_{\HS(\cU,\cZ)}^p 
        = \left( \sum_{i > r}^{\infty} \| D_z u(m,z)^* \phi_i \|_{\cZ}^2 \right)^{p/2},
    \]
    which converges monotonically to zero for every $z \in \cZ_{ad}$ as $r \rightarrow \infty$ and is dominated by the integrable function $\|D_z u(m,z)\|_{\HS(\cZ,\cU)}^p$.
    Once again, the dominated convergence theorem and Dini's theorem together implies that $e_r^{1}(z) \downarrow 0$ uniformly for $z \in \openset$.

    Since the convergence of both $e_{r}^{0}$ and $e_r^{1}$ is monotonic and uniform, it remains to select $r_{\epsilon}$ such that both errors are small, i.e.,
    \[
        \sup_{z \in \cZ_{ad}} e_{r_{\epsilon}}^{0}(z) \leq \epsilon^p, \quad \sup_{z \in \cZ_{ad}} e_{r_{\epsilon}}^{1} \leq \epsilon^p.
    \]
    The error bounds then hold for any $r \geq r_{\epsilon}$.
    
\end{proof}

\begin{lemma}[Partial input dimensionality reduction]\label{lemma:partial_input_truncation_lipschitz}
    Suppose $u$ 
    satisfies \Cref{assumption:pde_properties} with the Borel probability measure $\mu_m$, compact admissible set $\cZ_{ad}$, and open set $\openset \supset \cZ_{ad}$ for some $p \geq 2$.
    Let $\{\psi_i\}_{i=1}^{\infty}$ be an orthonormal basis of $\cM$ and let $\projrm{r} = \sum_{i=1}^{r} \psi_i \psi_i^T$
    be the orthogonal projection onto the reduced basis $\{\psi_i\}_{i=1}^{r}$.
    Then for any $\epsilon > 0$, there exists $r_{\epsilon} \in \bN$ 
    such that simultaneously
    \[
        \sup_{z \in \cZ_{ad}} \bE_{m \sim \mu_m}[\| u(m,z) - u(\projrm{r} m,z)\|_{\cU}^p] \leq \epsilon^p
    \]
    and
    \[
        \sup_{z \in \cZ_{ad}} \bE_{m \sim \mu_m}[\|D_z u(m,z)- D_z u(\projrm{r} m,z)\|_{\HS(\cZ,\cU)}^p] \leq \epsilon^p
    \]
    whenever $r \geq r_{\epsilon}$
\end{lemma}
\begin{proof}
    By the Lipschitz continuity from \Cref{assumption:pde_properties}, we have
    \[
        \bE_{m \sim \mu_m}   [ \|u(m,z) - u(\projrm{r} m, z) \|_{\cU}^p\| \leq \lip_u^p \bE_{m \sim \mu_m} [ \|(I - \projrm{r}) m\|_{\cM}^p ].
    \]
    and 
    \[
        \bE_{m \sim \mu_m}[ \|D_z u(m,z) - D_z u(\projrm{r} m, z)\|_{\HS(\cZ, \cU)}^p] 
        \leq \lip_{u_z}^p 
         \bE_{m \sim \mu_m} [ \|(I  - \projrm{r}) m\|_{\cM}^p ].
    \]
    For the upper bounds, the integrand $\|(I - \projrm{r}) m\|_{\cM}^p \downarrow 0$ as $r \rightarrow \infty$ for every $m \in \cM$. 
    Moreover, $\|(I - \projrm{r})m\|_{\cM}^p \leq \|m\|_{\cM}^p$, where $\|m\|_{\cM}^p$ is assumed to be integrable by our moment assumption on $\mu_m$.
    Thus, the upper bound $\bE_{m \sim \mu_m} [ \|(I - \projrm{r}) m\|_{\cM}^p ] \downarrow 0$ as $r \rightarrow \infty$.
    It suffices to choose a sufficiently large $r_{\epsilon}$ such that 
    \[
       \bE_{m \sim \mu_m} [ \|(I - \projrm{r}) m\|_{\cM}^p ] \leq \epsilon^p \min \left\{\frac{1}{\lip_u^{p}}, \frac{1}{\lip_{u_z}^p} \right\},
    \]
    for which the desired error bounds hold.
\end{proof}

\begin{lemma}[Joint truncation errors]\label{lemma:joint_truncation_errors}
Suppose $u$ 
satisfies \Cref{assumption:pde_properties} with the Borel probability measure $\mu_m$, compact admissible set $\cZ_{ad}$, and open set $\openset \supset \cZ_{ad}$ for some $p \geq 2$.
Let $\{\phi_i\}_{i=1}^{\infty}$ and $\{\psi_i\}_{i=1}^{\infty}$ be orthonormal bases of $\cM$ and $\cU$, respectively. 
Then, for any $\epsilon > 0$, there exist ranks $\rrm$ and $\rru$ such that the truncated operator $u_r(m,z) := \projru{\rru} u(\projrm{\rrm} m, z)$ satisfies
\begin{equation}
    \sup_{z \in \cZ_{ad}} 
    \bE_{m \sim \mu_m} [ \| u(m,z) - u_{r}(m,z) \|_{\cU}^p ]
    + \bE_{m \sim \mu_m} [ \| D_z u(m,z) - D_z u_{r}(m,z) \|_{\HS(\cZ,\cU)}^p ]
    \leq \epsilon^p.
\end{equation}
\end{lemma}
\begin{proof}
    This is a consequence of \Cref{lemma:output_truncation_lipschitz,lemma:partial_input_truncation_lipschitz}.
    By the triangle inequality, we have
    \[
        \| u(m,z) - u_r(m,z)\|_{\cU} \leq 
        \| u(m,z) - \projru{\rru} u(m,z)\|_{\cU} +
        \| \projru{\rru} u(m,z) - \projru{\rru} u(\projrm{\rrm} m, z)\|_{\cU}.
    \]
    Since $\projru{\rru}$ is an orthogonal projection, this bound simplifies to
    \[
        \| u(m,z) - u_r(m,z)\|_{\cU} \leq 
        \| u(m,z) - \projru{\rru} u(m,z)\|_{\cU} +
        \| u(m,z) - u(\projrm{\rrm} m, z)\|_{\cU}.
    \]
    Similarly, for the derivatives, we have 
    \begin{align*}
        & 
        \| D_z u(m,z) - D_z u_r(m,z)\|_{\HS(\cZ,\cU)}  
        \\
        & \qquad 
        \leq 
        \| (I - \projru{\rru}) D_z u(m,z)\|_{\HS(\cZ,\cU)}
        + \| D_z u(m,z) - D_z u(\projrm{\rrm} m, z)\|_{\HS(\cZ,\cU)}.
    \end{align*}
    Using the inequality $(a+b)^{p} \leq 2^{p-1} (a^p + b^p)$,
    we have that for and $z \in \cZ_{ad}$,
    \begin{align}\label{eq:various_truncation_errors}
        & 
        \bE_{m \sim \mu_m} [\| u(m,z) - u_r(m,z)\|_{\cU}^p]
        + \bE_{m \sim \mu_m} [\| D_z u(m,z) - D_z u_r(m,z)\|_{\HS(\cZ,\cU)}^p]  
        \nonumber \\
        & \qquad \leq 
        2^{p-1} 
        \left( 
        \bE_{m \sim \mu_m}[\|(I - \projru{\rru}) u(m,z) \|_{\cU}^p]
        + 
        \bE_{m \sim \mu_m}[\| (I - \projru{\rru}) D_z u(m,z)\|_{\HS(\cZ,\cU)}^p]
        \right)
        \nonumber \\
        & \qquad
        + 2^{p-1} 
        \left( 
        \bE_{m \sim \mu_m}[\|u(m,z) - u(\projrm{\rrm} m,z) \|_{\cU}^p] 
        + 
        \bE_{m \sim \mu_m}[\| D_z u(m,z) - D_z u(\projrm{\rrm} m, z)\|_{\HS(\cZ,\cU)}^p]
        \right).
    \end{align}
    By \Cref{lemma:output_truncation_lipschitz}, we can choose $\rru$ such that 
    \begin{equation}\label{eq:output_rank_choice}
        \sup_{z \in \cZ_{ad}} \bE_{m \sim \mu_m}[\|(I - \projru{\rru}) u(m,z) \|_{\cU}^p] + 
        \bE_{m \sim \mu_m}[\| (I - \projru{\rru}) D_z u(m,z)\|_{\HS(\cZ,\cU)}^p]
        \leq \frac{\epsilon^p}{2^p},
    \end{equation}
    while by \Cref{lemma:partial_input_truncation_lipschitz}, we can choose $\rrm$ such that 
    \begin{equation}\label{eq:input_rank_choice}
    \sup_{z \in \cZ_{ad}} \bE_{m \sim \mu_m}[\|u(m,z) - u(\projrm{\rrm} m,z) \|_{\cU}^p] 
    + \bE_{m \sim \mu_m}[\| D_z u(m,z) - D_z u(\projrm{\rrm} m, z)\|_{\HS(\cZ,\cU)}^p]
    \leq \frac{\epsilon^p}{2^p}.
    \end{equation}
    Taking the supremum in \eqref{eq:various_truncation_errors} and substituting in \eqref{eq:output_rank_choice} and \eqref{eq:input_rank_choice}
    yields 
    \[
        \sup_{z \in \cZ_{ad}} \bE_{m \sim \mu_m} [\| u(m,z) - u_r(m,z)\|_{\cU}^p]
        + \bE_{m \sim \mu_m} [\| D_z u(m,z) - D_z u_r(m,z)\|_{\HS(\cZ,\cU)}^p]  \leq \epsilon^p.
    \]
    
\end{proof}

\subsection{Proof of Theorem \ref{theorem:rbno_ua_semibounded}}
\label{appendix:proof_rbno_ua_semibounded}
We can now combine the bounds for the dimension reduction errors \Cref{lemma:joint_truncation_errors} with the neural network approximation result \Cref{theorem:ua_semibounded} to prove \Cref{theorem:rbno_ua_semibounded}.

\begin{proof}
    (Of \Cref{theorem:rbno_ua_semibounded})
    We begin by assuming $\shiftm = \shiftu = 0$. The case for $\shiftm \neq 0$ and $\shiftu \neq 0$ follow as corollaries.
    Under this assumption, 
    let $u_r(m,z) 
    := \projru{\rru} u (\projrm{\rrm} m, z)
    = \mathbf{\Phi}_{\rru} \mathbf{\Phi}_{\rru}^* u(\mathbf{\Psi}_{\rrm} \mathbf{\Psi}_{\rrm}^*m, z) 
    $,
    We decompose the approximation error into the truncation error $u(m,z) - u_r(m,z)$ and the latent approximation error $u_r(m,z) - u_{\theta}(m,z)$.
    That is, 
    \begin{align}
        \nonumber
        & \bE_{m \sim \mu_m} [ \| u(m,z) - u_{\theta}(m,z) \|_{\cU}^p ]
        + \bE_{m \sim \mu_m} [ \| D_z u(m,z) - D_z u_{\theta}(m,z) \|_{\HS(\cZ,\cU)}^p ]  \\
        \nonumber
        & \quad 
        \leq 2^{p-1} 
        \left( 
        \bE_{m \sim \mu_m} [ \| u(m,z) - u_r(m,z) \|_{\cU}^p ]
        + 
        \bE_{m \sim \mu_m} [ \| D_z u(m,z) - D_z u_{r}(m,z) \|_{\HS(\cZ,\cU)}^p ] 
        \right)
        \\
        & 
        \quad 
        + 2^{p-1} \left( \bE_{m \sim \mu_m} [ \| u_r(m,z) - u_{\theta}(m,z) \|_{\cU}^p ] 
        + 
        \bE_{m \sim \mu_m} [ \| D_z u_r(m,z) - D_z u_{\theta}(m,z) \|_{\HS(\cZ,\cU)}^p ] \right). 
        \label{eq:combined_surrogate_error}
    \end{align}
    By \Cref{lemma:joint_truncation_errors}, there exist $\rru$ and $\rrm$ such that 
    \begin{equation}\label{eq:approx_error_truncation_def}
        \sup_{z \in \cZ_{ad}} 
        \bE_{m \sim \mu_m}[\|u(m,z) - u_r(m,z)\|_{\cU}^p]
        + \bE_{m \sim \mu_m}[\|D_z u(m,z) - D_z u_r(m,z)\|_{\HS(\cZ,\cU)}^p]
        \leq \frac{\epsilon^p}{2^p}.
    \end{equation}

    We now seek an approximation of $u_r$ for the chosen $\rru$ and $\rrm$.
    To this end,we define the mapping $g : \mathbb{R}^{\rrm} \times \openset \rightarrow \mathbb{R}^{\rru}$,
    \begin{equation}
        g(x,z)
        = \mathbf{\Phi}_{\rru}^* u( \mathbf{\Psi}_{\rrm} x, z),
    \end{equation}
    such that  
    \begin{equation}
        u_r(m,z) = \mathbf{\Phi}_{\rru} g(\mathbf{\Psi}_{\rrm}^* m, z).
    \end{equation}
    We can verify that $g$ satisfies the assumptions in \Cref{theorem:ua_semibounded}.
    First, $g$ is continuously differentiable since $u$ is continuously differentiable.
    Moreover, the mappings $\openset \ni z \mapsto g(\cdot, z) \in L^2_{\mur}(\mathbb{R}^{\rrm}; \mathbb{R}^{\rru})$ 
    and $\openset \ni z \mapsto D_z g(\cdot, z) \in L^2_{\mur}(\mathbb{R}^{\rrm}; \mathbb{R}^{\rru \times d_z})$
    are Lipschitz continuous with respect to $z$,
    where $\mur = \left( \mathbf{\Psi}_{\rrm}^{*}\right)^{\sharp} \mu_m$ denotes the pushforward measure of $\mu_m$ under the projection $\mathbf{\Psi}_{\rrm}^{*}$.
    To see this, consider the $z_1, z_2 \in \openset$, and the difference
    \[
        \| g(\cdot, z_2) - g(\cdot, z_1) \|_{L^p_{\mu_m}}^p
        = \int \| g(x, z_2) - g(x, z_1) \|_2^p d \mur (x).
    \]
    Using the definition of $g$ and the 
    change of variables $x = \mathbf{\Psi}_{r_m}^{*} m$ with the pushforward measure $\mur$, we have 
    \[
        \int \| g(x, z_2) - g(x, z_1) \|_2^p d\mur(x)
        = \int \| \mathbf{\Phi}_{\rru}^{*} u (\mathbf{\Psi}_{\rrm} \mathbf{\Psi}_{\rrm}^* m, z_2) - \mathbf{\Phi}_{\rru}^{*} u (\mathbf{\Psi}_{\rrm} \mathbf{\Psi}_{\rrm}^* m, z_2) \|_2^p d\mu_m(m).
    \]
    Since $\mathbf{\Phi}_{\rru}$ is an orthonormal basis, we have $\|\mathbf{\Phi}_{\rru}^*\|_{\Op(\cU, \bR^{\rru})} = 1$, such that
    \[
        \int \| g(x, z_2) - g(x, z_1) \|_2^p d\mur(x)
        \leq \int \| u (\mathbf{\Psi}_{\rrm} \mathbf{\Psi}_{\rrm}^* m, z_2) - u (\mathbf{\Psi}_{\rrm} \mathbf{\Psi}_{\rrm}^* m, z_2) \|_{\cU}^p d\mu_m(m)
        \leq \lip_u^{p} \|z_2 - z_1\|_{\cZ}^p.
    \]
    This implies that $z \mapsto g(\cdot, z)$ is Lipschitz continuous with respect to $z$. 
    An analogous procedure can be used to show that $z \mapsto D_z g(\cdot, z)$ is also Lipschitz continuous with respect to $z$.

    This allows us to apply the result of \Cref{theorem:ua_semibounded}. That is, given $\epsilon > 0$, there exists a neural network approximation $g_{\theta}$ of $g$ satisfying 
    \begin{equation}\label{eq:approx_error_latent_def}
        \sup_{z \in \cZ_{ad}} 
        \bE_{x \sim \mur} [\| g(x,z) - g_{\theta}(x,z) \|^p]
        + \bE_{x \sim \mur} [\| D_z g(x,z) - D_z g_{\theta}(x,z) \|_F^p] \leq \frac{\epsilon^p}{2^p}.
    \end{equation}
    Then the error between $u_r(m,z)$ and the RBNO $u_{\theta}(m,z) = \mathbf{\Phi}_{\rru} g_{\theta}(\mathbf{\Psi}_{\rrm}^{*}, z)$
    can be written directly in terms of $g$ and $g_{\theta}$. 
    That is,
    \begin{align}
        \nonumber
        & \bE_{m \sim \mu_m} \left[ 
            \| u_r(m,z) - u_{\theta}(m,z) \|_{\cU}^p
        \right] \\
        \nonumber 
        &\qquad =
        \bE_{m \sim \mu_m} \left[ 
            \| \mathbf{\Phi}_{\rru} 
            \left( g(\Psi_{\rrm}^* m,z) - g_{\theta}(\Psi_{\rrm}^* m,z) \right)
            \|_{\cU}^p
        \right] \\
        &\qquad = 
        \bE_{x \sim \mur} \left[ 
            \| g(x,z) - g_{\theta}(x,z) \|_{2}^p
        \right], \label{eq:convert_latent_error}
    \end{align}
    and 
    \begin{align}
        \nonumber
        & \bE_{m \sim \mu_m} \left[ 
            \| D_z u_r(m,z) - D_z u_{\theta}(m,z) \|_{\HS(\cZ, \cU)}^p
        \right] \\
        \nonumber
        & \qquad = 
        \bE_{m \sim \mu_m} \left[ 
            \| \mathbf{\Phi}_{\rru} \left( 
                D_z g(\mathbf{\Psi}_{\rrm}^* m,z) - 
                D_z g_{\theta}(\mathbf{\Psi}_{\rrm}^* m,z) 
                \right) \|_{\HS(\cZ, \cU)}^p
        \right] \\
        & \qquad =
        \bE_{x \sim \mur} \left[ 
            \| D_z g(x, z) - D_z g(x, z) \|_{F}^p
        \right]. \label{eq:convert_latent_derivative_error}
    \end{align}
    Combining the bounds 
    \eqref{eq:approx_error_truncation_def}
    and 
    \eqref{eq:approx_error_latent_def},
    along with \eqref{eq:convert_latent_error} and \eqref{eq:convert_latent_derivative_error}
    yields the desired approximation error when substituted into \eqref{eq:combined_surrogate_error}.

    When $\shiftm \neq 0$, 
    we have 
    \[  
        \mathbf{\Psi}^{*}_{\rrm} m = 
        \mathbf{\Psi}^{*}_{\rrm} (m - \shiftm) + \mathbf{\Psi}^*_{\rrm} \shiftm.
    \]
    Thus, for the neural network $g_{\theta}$ constructed using the procedure above, 
    we have 
    \[
        g_{\theta}(\mathbf{\Psi}^{*}_{\rrm} m, z)
        = 
        g_{\theta}(\mathbf{\Psi}^{*}_{\rrm} (m -\shiftm) + \mathbf{\Psi}^*_{\rrm} \shiftm, z)
        = \tilde{g}_{\theta}(\mathbf{\Psi}^{*}_{\rrm} (m - \shiftm), z),
    \]
    where $\tilde{g}_{\theta}$ is also a GELU neural network whose input layer bias is shifted by $\mathbf{\Psi}^*_{\rrm} \shiftm \in \bR^{\rrm}$.
    Thus, the RBNO has the desired representation
    \[
        u_{\theta}(m,z) = \mathbf{\Phi}_{\rrm} \tilde{g}_{\theta}( \mathbf{\Psi}^*_{\rrm} (m - \shiftm), z ).
    \]
    Finally, when $\shiftu \neq 0$, we can construct the approximation $\tilde{u}_{\theta}(m,z)$ for the shifted operator $u(m,z) - \shiftu$ using the same procedures above, noting that 
    $u(m,z) - \shiftu$ also satisfies \Cref{assumption:pde_properties}.
    The resulting approximation of $u(m,z)$ is then given by the affine shift, 
    \[ 
    u_{\theta}(m,z) = \tilde{u}_{\theta}(m,z) + \shiftu = \mathbf{\Phi}_{\rru} g_{\theta}(\mathbf{\Psi}_{\rrm}^* (m - \shiftm), z) + \bar{u}.
    \]
\end{proof}

\subsection{Universal Approximation for Optimization under Uncertainty}
\label{appendix:uaouu}
We can now combine \Cref{theorem:rbno_ua_semibounded} with \Cref{theorem:optimization_error} to prove \Cref{theorem:rbno_ouu_approx},
which guarantees the existence of an RBNO with sufficiently accurate output and derivatives such that when used as a surrogate in expectation minimization (assuming the objective is convex), its stationary points are accurate.

\begin{proof} (Of \Cref{theorem:rbno_ouu_approx})
    First, we note that any GELU neural network $g_{\theta}$ is smooth and its derivatives of all orders are globally bounded by the nature of the GELU activation functions.
    Thus, given any reduced bases $\mathbf{\Phi}_{\rru}$ and $\mathbf{\Psi}_{\rrm}$, the resulting RBNO is also $u_{\theta}$ constructed from $g_{\theta}$ is also smooth with bounded derivatives.
    This implies that $u_{\theta}$ satisfies \Cref{assumption:pde_properties} with the same open set $\openset$.

    Consider now any $0 < \epsilon_{u} < 1$. 
    By the universal approximation result of \Cref{theorem:rbno_ua_semibounded},
    there exists an RBNO such that 
    \begin{equation}
        \sup_{z \in \cZ_{ad}} \|u(\cdot, z) - u_{\theta}(\cdot, z)\|_{L^2_{\mu_m}} \leq \epsilon_u, 
    \end{equation}
    and 
    \begin{equation}
        \sup_{z \in \cZ_{ad}} \|D_z u(\cdot, z) - D_z u_{\theta}(\cdot, z)\|_{L^2_{\mu_m}} \leq \epsilon_u.
    \end{equation}
    Since both $u$ and $u_{\theta}$ satisfy \Cref{assumption:pde_properties} while $Q$ satisfies \Cref{assumption:qoi_properties}, \Cref{prop:surrogate_gradient_error} implies that there exists a constant $\const > 0$ independent of $u_{\theta}$ such that 
    \begin{equation}
        \sup_{z \in \cZ_{ad}} \| D_z \cJ(z) - D_z \cJ_{\theta}(z) \|_{\cZ'} \leq \const \epsilon_u,
    \end{equation}
    where the linear error term dominates for $\epsilon_u < 1$.
    By \Cref{theorem:optimization_error}, we have bounds for both the optimization error and the optimality gap, given by 
    \begin{align}
        \|z^{\star} - z^{\dagger}\|_{\cZ} &\leq \frac{\const}{\lambda} \epsilon_u, \\
        \cJ(z^{\dagger}) - \cJ(z^{\star}) &\leq \frac{\const^2}{2 \lambda} \epsilon_u^2. 
    \end{align}
    From this, we see that it suffices to construct $u_{\theta}$ using $\epsilon_u$ satisfying
    \[
        \epsilon_u \leq \min \left\{ 
            \frac{\lambda \epsilon}{\const}, 
            \frac{\sqrt{2 \lambda \epsilon}}{\const},
            1
        \right\}.
    \]
    
\end{proof}

\end{document}